\documentclass[a4paper]{amsart}
\usepackage{amssymb,amsmath,array,amsthm,amsfonts,amscd,mathrsfs}
\usepackage{fullpage}
\usepackage[latin1]{inputenc}
\usepackage{url}
\usepackage{enumerate}
\usepackage{comment}
\usepackage{color}
\usepackage[all]{xy}    % Pour les diagrammes commutatifs
\usepackage{soul}
\sethlcolor{red}

\allowdisplaybreaks[1]

\numberwithin{equation}{section}

\theoremstyle{plain}
\newtheorem{theorem}[equation]{Theorem}
\newtheorem{proposition}[equation]{Proposition}
\newtheorem{corollary}[equation]{Corollary}
\newtheorem{lemma}[equation]{Lemma}

\theoremstyle{definition}
\newtheorem{definition}[equation]{Definition}
\newtheorem{defn}[equation]{Definition}

\newtheorem{example}[equation]{Example}

\theoremstyle{remark}
\newtheorem{remark}[equation]{Remark}
\newtheorem{rem}[equation]{Remark}

\newcommand{\Q}{\mathbb{Q}}
\newcommand{\Z}{\mathbb{Z}}
\newcommand{\N}{\mathbb{N}}

\newcommand{\kk}{{\Bbbk}}
\newcommand{\al}{\alpha}
\newcommand{\be}{\beta}
\newcommand{\im}{\mbox{\rm im\,}}
\newcommand{\car}{\mbox{\rm char\,}}
\newcommand{\id}{\mbox{\rm id}}

\newcommand{\DMO}{\DeclareMathOperator}

\newcommand{\diag}{\ensuremath{\mathrm{diag\,}}}
\newcommand{\Pder}{\ensuremath{\mathrm{P.der\,}}}
\newcommand{\diver}{\ensuremath{\mathrm{div}}}
\newcommand{\wteps}{\ensuremath{\widetilde{\varepsilon}}}

\DMO{\Spm}{Spm}
\DMO{\F}{Frac\,}
\DMO{\Ext}{Ext}
\DeclareMathOperator{\gr}{gr}
\newcommand{\ol}[1]{\overline{#1}}
\title{Rank two Artin-Schelter regular algebras from non commuting derivations}
\author{Vincent Beck and C\'esar Lecoutre}

\date{\today}

\begin{document}
\begin{abstract}
If $\Delta$ and $\Gamma$ are two derivations of a commutative algebra $A$ such that $\Delta\Gamma-\Gamma\Delta=\Delta$ is locally nilpotent, one can endow $A$ with a new product $\ast$ whose filtered semiclassical limit is the Poisson structure $\Delta\wedge\Gamma$.
In this article we first study theses (Poisson) algebras from an algebraic point of view, and when $A$ is a polynomial algebra, we investigate their homological properties.
In particular, when the derivations $\Delta$ and $\Gamma$ are linear, the algebras $(A,\ast)$ provide, in each dimension at least four, new examples of multiparameter families of Artin-Schelter regular algebras. These algebras are deformations of Poisson algebras $(A,\Delta\wedge\Gamma)$ of rank $2$, thus explaining the title of the article.

Assuming furthermore a technical condition on $\Gamma$, we show that the algebra $(A,\ast)$ is Calabi-Yau if and only if the trace of $\Gamma$ is equal to $1$ if and only if the Poisson algebra $(A,\Delta\wedge\Gamma)$ is unimodular.
Since the trace of $\Gamma$ is a linear function of the parameters, the algebras $(A,\ast)$ also provide, in each dimension at least four, new examples of multiparameter families of Calabi-Yau algebras.
\end{abstract}

\maketitle

\tableofcontents

\section*{Introduction}

The philosophy behind noncommutative projective algebraic geometry \cite{Rog} is to study noncommutative version of the projective spaces $\mathbb{P}^n$ by describing suitable noncommutative algebras that would be thought of as their homogeneous coordinates rings: the so-called Artin-Schlelter regular algebras (AS regular algebras for shorts) \cite{AS}.
Numerous examples of AS regular algebras are known but a complete classification is achieved only up to projective dimension $2$.
The work of Pym \cite{Pym} provides a partial answer in dimension $3$ by focusing on a classification of graded Calabi-Yau algebras \cite{Ginz} that are flat deformations of $\mathbb{C}[x_0,x_1,x_2,x_3]$.
In~\cite{LS}, Sierra and the second author introduced a family of AS regular algebras $R(n,a)$ of Gel'fand-Kirillov dimension $n+1$ indexed by a scalar $a$.
The authors proved that for any integer $n\geqslant0$ there exists a unique scalar $a$ such that the algebra $R(n,a)$ is a Calabi-Yau.
This is of particular interest since it explains a specific choice of parameter made in the exceptional component $E(3)$ of Pym's classification \cite{Pym}.

Back to the algebra $R(n,a)$, the parameter $n+1$ can be interpreted as the index of nilpotency of a maximal Jordan block seen as the restriction to degree $1$ polynomials of a locally nilpotent derivation $\Delta$ %, the so-called basic Weitzenb\"ock derivation,
of a polynomial algebra in $n+1$ variables.
The construction of $R(n,a)$ then relies on a second derivation $\Gamma$ such that $\Delta\Gamma-\Gamma\Delta=\Delta$ is nilpotent together with a deformation formula of Coll, Gerstenhaber, and Giaquinto~\cite{CGG}.
The aim of this article is to investigate the following natural generalization: what are the possible algebras arising from this construction by relaxing the conditions on $\Delta$?
The case of homogeneous locally nilpotent derivations $\Delta$ of a polynomial rings whose canonical Jordan normal form acting on the set of degree 1 polynomials is not a maximal block will be of particular interest.
In this situation we obtain a wide family of non isomorphic AS regular algebras in each dimension since the Jordan type of $\Delta$ must be preserved by isomorphism (Theorem \ref{thm-isom}).
%parametrized at least by the Jordan type of $\Delta$ and parameters depending on $\Gamma$.}
Among them, we describe a family of Calabi-Yau algebras indexed by $r-1$ scalar parameters, where $r$ is the number of blocks in the canonical Jordan normal form of $\Delta$ (the algebras studied in~\cite{LS} correspond to the case $r=1$).
%The loss of a parameter with respect to the number of block is due to the fact that among these algebras the Calabi-Yau ones are obtained for a certain value of the trace of $\Gamma$ corresponding to the case where the associated Poisson algebra is unimodular.

This article is divided into eight sections.
In Section~\ref{sec-present} we present the general construction of a Poisson algebra $A_{(\Delta,\Gamma)}$ and of an associative algebra $R_{(\Delta,\Gamma)}$ associated to a pair of derivations $(\Delta,\Gamma)$ of a commutative algebra $A$ such that $\Delta\Gamma-\Gamma\Delta=\Delta$ is nilpotent.
Such a pair of derivation will be called a \textit{solvable pair}.
Section~\ref{sec-solvable-pair} is devoted to the study of general algebraic properties of $A_{(\Delta,\Gamma)}$ and $R_{(\Delta,\Gamma)}$.
A lots of example of solvable pairs are given.
Moreover we compute the (Poisson) center in many cases (Lemma~\ref{lem-Pcenter}) and exhibit examples of important Poisson derivations of $A_{(\Delta,\Gamma)}$ and automorphisms of $R_{(\Delta,\Gamma)}$ (see~Lemma~\ref{lem-pder} and Theorem~\ref{deltaR}).
Furthermore we remark that the natural filtration associated to the locally nilpotent derivation $\Delta$ is compatible with both the Poisson structure of $A_{(\Delta,\Gamma)}$ and the associative structure of $R_{(\Delta,\Gamma)}$.
This filtration is presented and studied in Section~\ref{ssec-filtration}, it will be one of our main tool to prove a lot of results in this article.
In the general situation not much can be said about normal elements.
However, among them, we identify and study a significant subset consisting of the so-called strongly normal elements~(Definition~\ref{dfn-strongly-ne}) which will be of particular interest latter on in this article.
To end Section~\ref{sec-solvable-pair} we observe that up to localization the Poisson algebra $A_{(\Delta,\Gamma)}$ and the algebra $R_{(\Delta,\Gamma)}$ are particularly nice: they are isomorphic to (Poisson) Ore extensions over (a localization of) the kernel of $\Delta$ thanks to the local slice construction.

The rest of the article is devoted to the case where the solvable pair is defined on a polynomial algebra~$A$.
In Section~\ref{sec-pol-ring} the rank of a non abelian Poisson algebra $A_{(\Delta,\Gamma)}$ is shown to be equal to $2$.
Section~\ref{sec-linear} deals with the case of homogeneous derivations $\Delta$ and $\Gamma$.
In this situation both the Poisson center of $A_{(\Delta,\Gamma)}$ and the center of $R_{(\Delta,\Gamma)}$ are completely determined~(Corollary~\ref{cor-poissoncom}).
Moreover we construct a finer filtration than the previous one, having the benefit that the associated graded algebra is a polynomial algebra.
In Section~\ref{sec-trigo} we show that the algebra $R_{(\Delta,\Gamma)}$ is AS regular and we completely determine the (Poisson) normal elements: they are precisely the strongly normal elements introduced in Section~\ref{sec-solvable-pair}.
Moreover we prove that the Jordan type of $\Delta$ is an invariant of both $A_{(\Delta,\Gamma)}$ and $R_{(\Delta,\Gamma)}$.
In Sections~\ref{sec-diago} and~\ref{sec-diago-R} we focus on the case where $\Gamma$ is diagonalizable.
In this situation, we are able to determine the linear Poisson derivations when $\Gamma$ is generic.
We also provides a presentation of $R_{(\Delta,\Gamma)}$ by generators and relations.
Our last section is devoted to characterize when our algebras (resp. Poisson algebras) are Calabi-Yau (resp. unimodular),  see Corollary~\ref{cor-calabiyau}.
%{\color{red} a modifier selon ce que l'on garde en appendix. 
Finally the article ends with Appendices which recap some of the results frequently used in the article, in particular it explains the structure of the graded ring for the finer of the filtrations we study.
%}
In this article we always assume that $\kk$ is a field of characteristic zero.

\section{A deformation formula for non commuting derivations}\label{sec-present}

Let $A$ be a commutative $\kk$-algebra and $\Delta,\Gamma$ be two derivations of $A$ such that $[\Delta,\Gamma]=\Delta\Gamma-\Gamma\Delta=\Delta$.
It is easily verified that the biderivation $\{-,-\}=\Delta\wedge\Gamma$ satisfies the Jacobi identity, hence is a Poisson bracket on $A$.
More precisely we have
\begin{equation}
\label{pbdg}
\{f,g\}=\Delta(f)\Gamma(g)-\Delta(g)\Gamma(f)
\end{equation}
for any $f,g\in A$.
Moreover, if $\Delta$ is locally nilpotent, by setting for any $f,g\in A$
\begin{equation}
\label{ast}
f\ast g=\sum_{i\geqslant0}\Delta^i(f)\binom{\Gamma}{i}(g)=fg+\Delta(f)\Gamma(g)+\sum_{i\geqslant 2}\Delta^i(f)\binom{\Gamma}{i}(g),
\end{equation}
where $\dbinom{\Gamma}{i}=\dfrac{1}{i!}\Gamma\circ(\Gamma-\id)\circ\cdots\circ(\Gamma-(i-1)\id)$, we define an associative product on $A$, see \cite{CGG} for the formal version and \cite[Lemma 3.3]{Pym} for the algebraic version.
We set 
\begin{itemize}
\item[$\bullet$] $A=A_{(\Delta,\Gamma)}=(A,\{-,-\})$,
\item[$\bullet$] $R=R_{(\Delta,\Gamma)}=(A,\ast)$.
\end{itemize}
The noncommutative algebra $R$ is a deformation of the Poisson algebra $A$ in the following sense.
Consider $A$ as an algebra over $\kk[t]$ and define
\[f\ast_t g=\sum_{i\geqslant0}t^i\Delta^i(f)\binom{\Gamma}{i}(g).\]
Then $A_t=(A,\ast_t)$ is an associative and noncommutative algebra over $\kk[t]$ such that
\begin{itemize}
\item $A_t/(t-1)A_t\cong R$,
\item $A_t/tA_t$ is commutative and can be endowed with a Poisson structure given by
\[\{f+tA_t,g+tA_t\}=\frac{f\ast_t g-g\ast_t f}{t}+tA_t\]
\item $A_t/tA_t$ and $A$ are isomorphic as Poisson algebras.
\end{itemize}
Therefore the Poisson algebra $A_{(\Delta,\Gamma)}$ is the commutative fibre version of the semiclassical limit of $R_{(\Delta,\Gamma)}$ in the sense of \cite[Section 2.1]{Goo}.

\begin{example}
\label{LSnot}
For $A=\kk[X_0,\dots,X_n]$ and the derivations $\Delta(X_i)=X_{i-1}$ ($X_{-1}=0$) and $\Gamma(X_i)=(a+i)X_i$ ($a\in\kk$), 
the (Poisson) algebras obtained by the formulae \eqref{pbdg} and \eqref{ast} have been studied in~\cite{LS}.
The Poisson algebra is denoted by $A(n,a)$ and 
the noncommutative algebra $R(n,a)$.
% In particular the algebras $R(n,a)$ are new examples of Artin-Schelter regular algebras for $n\geq3$.
%\vspace{1em}
In this article, our aim is to generalize results of~\cite{LS} to other pairs of derivations. 
\end{example}

\section{Solvable pairs}\label{sec-solvable-pair}

\subsection{Generalities}

A pair $(\Delta,\Gamma)$ of derivations of $A$ such that $[\Delta,\Gamma]=\Delta$ is locally nilpotent 
will be called a {\em solvable pair}.
Since $[\Delta,\Gamma]=\Delta$ we observe thanks to equations (\ref{pbdg}) and (\ref{ast}) that
\[\Gamma=0\implies \Delta=0\implies \left\{\begin{array}{ll}
A_{(\Delta,\Gamma)}\textrm{ is Poisson commutative}, \\
R_{(\Delta,\Gamma)}\textrm{ is commutative}.
\end{array}\right.\]
%Therefore we will most of the time assume that $\Delta\neq0$ in the following.
The aim of this section is to study general facts about solvable pairs. 
The following result study the influence of a change of generators of $A$ on the algebras $A_{(\Delta,\Gamma)}$  and $R_{(\Delta,\Gamma)}$.

\begin{lemma}
\label{iso} Let $(\Delta,\Gamma)$ be a solvable pair on $A$ and $\sigma$ and automorphism of the algebra $A$. Then 
\begin{enumerate}[{\rm (i)}]
\item $(\widetilde{\Delta},\widetilde{\Gamma}):=(\sigma\Delta\sigma^{-1},\sigma\Gamma\sigma^{-1})$ is a %(linear if $\sigma$ is linear) 
solvable pair on $A$, 
\item the Poisson algebras $A_{(\Delta,\Gamma)}$ and $A_{(\widetilde{\Delta},\widetilde{\Gamma})}$ are isomorphic,
\item the algebras $R_{({\Delta},{\Gamma})}$ and $R_{(\widetilde{\Delta},\widetilde{\Gamma})}$ are isomorphic.
\end{enumerate}
\end{lemma}

\begin{proof}
(i) is a straightforward computation.

(ii) and (iii). In both case we show that $\sigma$ is an isomorphism between the appropriate algebras. 
Since $\sigma\Delta=\widetilde{\Delta}\sigma$ and $\sigma\Gamma=\widetilde{\Gamma}\sigma$ we have for all $f,g\in A$
\begin{align*}
\sigma(\{f,g\}_{(\Delta,\Gamma)})&=\sigma\Delta(f)\sigma\Gamma(g)-\sigma\Delta(g)\sigma\Gamma(f)\\
&=\widetilde{\Delta}\sigma(f)\widetilde{\Gamma}\sigma(g)-\widetilde{\Delta}\sigma(g)\widetilde{\Gamma}\sigma(f)=\{\sigma(f),\sigma(g)\}_{(\widetilde{\Delta},\widetilde{\Gamma})}.
\end{align*}
By induction we obtain easily that $\sigma\Delta^\ell=\widetilde{\Delta}^\ell\sigma$ and $\sigma\binom{\Gamma}{\ell}=\binom{\widetilde{\Gamma}}{\ell}\sigma$ for all $\ell\in\N$. 
The equality $\sigma(f\ast_{(\Delta,\Gamma)} g)=\sigma(f)\ast_{(\widetilde{\Delta},\widetilde{\Gamma})}\sigma(g)$ follows from the definition of the product $\ast$ given by (\ref{ast}).
\end{proof}

\begin{remark}
When $\Delta=0$ the Poisson algebra $A$ is Poisson commutative and the algebra $R$ is commutative for every choice of derivation $\Gamma$.
Therefore there is no hope in general that a Poisson isomorphism between $A_{(\Delta,\Gamma)}$ and 
$A_{(\Delta',\Gamma')}$ implies that there exists $\sigma\in\mbox{Aut}(A)$ such that $(\Delta',\Gamma') = (\sigma\Delta\sigma^{-1},\sigma\Gamma\sigma^{-1})$ (the same remark also applies to $R_{(\Delta,\Gamma)}$).
In other words, the Poisson isomorphism class of $A_{(\Delta,\Gamma)}$ or the $\kk$-algebra isomorphism class of 
$R_{(\Delta,\Gamma)}$ does not determine the isomorphism class of the representation of the 2-dimensional solvable Lie algebra.

For a less trivial example showing that even the conjugacy class of $\Delta$ is not determined, 
see~Remark~\ref{rem-commutative-example}.
\end{remark}

We conclude this section with some examples. 
If $A=\kk[X_0,\dots,X_n]$ is a polynomial algebra, then we denote by $\partial_{X_i}$ the usual derivative with respect to the indeterminate $X_i$.
Recall that the set of derivations of $A$ is a free $A$-module with basis $(\partial_{X_0},\dots,\partial_{X_n})$.

\begin{proposition}
\label{classificationdim2}
Let $A=\kk[X_0,X_1]$.
Up to isomorphism, any solvable pair $(\Delta,\Gamma)$ on $A$ with $\Delta\neq0$ is of the form
\[\Delta=P\partial_{X_1},\quad \Gamma=Q\partial_{X_0}+R\partial_{X_1}\]
where $P,Q\in\kk[X_0]$, $P\neq0$ and $R=R_0 + R_1 X_1$ for some $R_0,R_1\in \kk[X_0]$ are such that
\begin{enumerate}
\item[(1)] if $P\in\kk^{\ast}$ or $Q=0$, then $R_1=1$
\item[(2)] if $\deg_{X_0}(P)>0$ and $Q\neq0$, then $R_1\neq1$ and we have $P(R_1-1)=Q\partial_{X_0}(P)$.
\end{enumerate}
In particular if $\deg_{X_0}(P)>0$, $Q\neq0$ and $R_1\in\kk\setminus\{1\}$, then there exists $\Lambda\in\kk^\ast$, $\alpha\in\kk$ and an integer $n>0$ such that $Q=\frac{R_1-1}{n}(X_0-\alpha)$ and $P=\Lambda(X_0-\alpha)^n$.
\end{proposition}

\begin{proof}
Let $(\delta,\gamma)$ be a solvable pair on $A$.
Thanks to Rentschler's theorem, see \cite[Theorem 4.1]{Fre}, there exists $\alpha\in\mbox{Aut}(A)$ such that
$\Delta:=\alpha\gamma\alpha^{-1}=P\partial_{X_1}$ for some nonzero polynomial $P\in\kk[X_0]$.
Set $\Gamma:=\alpha\delta\alpha^{-1}=Q\partial_{X_0}+R\partial_{X_1}$ for some $Q,R\in A$.
Then $[\Delta,\Gamma]=\Delta$ implies that
\[0=\Delta(X_0)=\Delta\Gamma(X_0)-\Gamma\Delta(X_0)=\Delta(Q)\]
so that $Q\in\ker\Delta=\kk[X_0]$ since $P\neq0$.
Moreover we have
	\[P=\Delta(X_1)=\Delta\Gamma(X_1)-\Gamma\Delta(X_1)=\Delta(R)-\Gamma(P)=P\partial_{X_1}(R)-Q\partial_{X_0}(P)\]
so that $P(\partial_{X_1}(R)-1)=Q\partial_{X_0}(P)$.
Since $P\neq0$ two cases can happened
\begin{enumerate}
\item $(Q=0$ or $P\in\kk^{\ast})\iff \partial_{X_1}(R)=1 \iff R=X_1+R_0$ with $R_0\in\kk[X_0]$
\item $(Q\neq0$ and $\deg_{X_0}P>0) \iff \partial_{X_1}(R)\neq1$.
\end{enumerate}
In the second case, every factors in $P(\partial_{X_1}(R)-1)=Q\partial_{X_0}(P)$ is nonzero so that the degree in $X_1$ of $\partial_{X_1}(R)-1$ is equal to $0$.
Hence $R=R_0+R_1X_1$ with $R_0,R_1\in\kk[X_0]$ and $R_1\neq1$.
When $R_1=\mu \in \kk\setminus\{1\}$ we have $(\mu-1)P=(-Q)P'$ which classically implies that $\deg_{X_0}(Q)=1$ and thus $Q=\lambda(X_0-\alpha)$ for some $\lambda\in\kk^\ast$ and some 
$\alpha\in\kk$ and $P=\Lambda(X_0-\alpha)^n$ for some $\Lambda\in\kk^\ast$ and some integer $n$ and the relation $(\mu-1)P=(-Q)P'$ gives $\lambda=(\mu-1)/n$.
\end{proof}

\begin{corollary}
\label{S}
Every Poisson structure on $A=\kk[X_0,X_1]$ arising from a solvable pair is, up to isomorphism, of the form $\{X_0,X_1\}=S$, where $S\in\kk[X_0]$.
The algebra $(R,\ast)$ is given by two generators $X_0,X_1$ and one relation $X_0\ast X_1-X_1\ast X_0=S$.
In particular, the Poisson field $\F A$ is either Poisson commutative or a Poisson Weyl field. 
Similarly, the skewfield $\F R$ is either commutative or is a Weyl skewfield.
\end{corollary}

\begin{proof}
We have the result for $A$ with $S=-PQ$ thanks to Proposition \ref{classificationdim2} and by setting, if $S\neq0$, $X_1'=X_1S^{-1}$ so that $\{X_0,X_1'\}=1$.
We denote by $X_1^{\ast i}$ the $i^\textrm{th}$ power of $X_1$ with respect to the product $\ast$.
Observe that the powers $X_1^{\ast i}$ for $i\geqslant 0$ form a basis of $R=(A,\ast)$ over $\kk[X_0]$ so that $R$ is isomorphic to the Ore extension $\kk[X_0][X_1;S\partial_{X_0}]$ (see Proposition \ref{Oreext} for a proof in a more general situation).
In particular $R$ is a Noetherian domain, hence admits a skewfield of fraction $\F R$.
If $S\neq0$, set $X_1'=X_1\ast S^{-1}$ so that $X_0 \ast X_1' - X_1' \ast X_0 = 1$ and $R$ is isomorphic to a Weyl skewfield.
\end{proof}

\begin{example}\label{exdim2produit}
The Poisson structure on $A=\kk[X_0,X_1]$ given by $\{X_0,X_1\}=X_0X_1$ cannot be obtained by a solvable pair.
It follows from Corollary \ref{S} since $\F A$ (resp. $\F R$) cannot contains any Poisson bracket (resp. commutator) equal to $1$, see \cite{GL} (resp. \cite{AD}).
%(Generalization : orbit under automorphism group of $\{X_0,X_1\}=S(X_0)$ ??) 
\end{example}

\begin{example}
\label{exdim2}
Let $A=\kk[X_0,X_1]$ and set $(\Delta,\Gamma)=(\partial_{X_0},{X_0}\partial_{X_0}+X_1^{i}\partial_{X_1})$ for some integer $i\in\N$.
One easily verifies that $(\Delta,\Gamma)$ is a solvable pair on $A$ and that
\[\{X_0,X_1\}=X_1^i \qquad\mbox{and}\qquad X_0\ast X_1- X_1\ast X_0=X_1^i = X_1^{\ast i}\]
where $X_1^{\ast i}$ denote the $i^\textrm{th}$ power of $X_1$ with respect to the product $\ast$. 
We retrieve the following classical (Poisson) algebras:
\begin{itemize}
\item[$i=0$] the (Poisson) Weyl algebra,
\item[$i=1$] the (symmetric) enveloping algebra of the two dimensional non abelian Lie algebra,
\item[$i=2$] the (Poisson) Jordan plane.
\end{itemize}
Note that the (Poisson) Weyl algebra is classically obtained via the commuting pair of derivations $(\partial_{X_0},\partial_{X_1})$.
\end{example}

\begin{example}
\label{envelopingalg}
Let $A=\kk[X,Y_1,\dots,Y_n]$ and set $\Delta=\partial_{X}$ and $\Gamma=X\partial_{X}+\lambda_1Y_1\partial_{Y_1}+\cdots+\lambda_nY_n\partial_{Y_n}$ for some scalars $\lambda_1,\dots,\lambda_n$.
Then $(\Delta,\Gamma)$ is a solvable pair on $A$ and $A_{(\Delta,\Gamma)}$ (resp. $R_{(\Delta,\Gamma)}$) is isomorphic to the symmetric (resp. enveloping) algebra of the solvable Lie algebra $\mathfrak{g}$ with basis $\{x,y_1,\dots,y_n\}$ and nonzero Lie bracket 
\[[x,y_i]=\lambda_i y_i,\quad i=1,\dots,n.\] 
\end{example}

\subsection{Center} In this section, we study the elementary and general properties of solvable pairs related to the center of $A_{(\Delta,\Gamma)}$ and $R_{(\Delta,\Gamma)}$.

\begin{definition}
Let $A$ be a Poisson algebra.
The Poisson center of $A$ is the set 
\[Z_P(A)=\{z\in A\ |\ \{z,a\}=0\mbox{ for all }a\in A\}\]
of Poisson central elements.
It is a Poisson-commutative subalgebra of $A$.
\end{definition}

\begin{lemma}\label{lem-Pcenter} Assume that $\Delta\neq0$ and that $A=A_{(\Delta,\Gamma)}$ is a domain.
Set $R=R_{(\Delta,\Gamma)}$.
\begin{enumerate}
\item[(1)] Then $\ker\Delta\cap\ker\Gamma\subseteq Z_P(A)$ with equality if $\ker\Delta\neq\ker\Gamma$.
\item[(2)] We have $\ker\Delta\cap\ker\Gamma\subseteq Z(R)$. Moreover, 
$Z(R) \cap \ker \Delta \subseteq \ker \Gamma$. 
\item[(3)] Assume that $\ker \Delta \neq\ker \Gamma$. Then $Z(R) =\ker \Delta \cap \ker \Gamma$.
% Ancienne version améliorée par la précédente
%\item[(3)] Assume that there exists $x \in \ker \Delta$ such that $x$ is an eigenvector for $\Gamma$ for a nonzero eigenvalue
%$\lambda$. Then $Z(R) =\ker \Delta \cap \ker \Gamma$. 
\end{enumerate}
\end{lemma}

\begin{proof}
(1) The inclusion $\ker\Delta\cap\ker\Gamma\subseteq Z_P(A)$ is clear.
Let $z\in Z_P(A)$.
First assume that there exists $x\in\ker\Delta\setminus\ker\Gamma$.
Then $0=\{z,x\}=\Delta(z)\Gamma(x)$, which implies $z\in\ker\Delta$.
For $y\notin\ker\Delta$ we have $0=\{y,z\}=\Delta(y)\Gamma(z)$ and so $z\in\ker\Gamma$. 
Now assume that there exists $x\in\ker\Gamma\setminus\ker\Delta$.
Then $0=\{x,z\}=\Delta(x)\Gamma(z)$ so that $z\in\ker\Gamma$.
Since $\ker\Gamma\neq A$ (otherwise $\Gamma=0$ which implies $\Delta=0$) we have for any $y\notin\ker\Gamma$ that $0=\{z,y\}=\Delta(z)\Gamma(y)$ and so $z\in\ker\Delta$.

(2) The inclusion $\ker\Delta\cap\ker\Gamma\subseteq Z(R)$ is clear.
Let us now consider $z \in Z(R) \cap \ker \Delta$. 
Since $\Delta \neq 0$ and $\Delta$ is locally nilpotent, there exists $x \in \ker \Delta^2 \setminus \ker \Delta$. 
We then have $z * x= zx$ since $z\in\ker \Delta$ and $x*z= zx+ \Delta(x)\Gamma(z)$ since $x \in \ker \Delta^2$. 
But $z \in Z(R)$ hence $\Delta(x)\Gamma(z)=0$ and $\Delta(x) \neq 0$. Thus $\Gamma(z)=0$. 

(3) We first show that $\ker \Delta \not \subseteq \ker \Gamma$.
If $\ker \Delta \subseteq \ker \Gamma$ then $\ker \Delta \varsubsetneq \ker \Gamma$ since $\ker \Delta \neq \ker \Gamma$.
Consider $x \in \ker \Gamma \setminus \ker \Delta$ and $n \geqslant 0$ such that $\Delta^n(x) \neq 0$ and $\Delta^{n+1}(x)=0$.
In particular $\Delta^n(x) \in \ker \Delta \subseteq \ker \Gamma$. 
Thus $[\Delta^n,\Gamma](x)= n\Delta^n(x)$ (see Lemma~\ref{lem-crochet}) implies $\Delta^n(\Gamma(x)) = n\Delta^n(x)$.
But $x \in \ker\Gamma$ thus $n\Delta^n(x)=0$ thus $n=0$ and $\Delta(x)=0$ which is absurd.
Hence we can consider $x \in \ker \Delta$ such that $\Gamma(x)\neq 0$ and choose %\textcolor{red}{a nonzero element} -> inutile de le supposer ton nul 
$z \in Z(R)$.
Then $x*z=xz$ since $x \in \ker \Delta$.
Moreover
\begin{equation}
\label{comuzx}
0= z*x- x*z = \sum_{i\geqslant 1}\Delta^i(z) \binom{\Gamma}{i}(x)
\end{equation}
But $[\Delta,\Gamma]= \Delta$ implies that $\ker\Delta$ is stable by $\Gamma$, hence stable by $\binom{\Gamma}{i}$ for all $i\geqslant 1$ (see Lemma \ref{comDG}).
Therefore we have $\binom{\Gamma}{i}(x) \in \ker \Delta$ for all $i\geqslant 1$. 
Let $N \geqslant 1$ be the smallest positive integer such that $\Delta^N(z)=0$.
Assume that $N\geqslant 2$.
By applying $\Delta^{N-2}$ to equation \eqref{comuzx} we obtain that $\Gamma(x) \Delta^{N-1}(z)=0$. 
Since $\Gamma(x)\neq 0$, we get a contradiction.
Hence $N=1$ i.e. $z \in \ker \Delta$ and assertion $(2)$ provides us with $z \in \ker \Gamma$.
% Ancien point numéro 3
%For the third point, let $z \in Z(R)$. Then $x*z=xz$ (since $x \in \ker \Delta)$ and 
%$$0= z*x- x*z = \sum_{i\geq 1}\Delta^i(z) \binom{\Gamma}{i}(x)=x \sum_{i\geq 1}\binom{\lambda}{i} \Delta^i(z) $$
%We then deduce that $\sum_{i\geq 1}\binom{\lambda}{i} \Delta^i(z)=0$. Let $N \geq 1$ the smallest positive integer $i$ such that
%$\Delta^i(z)=0$. Applying $\Delta^{N-2}$ to the preceding relation we get that $\lambda \Delta^{N-1}(z)=0$. 
%Since $\lambda\neq 0$, we get a contradiction. Hence $z \in \ker \Delta$. 
\end{proof}

\begin{example}
Recall the derivations $\Delta=\partial_{X_0}$ and $\Gamma=X_0\partial_{X_0}+X_1^i\partial_{X_1}$ from Example \ref{exdim2}.
One easily verifies that $\ker\Delta=\kk[X_1]$ and $\ker\Gamma=\kk$ so that $Z_P(A)=Z(R)=\kk$. %
%Ancienne version de la rédaction, avant l'amélioration du (3).
%For the case $i=1$, take $x=X_1$ in the proposition. For the others values of $i$, the proof of the third  point of Lemma~\ref{lem-Pcenter} can be adapted by considering $x \in \kk[X_1] \setminus \kk$ which verifies $x \in \ker \Delta$ and $\Gamma(x)\neq 0$. Since $\binom{\Gamma}{i}(x)\in \kk[X_1]$ and $\Delta$ is $k[X_1]$-linear, applying $\Delta^{N-2}$ to the relation $z*x-x*z$ gives $\Delta^{N-1}(z)\Gamma(x)=0$.
\end{example}

\begin{example}
Recall the derivations $\Delta=\partial_{X}$ and $\Gamma=X\partial_{X}+\lambda_1 Y_1\partial_{Y_1}+\dots+\lambda_n  Y_n\partial_{Y_n}$ from Example~\ref{envelopingalg}.
Then $\ker\Delta=\kk[Y_1,\dots,Y_n]\neq\ker\Gamma$ (as long as not all the $\lambda_i$ are zeros) so that $Z_P(A)=Z(R)=\ker\Delta\cap\ker\Gamma$. %(take $x=Y_i$ for $\lambda_i\neq0$ in the proposition).
Note that to understand
\[\ker\Delta\cap\ker\Gamma=\ker\Gamma|_{\kk[Y_1,\dots,Y_n]}=\mbox{Vect}\{Y_1^{\alpha_1}\cdots Y_n^{\alpha_n}\ |\ \lambda_1\alpha_1+\cdots+\lambda_n\alpha_n=0,\ \alpha_i\in\N\}\]
one needs to study the structure of the submonoid $\sum_{i=1}^n \lambda_i\N$ of $(\kk,+)$. %(or the $\Z$-module $\sum_{i=1}^n \lambda_i\Z$ if we allow for localization of the $Y_i$'s).
For instance, if the $\lambda_i$ are $\N$-linearly independent, then $Z_P(A)=Z(R)=\kk$. 
\end{example}

The following example illustrate the fact that when $\ker\Delta=\ker\Gamma$ the inclusion $\ker\Delta\cap\ker\Gamma\subseteq Z_P(A)$ can be strict. 

\begin{example}
\label{rem-commutative-example}
Let $A=\kk[X_0,\dots,X_n]$ and consider the solvable pair $(\Delta,\Gamma)=(X_0\partial_{X_1},X_1\partial_{X_1})$.
We have $A=Z_P(A)$ but $\ker\Delta=\ker\Gamma=\kk[X_0,X_2,X_3,\dots,X_n]\subsetneq A$ (see also~Corollary~\ref{cor-poissoncom}).
Similarly we have $Z(R)=R$ so $Z(R)\not\subseteq\ker\Delta=\ker\Gamma$.
\end{example}

\subsection{Derivation and automorphism}

\begin{definition} Let $A$ be a Poisson algebra.
\begin{enumerate}[{\rm (1)}]
\item A {\em Poisson derivation} is a derivation $\delta$ of $A$ such that for all $a,b\in A$
\[\delta(\{a,b\})=\{\delta(a),b\}+\{a,\delta(b)\}.\]
\item The set ${\rm P.Der}(A)$ of Poisson derivations of $A$ is a Lie subalgebra of the Lie algebra of derivations of 
the associative commutative algebra $A$.
%\item A {\em Poisson automorphism} is an automorphism $\Phi$ of $A$ such that for all $a,b\in A$
%\[\Phi(\{a,b\})=\{\Phi(a),\Phi(b)\}.\]
%\item The set ${\rm P.Aut}(A)$ of Poisson automorphisms of $A$ is a subgroup of the group of automorphisms of the associative commutative algebra $A$.
\end{enumerate}
\end{definition}

Let $A=A_{(\Delta,\Gamma)}$.
When $\Delta \neq 0$, the Lie algebra ${\rm P.Der}(A)$ is not reduced to $0$ thanks to the following 
easy lemma. 

\begin{lemma}\label{lem-pder}
The derivation $\Delta$ is a Poisson derivation of $A$.%
%\textcolor{red}{Suite \`a supprimer ? Réponse de V le 15 janvier : oui, d'accord, je mets en commentaire. 
%the derivation $\Gamma$ is a Poisson $(-\Delta)$-derivation, that is a derivation of $A$ such that for all $f,g\in A$
%\[\Gamma(\{f,g\})=\{\Gamma(f),g\}+\{f,\Gamma(g)\}-\Delta(f)\Gamma(g)+\Delta(g)\Gamma(f).\]
%The set of Poisson $(-\Delta)$-derivations is a $\kk$-vector space but not a Lie sub-algebra of ${\rm Der}(A)$ in general.}
\end{lemma}

\begin{remark}\label{rem-der-center} Poisson center and Poisson derivation are compatible in the following sense. 
Let $A$ be a Poisson algebra. The Poisson center of $A$ is stable by any Poisson derivation of $A$: 
for $\delta\in{\rm P.Der}(A)$ and $a \in Z_P(A)$, an easy computation shows that $\delta(a) \in Z_P(A)$.
\end{remark}

Note that the map $\Delta$ is not a derivation of $R_{(\Delta,\Gamma)}$ in general.
The following proposition can be seen as a deformation formula for the Poisson derivation $\Delta$.

\begin{theorem} \label{deltaR}
Let $(\Delta,\Gamma)$ be a solvable pair.
For $a \in \kk$, the map $\phi_a :=(\id+ \Delta)^a=\sum_{k \geqslant 0} \binom{a}{k} \Delta^k$ is a (graded if $\Delta$ is graded) $\kk$-algebra automorphism of $R=R_{(\Delta,\Gamma)}$.
The linear map
\[\delta=\sum_{r\geqslant 1}\frac{(-1)^{r-1}}{r}\Delta^r\]
is a (graded if $\Delta$ is graded) derivation of $R$.
\end{theorem}

Strictly speaking $\delta$ is an element of the formal power series ring $\kk[\![\Delta]\!]$ but since $\Delta$ is locally nilpotent, $\delta(f)$ is a well-defined element of $R$ for any $f\in R$.

\begin{proof} We first prove that $\id + \Delta$ is an automorphism of $R$. 

\noindent $\begin{array}{r@{\,}c@{\,}l@{\,}l@{}}(\id + \Delta)(f)*(\id + \Delta)(g)&=& \sum_{i\geqslant 0} (\Delta^{i}(f) + \Delta^{i+1}(f))
(\binom{\Gamma}{i}(g) + \binom{\Gamma}{i}\Delta(g))& \\
&=& f*g+\sum_{i\geqslant 0}\Delta^{i+1}(f)\binom{\Gamma}{i}(g)+ \Delta^{i}(f)\binom{\Gamma}{i}\Delta(g) +
	\Delta^{i+1}(f)\binom{\Gamma}{i}\Delta(g)&\!(\ref{comDG})\\ 
&=& f*g+\sum_{i\geqslant 0}\Delta^{i+1}(f)\binom{\Gamma}{i}(g) + f*\Delta(g)+ 
\sum_{i\geqslant 0}\Delta^{i+1}(f)\left[\Delta,\binom{\Gamma}{i+1}\right](g) & \\
&=& f*g+\sum_{i\geqslant 0}\Delta^{i+1}(f)\binom{\Gamma}{i}(g) + f*\Delta(g)+ 
\sum_{i\geqslant 0}\Delta^{i}(f)\left[\Delta,\binom{\Gamma}{i}\right](g)&\\
&=&f*g+\sum_{i\geqslant 0}\Delta^{i+1}(f)\binom{\Gamma}{i}(g) +
 \sum_{i\geqslant 0}\Delta^{i}(f)\Delta\binom{\Gamma}{i}(g) &\\
&=& f*g + \Delta(f*g)
\end{array}$

\noindent Proposition~2 of~\cite{Vid} allows us to conclude that $\delta$ is a derivation of $R$. Hence $a \delta$ is a locally nilpotent derivation of $R$ for any $a\in\kk$.
So $\exp(a\delta)=(\id + \Delta)^a$ is an automorphism of $R$. 

This last property can be obtained by a direct computation using the combinatorial relation
% Commentaire sur la preuve : pour le calcul de (Id+Delta)^a(f*g) : on appelle k l'indice de developpement de (Id+Delta)^a, \ell l'indice de développement de f * g, i l'indice de développement 
% de la formule de Leibniz généralisée et j l'indice qui permet échange de Delta et Gamma (relation de l'appendice A). Alors en posant k-i=v, j =w et l+i=w+u, on obtient
% (Id+Delta)^a(f)*(Id+Delta)^a(g) grâce à l'égalité ci-dessous où u désigne l'indice de développement (Id + Delta)(f), v désigne l'indice de développement de (Id+ Delta)(g) et w l'indice de 
% developpement de \Delta^u(f)*\Delta^v(g). 

\noindent
\begin{center}$\begin{array}{r@{\,}c@{\,}l}\sum_{i=0}^{u} \binom{a}{v+i}\binom{v+i}{i}\binom{v}{u-i}= \sum_{i=0}^{u} \frac{a(a-1)\cdots (a-v-i+1)}{i!(v-u+i)!(u-i)!} &=&
\binom{a}{v}\sum_{i=0}^{u} \frac{(a-v)\cdots (a-v-i+1) v!}{i!(v-u+i)!(u-i)!}\\[1ex] &=& \binom{a}{v}\sum_{i=0}^{u} \binom{a-v}{i}\binom{v}{u-i}=\binom{a}{v}\binom{a}{u}\end{array}$
\end{center}
where the last equality is the Chu-Vandermonde identity (see~Appendix~\ref{combinatorial}).
\end{proof}

\begin{lemma} \label{compo}
For any $a,b\in\kk$ we have $\phi_a\circ\phi_b=\phi_{a+b}$.
Moreover if $\Delta \neq 0$, then $\phi_a = \id$ if and only if $a=0$. 
When $a \neq 0$ we have $\ker (\phi_a - \id)^n = \ker \Delta^n$  for every $n \in \N$.
\end{lemma}

\begin{proof}
Let $f\in R$.
By using equation (\ref{vandermonde}) we have
\begin{align*}
\phi_a\circ\phi_b(f)=\sum_{k,\ell\geqslant0}\binom{a}{k}\binom{b}{\ell}\Delta^{k+\ell}(f)=\sum_{u\geqslant0}\left(\sum_{k+\ell=u}\binom{a}{k}\binom{b}{\ell}\right)\Delta^u(f)=\sum_{u\geqslant0}\binom{a+b}{u}\Delta^u(f)=\phi_{a+b}(f).
\end{align*}
When $\Delta \neq 0$, there exists $f \in R$ such that $\Delta^2(f)=0$ and $\Delta(f)\neq 0$.
Then $\phi_a(f)= f + a \Delta(f)$, hence the second assertion is proved.
%Hence $\phi_a$ is bijective with inverse $\phi_{-a}$ since $\phi_0=\mbox{id}$.

To prove the last assertion, observe that since $\Delta$ is a locally nilpotent derivation, the map $\kk[\![T]\!] \rightarrow {\rm End}_\kk(A)$ sending $T$ to $\Delta$ is 	a well-defined $\kk$-algebra homomorphism.
In particular, every element of the form $\sum_{k \in \N} a_k \Delta^k$ where $a_k \in \kk$ for $k \in \N$ is invertible if $a_0 \neq 0$.
Since $(\phi_a-\id)^n = 
(a^n + \sum_{k \geqslant 1} b_k \Delta^k)\Delta^n$ for some $b_k \in \kk$ for all $k \geqslant 1$, we obtain the result.
\end{proof}

\subsection{A filtration on $R_{(\Delta,\Gamma)}$} \label{ssec-filtration}

It is a classical fact that locally nilpotent derivations induce filtrations on commutative algebras.
In this section we show that the sequence $(\ker\Delta^{i+1})_{i\in\N}$ provides us with a filtration of both the Poisson algebra $A_{(\Delta,\Gamma)}$ and the noncommutative algebra $R_{(\Delta,\Gamma)}$.

\begin{definition} Let $\Delta$ be a locally nilpotent derivation on $A$.
Set $\varepsilon(0)=-\infty$ and for any nonzero element $f \in A$ define $\varepsilon(f)=\min \{i \in \N,\ \Delta^{i+1}(f)=0\}$ and set $A^{\leqslant i}=\{f \in A,\ \varepsilon(f) \leqslant i\} = \ker\Delta^{i+1}$.
By convention set $A^{\leqslant-1}=\{0\}$.
\end{definition}

Note that $\varepsilon=\deg_\Delta$ in the notation of \cite[Section 1.1.8]{Fre} which is a degree function thanks to \cite[Proposition 6.1.1]{Now} since $\Delta$ is locally nilpotent.

\begin{lemma}\label{lem-crochet} For any polynomial $P\in\kk[T]$ we have $[P(\Delta),\Gamma]=Q(\Delta)$ where $Q=XP'$.
In particular, for any integer $i\geqslant0$ we have $\Delta^{i}\Gamma=\Gamma\Delta^{i}+i\Delta^{i}$.
\end{lemma}

\begin{proof} By linearity, it is enough to prove the relation for $P=X^i$, where $i\in \N$, which is easily obtained by induction.
\end{proof}

\begin{lemma}
\label{filtdg}
Let $(\Delta,\Gamma)$ be a solvable pair on $A$ and fix integers $i,j\geqslant0$.
\begin{enumerate}
 \item[(1)] We have $\Delta(A^{\leqslant i})\subseteq A^{\leqslant i-1}$ and more precisely $\varepsilon(\Delta(f))=\varepsilon(f)-1$ for every $f \in A$.
 \item[(2)] We have $\Gamma(A^{\leqslant i})\subseteq A^{\leqslant i}$. Hence $\varepsilon(P(\Gamma)(f)) \leqslant \varepsilon(f)$ for every $f \in A$ and $P \in \kk[T]$. 
 % Remarque V (15/01) on a bien une inégalité et pas une égalité... Gamma peut faire descendre le degré (il suffit de penser à un élément non nul de Ker Gamma.
 \item[(3)] If $f,g\in A$ are such that $\varepsilon(f)=i$ and $\varepsilon(g)=j$, then $\varepsilon(fg)=i+j$.
 \item[(4)] For any $f,g \in A$ we have
 \[f\ast g-g\ast f = \Delta(f)\Gamma(g)-\Delta(g)\Gamma(f)+\sum_{i\geqslant 2}\left(\Delta^i(f)\binom{\Gamma}{i}(g)-\Delta^i(g)\binom{\Gamma}{i}(f)\right)\]
 In particular, if $\varepsilon(f)=i$ and $\varepsilon(g)=j$ then $\varepsilon( f*g - g*f ) \leqslant i + j - 1$.
\end{enumerate}
\end{lemma}

\begin{proof}
Assertion (1) is obvious since $\Delta(f) \in \ker \Delta^{i-1}$ if and only if $f \in \ker \Delta^i$. Assertion~(2) follows from~Lemma~\ref{lem-crochet} showing that $\ker \Delta^i$ is stable
by $\Gamma$ and hence by every polynomial in $\Gamma$. 
Assertion (3) is true since from Leibniz formula we get $\Delta^{i+j}(fg)=\binom{i+j}{i}\Delta^i(f)\Delta^j(g)\neq0$ because $A$ is a domain and $\car\kk=0$.
Assertion (4) follows from \eqref{ast} and then assertions (1) and (2).
\end{proof}

Recall that if $(T^{\leqslant i})_{i\in\N}$ is a filtration of a ring $T$, its associated graded ring $\gr(T)=\bigoplus T^{\leqslant i}/T^{\leqslant i-1}$ is an $\N$-graded ring whose homogeneous elements of degree $i$ are denoted by $x+T^{\leqslant i-1}$ for an element $x\in T^{\leqslant i}\setminus T^{\leqslant i-1}$.

\begin{proposition}\label{prop-linkRA} Assume that $A$ is a domain.
\begin{enumerate}
\item[(1)] The family $(A^{\leqslant i})_{i\geqslant0}$ is a Poisson algebra filtration of $A_{(\Delta,\Gamma)}$ of degree $-1$, meaning that it is an algebra filtration of $A$ together with $\{A^{\leqslant i},A^{\leqslant j}\}\subseteq A^{\leqslant i+j-1}$ for all integers $i,j\geqslant0$.
Moreover, the associated graded algebra $\gr(A)$ is a domain. 
\item[(2)] The family $(A^{\leqslant i})_{i\geqslant0}$ is an algebra filtration of $R_{(\Delta,\Gamma)}$.
Moreover, the associated graded algebra $\gr(R)$ is equal to $\gr(A)$.
\item[(3)] The commutative algebra $\gr(R)=\gr(A)$ can be endowed with the following three Poisson brackets
\begin{enumerate}
 \item[(a)] $\{\overline{f},\overline{g}\}':= (f*g-g*f) + A^{\leqslant i+j-2}$
 \item[(b)] $\{\overline{f},\overline{g}\}'':= \{f,g\}+ A^{\leqslant i+j-2}$
 \item[(c)] $\{\overline{f},\overline{g}\}''':= \overline{\Delta}(\overline{f})\overline{\Gamma}(\overline{g}) - \overline{\Gamma}(\overline{f})\overline{\Delta}(\overline{g}) \in A^{\leqslant i+j-1}/A^{\leqslant i+j-2}$
\end{enumerate}
for homogeneous elements $\overline{f}$ and $\overline{g}$ of respective $\varepsilon$-degree $i$ and $j$.
In (c) the pair of maps $(\overline{\Delta},\overline{\Gamma})$ is the solvable pair of homogeneous derivations of $\gr(A)$ of respective degree $-1$ and $0$ which is induced by the solvable pair of filtered derivations $\Delta$ and $\Gamma$ of $A$.
%\[
%\{\overline{f},\overline{g}\}':= \overline{f*g-g*f}= \overline{\Delta(f)\Gamma(g) - \Gamma(f)\Delta(g)}=\overline{\{f,g\}}\in A^{\leqslant i+j-1}/A^{\leqslant i+j-2}
%\]
%for homogeneous elements $\overline{f}\in A^{\leqslant i}/A^{\leqslant i-1}$ and $\overline{g}\in A^{\leqslant j}/A^{\leqslant j-1}$.
%This Poisson bracket is graded of degree $-1$. 
%For $x \in R^{\leqslant i+j-1}$, we denote by $\overline{x}$ the image of $x$ 
%in $R^{\leqslant i+j-1}/R^{\leqslant i+j-2}$. The map
%$$(\textrm{gr}(f),\textrm{gr}(g)) \in {\textrm{gr}}(A) \times \textrm{gr}(A) \longmapsto 
%\{\textrm{gr}(f),\textrm{gr}(g)\}:= \overline{f*g-g*f}= \overline{\Delta(f)\Gamma(g) - \Gamma(f)\Delta(g)} \in \textrm{gr}(A)$$
%\noindent is then well defined and defines a Poisson structure on $\textrm{gr}(A)$.
\item[(4)] The Poisson structures defined in (3) are all equal and make $\gr(A)=\gr(R)$ into a graded Poisson algebra of degree $-1$.
%The maps $\overline{\Delta}$ and $\overline{\Gamma}$ given by $\overline{\Delta}(\overline{f})=\overline{\Delta(f)}$ and $\overline{\Gamma}(\overline{f})=\overline{\Gamma(f)}$ on homogeneous elements $\overline{f}$ extends uniquely by linearity to well-defined graded derivations of $\gr(A)$ of respective degree $-1$ and $0$.
%Moreover they satisfy $[\overline{\Delta},\overline{\Gamma}]=\overline{\Delta}$. 
%\item[(7)] The Poisson structure obtained in (3) on $\gr(R)$ is equal to the Poisson structure given by the solvable pair $(\overline{\Delta},\overline{\Gamma})$ on $\gr(A)$, that is :
%\[\{\overline{f},\overline{g}\}=\overline{\Delta}(\overline{f})\overline{\Gamma}(\overline{g}) - \overline{\Gamma}(\overline{f})\overline{\Delta}(\overline{g})\]
%for homogeneous elements $\overline{f}\in A^{\leqslant i}/A^{\leqslant i-1}$ and $\overline{g}\in A^{\leqslant j}/A^{\leqslant j-1}$.
\end{enumerate}
\end{proposition}

\begin{proof} 
(1). The fact that $(A^{\leqslant i})_{i\geqslant0}$ is an algebra filtration of $A$ is an easy consequence of Leibniz formula, see \cite[Section 1.1.8]{Fre} for instance. 
It is a Poisson filtration of degree $-1$ thanks to equation \eqref{pbdg} and assertions (1) and (2) of Lemma \ref{filtdg}.
Moreover $\gr(A)$ is a domain thanks to assertion (3) of Lemma \ref{filtdg}.

(2). The fact that $(A^{\leqslant i})_{i\geqslant0}$ is an algebra filtration of $R_{(\Delta,\Gamma)}$ follows from equation \eqref{ast} since the degree $\varepsilon$ is decreased by $\Delta$ and preserved by $\Gamma$ thanks to Lemma \ref{filtdg}.
More precisely, if $f,g\in A$ are such that $\varepsilon(f)=i$ and $\varepsilon(g)=j$, then $f\ast g =fg + u$ for some $u\in A^{\leqslant i+j-1}$. 
This also implies that the associated graded algebras $\gr(A)$ and $\gr(R)$ have the same multiplication, so that they are indeed equal.

(3). The fact that $\{-,-\}'$ is a Poisson bracket on $\gr(R)$ follows the filtered version of the semi-classical limit construction, see \cite[Section 2.4]{Goo}.
The fact that $\{-,-\}''$ is a well-defined biderivation satisfying the Jacobi identity on $\gr(A)$ follows by tedious but straightforward computation from the fact that $\{-,-\}$ is a filtered Poisson bracket on $A$.
Finally $\{-,-\}'''$ is a Poisson bracket since $(\overline{\Delta},\overline{\Gamma})$ is a solvable pair of derivations of $\gr(A)$.

(4). The Poisson brackets $\{-,-\}'$ and $\{-,-\}''$ are the same since both $f\ast g-g\ast f$ and $\{f,g\}$ belong to $A^{\leqslant i+j-1}$ combined with $f\ast g-g\ast f-\{f,g\}\in A^{\leqslant i+j-2}$ thanks to assertion 4 of Lemma \ref{filtdg}.
Finally, for homogeneous elements $\overline{f}$ and $\overline{g}$ of $\varepsilon$-degree $i$ and $j$ we have 
\begin{align*}
 \{\overline{f},\overline{g}\}'''
 & = \overline{\Delta}(\overline{f})\overline{\Gamma}(\overline{g}) - \overline{\Gamma}(\overline{f})\overline{\Delta}(\overline{g}) \\
 & = (\Delta(f)+A^{\leqslant i-2})(\Gamma(g)+A^{\leqslant i-1}) -  (\Delta(g)+A^{\leqslant i-2})(\Gamma(f)+A^{\leqslant i-1}) \\
 & = (\Delta(f)\Gamma(g)+A^{\leqslant i+j-2}) -  (\Delta(g)\Gamma(f)+A^{\leqslant i+j-2}) \\
 & = (\Delta(f)\Gamma(g) - \Delta(g)\Gamma(f))+A^{\leqslant i+j-2}  = \{f,g\} + A^{\leqslant i+j-2} = \{\overline{f},\overline{g}\}''
 \end{align*}
Hence $\{-,-\}''$ and $\{-,-\}'''$ are the same.
\end{proof}

In fact it is enough to assume that $\ker \Delta$ is a domain to obtain that $R$ is a domain.
More precisely we have the following corollary.

\begin{corollary}\label{cor-domain}
The following assertions are equivalent
\begin{enumerate}
\item $\ker \Delta$ is a domain,
\item $A$ is a domain,
\item $\gr(A) = \gr(R)$ is a domain,
\item $R$ is a domain.
\end{enumerate}
\end{corollary}

\begin{proof}
$(2) \implies (3)$ is part of assertion $(1)$ of Proposition~\ref{prop-linkRA}.

$(3) \implies (4)$ is classical. %Corollary~\ref{cor-domain}.

$(4) \implies (1)$ since $\ker \Delta$ is a subring of $R$.

$(1) \implies (2)$. 
With the notation of the proof of assertion $(3)$ of Lemma~\ref{filtdg}, we have $\Delta^{i}(f) \in \ker \Delta$ and $\Delta^j(g) \in \ker \Delta$. So 
$\binom{i+j}{i}\Delta^i(f)\Delta^j(g) \neq 0$ implies that $\Delta^{i+j}(fg) \neq 0$, hence $fg \neq 0$.
\end{proof}

\begin{example}
Consider the solvable pair on $A=\kk[X_0,X_1]$ given by $\Delta=X_0\partial_{X_1}$ and $\Gamma = a X_0 \partial_{X_0} + (a+1) X_1 \partial_{X_1}$ for $a\in\kk$.
We have $A^{\leqslant i}=\ker\Delta^{i+1}=\bigoplus_{\ell\leqslant i}\kk[X_0]X_1^\ell$.
Then $\gr(A)=\bigoplus_{i\geqslant0}A^{\leqslant i}/A^{\leqslant i-1}\cong\bigoplus_{i\geqslant0}\kk[X_0]X_1^i\cong A$.
Observe that $X_0\in A^{\leqslant 0}$ and $X_1\in A^{\leqslant 1}$ and
\[X_0\ast X_1-X_1\ast X_0 = - a X_0\ast X_0 \in A^{\leqslant 0}\]
so that the induced Poisson bracket on $\gr(A)$ is given by
\[\{X_0+A^{\leqslant -1},X_1+A^{\leqslant 0}\}= - a X_0\ast X_0 + A^{\leqslant -1}.\]
\end{example}

In the previous example $\gr(A)$ is isomorphic to $A$ as a Poisson algebra.
This is not always the case as the following example demonstrates.

\begin{example}
Consider the solvable pair on $A=\kk[X_0,X_1,X_2]$ given by $\Delta=X_0\partial_{X_1}$ and $\Gamma=aX_0\partial_{X_0}+(a+1)X_1\partial_{X_1}+(aX_2+X_0)\partial_{X_2}$.
Since $\ker\Delta=\kk[X_0,X_2]$ it is clear that $\gr(A)\cong A$.
%Hence $A^{\leqslant i}=\ker\Delta^{i+1}=\bigoplus_{\ell\leqslant i}\kk[X_0,X_2]X_1^\ell$ and
%then $\gr(A)=\bigoplus_{i\geqslant0}A^{\leqslant i}/A^{\leqslant i-1}\cong\bigoplus_{i\geqslant0}\kk[X_0,X_2]X_1^i\cong A$.
The induced solvable pair on $\gr(A)$ is given by $\overline{\Delta}=X_0\partial_{X_1}$ and $\overline{\Gamma}=aX_0\partial_{X_0}+(a+1)X_1\partial_{X_1}+aX_2\partial_{X_2}$.
Despite being isomorphic as associative algebras, $\gr(A)$ and $A$ are not isomorphic as Poisson algebras (the set of linear Poisson derivations provides an invariant, details are given in Section \ref{nondiaggam}).
\end{example}

The following example focuses on the associative structure of $\gr(A)$.
To lighten notation, we denote by $\gr(f)=f+A^{\leqslant i-1}$ the image in $\gr(A)$ of any element $f\in A$ with $\epsilon(f)=i\in\N$.

\begin{example}
Consider the derivation $\Delta=X_0\partial_{X_1}+X_1\partial_{X_2}$ of $A=\kk[X_0,X_1,X_2]$.
It is classical that $A^{\leqslant 0}=\ker\Delta=\kk[X_0,F_2]$ where $F_2=2X_0X_2-X_1^2$, see for instance \cite[Example 6.7.1.]{Now}. 
Recall that $\varepsilon(fg)=\varepsilon(f)+\varepsilon(g)$ and that $\varepsilon(X_k)=k$ for $k\in\{0,1,2\}$.
Therefore, for all $i\geqslant 1$, we have $\varepsilon(X_1{X_2}^{i-1})= 2i-1$ and $\varepsilon({X_2}^{i})=2i$.
Hence the family $\mathcal{B}=(1,X_1{X_2}^{i-1},{X_2}^{2i})_{i\geqslant 1}$ is free over $A^{\leqslant 0}$.
Moreover observe that the Hilbert series of $A^{\leqslant 0} \bigoplus_{i\geqslant 1} (A^{\leqslant 0} X_1 {X_2}^{i-1} \oplus  A^{\leqslant 0} {X_2}^{i})$ is the same that the one of $\kk[X_0, X_1, X_2]$.
Hence $\mathcal{B}$ is a basis of $\kk[X_0,X_1,X_2]$ over $A^{\leqslant 0}$.
This implies that $\gr(A) = A^{\leqslant 0} \oplus \bigoplus_{i\geqslant1} A^{\leqslant 0}\ \gr(X_1{X_2}^{i-1}) \oplus A^{\leqslant 0}\ \gr({X_2}^{i})$. 
Observe that $\mbox{gr}(X_2) \in \gr(A)$ generates a polynomial ring over $A^{\leqslant 0}$ and that we have the relation $X_1^2= 2X_0X_2 - F_2$.
Therefore $\gr(A) \cong k[X_0,F_2,X_2][T]/(T^2 - (2X_0 X_2 - F_2)) \cong k[X_0,X_2,T] \cong A$ since $T^2 - (2X_0 X_2 - F_2)$ is a degree one polynomial in $F_2$.
\end{example}

Although this is true in the three preceding examples, it is unclear to us whether or not $\gr(A)\cong A$ in general.
We will define in Section \ref{finerfiltation} a finer filtration for which it is always true.

\subsection{Normal elements}

The filtration introduced in the previous section allow us to study Poisson automorphisms and Poisson derivations of $\gr(R) = \gr(A)$ induced by (strongly) normal elements of $R$.
This will be useful in Section~\ref{sec-CYunimod}.
%We also introduce the notion of strongly normal element (see Definition~\ref{dfn-strongly-ne}) which are specific normal elements of $R$ and which will be very useful in Sections~\ref{sec-trigo}, \ref{sec-diago-R} and~\ref{unimod}.
Recall that if $\sigma$ and $\tau$ are endomorphisms of a ring $T$ a $(\sigma,\tau)$-derivation of $T$ is an endomorphism $\delta$ of $T$ such that $\delta(rs)=\sigma(r)\delta(s)+\delta(r)\tau(s)$ for all $r,s\in T$.

\begin{lemma}\label{lem-autom-deriv}
Let $(\Gamma,\Delta)$ be a solvable pair on $A$ and consider an automorphism $\Phi$ of $R=(A,\ast)$ that is compatible with the filtration $\varepsilon$, i.e. such that $\varepsilon(\Phi(f))=\varepsilon(f)$ for all $f \in R$.
Then $\Phi$ induces a Poisson automorphism $\ol{\Phi}$ of $\gr(R)=\gr(A)$ given by $\ol{\Phi}(f+A^{\leqslant i})=\Phi(f)+A^{\leqslant i}$ for all $f\in R$ with $\varepsilon(f)=i$.
Moreover if $\ol{\Phi}=\id$, then $\Phi-\id$ induces a Poisson derivation of $\gr(R)=\gr(A)$ of degree $-1$.
\end{lemma}

\begin{proof}
Since $\varepsilon(\Phi(f))=\varepsilon(f)$ for all $f \in R$ it is clear that $\ol\Phi$ is a well-defined automorphism of $\gr(R)$.
Consider $f,g \in R$ such that $\varepsilon(f)=i$ and $\varepsilon(g)=j$ and denote by $\ol{f},\ol{g}$ their images in $\gr(R)$.
We have
$$\ol{\Phi}(\{\ol{f},\ol{g}\})= \ol{\Phi}(f \ast g - g \ast f + A^{\leqslant i+j -2}) =\Phi(f) \ast \Phi(g) - \Phi(g) \ast \Phi(f) + A^{\leqslant i+j -2}=\{\ol{\Phi}(\ol{f}),\ol{\Phi}(\ol{g})\}$$
hence $\ol\Phi$ is a Poisson automorphism of $\gr(A)$. 

The equality $\ol{\Phi}=\id$ means that if $\varepsilon(f)=i$ then $\varepsilon(\Phi(f)-f) \leqslant i-1$.
Hence $\Phi-\id$ induces a well-defined vector space endomorphism $\ol{\Phi-\id}$ of $\gr(R)$ of degree $-1$ by setting $\ol{(\Phi-\id)}(\ol{f})=\Phi(f)-f + A^{\leqslant i-2}$ for all $\ol{f}\in \gr(R)$ where $f\in R$ is such that $\varepsilon(f)=i$.
By definition we have $\ol{(\Phi-\id)}(\{\ol{f},\ol{g}\})= (\Phi-\id)(f\ast g-g\ast f) + A^{\leqslant i+j-3}$.
Since $\Phi-\id$ is both a $(\Phi,\id)$-derivation and a $(\id,\Phi)$-derivation of $R$ we obtain 
$$(\Phi-\id)(f\ast g-g\ast f)=\big(\Phi(f)\ast (\Phi(g)-g) - (\Phi(g)-g)\ast \Phi(f)\big) + \big((\Phi(f)-f)\ast g - g \ast (\Phi(f)-f)\big)\,.$$
Hence $\ol{(\Phi-\id)}(\{\ol{f},\ol{g}\}) = \{\ol{\Phi}(\ol{f}), \ol{(\Phi-\id)}(\ol{g})\} + \{\ol{(\Phi-\id)}(\ol{f}), \ol{g}\}$ and we get the result using the fact that $\ol{\Phi}=\id$.
\end{proof}

%\begin{lemma}\label{lem-phi-id} Let $\Phi$ be an automorphism of a $\kk$-algebra $T$. Then $\Phi-\id$ is a $\Phi$-derivation of $T$ and also a $(\id,\Phi)$-derivation of $T$.
%\end{lemma}
%
%\begin{proof} For $u,v \in T$, we have $\Phi(uv)-uv= \Phi(u)(\Phi(v)-v) + (\Phi(u)-u)v$ and $\Phi(uv)-uv= u(\Phi(v)-v) + (\Phi(u)-u)\Phi(v)$
%\end{proof}

\begin{definition} Let $T$ be a $\kk$-algebra that is a domain and let $N\in T$ be a normal element.
For every $u \in T$ there exists a unique element $\Phi(u) \in T$ such that $uN= N\Phi(u)$.
Then $\Phi$ is a $\kk$-algebra automorphism of $T$ and is called the automorphism associated to $N$ and denoted by $\Phi_N$.
\end{definition}
%\begin{proof} The existence of $\Phi(u)$ follows from the definition of being normal. The unicity from the fact that $T$ is a domain. 
%The bilinearity of the product in $T$ shows that $\Phi$ is a $\kk$-linear map. The associativity of the product in $T$ shows that $\Phi(uv)= \Phi(u)\Phi(v)$ for $u,v \in T$. For
%$$N\Phi(uv)=(uv)N=u(vN)=u N \Phi(v))=N\Phi(u)\Phi(v)\,.$$
%\end{proof}

\begin{lemma}\label{lem-autom-normalR}
Let $(\Gamma,\Delta)$ be a solvable pair on a domain $A$ and consider $R=(A,\ast)$.
Let $N$ be a normal element in $R$ and consider $\Phi_N$ its associated automorphism.
We denote by $\ol{N}$ the image of $N$ in $\gr(R)$. 
Then $\Phi_N$ is compatible with $\varepsilon$.
Moreover $\ol{\Phi_N} = \id$ and $\ol{(\Phi_N-\id)}$ verifies that $\ol{N}\ \ol{(\Phi_N-\id)} = \{\,\cdot\,,\ol{N}\}$.
In particular, $\ol{N}$ is Poisson normal in $\gr(R)$.
\end{lemma}

\begin{proof}
Let $f \in R$ and set $i=\varepsilon(f)$ and $j=\varepsilon(N)$.
Since $f\ast N = N\ast\Phi_N(f)$ and $A$ is a domain, we have $i+j=\varepsilon(\Phi_N(f))+j$.
Therefore $\Phi_N$ is compatible with $\varepsilon$. 

The definition of the product $\ast$ implies that $f \ast N = fN + u$ where $u \in A^{\leqslant i+j-1}$.
Similarly $N\ast\phi_N(f) = N\Phi_N(f) + v$ where $v \in A^{\leqslant i+j-1}$.
Hence $N(\Phi_N(f)-f) \in A^{\leqslant i+j-1}$.
Since $A$ is a domain, we obtain that $\Phi_N(f)-f \in A^{\leqslant i-1}$ i.e. $\ol{\Phi_N}=\id$.
Moreover, we have the following equalities 
\[0=N \ast \Phi_N(f) - f \ast N= N(\Phi_N(f)-f) + \Delta(N)\Gamma(\Phi_N(f))-\Delta(f)\Gamma(N)+ u\]
for some $u \in A^{\leqslant i+j-2}$. 
Hence $N (\Phi_N(f)-f)= \Delta(f)\Gamma(N)-\Delta(N)\Gamma(\Phi_N(f)) - u$ with $\Delta(f)\Gamma(\Phi_N(f))-\Delta(N)\Gamma(f) \in A^{\leqslant i+j-1}$ since $\Phi_N$ is compatible with $\varepsilon$.
This implies the desired equality since $\ol{\Phi_N}=\id$.
\end{proof}

\begin{definition}[Strongly normal element]
\label{dfn-strongly-ne}
Let $(\Delta,\Gamma)$ be a solvable pair on $A$.
An element $f \in A$ is said to be strongly normal if $\Delta(f)=0$ and if there exists $\alpha\in \kk$ such that $\Gamma(f)=\alpha f$. 
\end{definition}

\begin{lemma}
\label{lem-strongly} 
Let $(\Delta,\Gamma)$ be a solvable pair on $A$.
If $f$ is strongly normal in $A$ then $f$ is normal in $R$ as well as Poisson normal in $A$.
Set $\Gamma(f)=\alpha f$ where $\alpha\in\kk$.
If $A$ is a domain, then the automorphism of $R$ associated to $f$ is $\phi_\alpha$ (see Theorem \ref{deltaR}), the Poisson derivation associated to $f$ is $\alpha \Delta$ and we have $\ol{(\phi_\alpha - \id)}=\alpha \ol{\Delta}$ as derivations of $\gr(A)$.
\end{lemma}

\begin{proof}
Since $f \in \ker \Delta$, we have $f \ast g=fg$ for all $g\in R$, hence $f \ast R = f A$.
Also we have
$$g\ast f=\sum_{k\geqslant0} \Delta^k(g)\binom{\Gamma}{k}(f)= \sum_{k\geqslant0} \Delta^k(g)\binom{\alpha}{k} f = f\sum_{k\geqslant0} \binom{\alpha}{k}\Delta^k(g)=f \phi_\al(g) \,.$$
%=f \ast \phi_\al(g)\,.$$
Recall from Theorem~\ref{deltaR} that $\phi_\alpha$ is an automorphism of $R$.
Hence the equality $g\ast f=f\phi_\alpha(g)$ implies that $ R\ast f= Af$.
Therefore we have $f \ast R = fA =R \ast f$ and $f$ is normal in $R$.
Moreover, for every $g \in A$, we have $\{g,f\}=\Delta(g)\Gamma(f)- \Gamma(g)\Delta(f)= \alpha f \Delta(g) \in f A$ and we conclude that $f$ is Poisson normal and that the Poisson derivation associated to $f$ is $\alpha \Delta$.

Assume that $A$ is a domain.
Since $f \in \ker \Delta$ and $\Gamma(f)=\alpha f$, we have $\ol{\Delta}(f)=0$ and $\ol{\Gamma}(\ol{f})=\alpha\ol{f}$.
Hence $\{\ol{g},\ol{f}\} = \alpha \ol{\Delta}(\ol{g})\ol{f}$.
Thanks to Lemma~\ref{lem-autom-normalR} we have $\{\ol{g},\ol{f}\}=\ol{(\phi_\alpha-\id)}(\ol{g})\ol{f}$.
Since $\gr(A)$ is a domain we obtain the relation $\ol{(\phi_\alpha-\id)}=\alpha \ol{\Delta}$.
\end{proof}

The converse of Lemma \ref{lem-strongly} will be proved in Section~\ref{sec-trigo}: in the case of a linear solvable pair, every normal element in $R$ or Poisson normal element in $A$ is strongly normal (Proposition~\ref{prop-normal-element}).

\subsection{A canonical form up to localization}\label{ssec-localization}

It is well know that locally nilpotent derivations with slices feature very interesting properties. 
For instance, if $A$ is a $\kk$-algebra that is a domain and $\Delta$ is a nilpotent derivation of $A$ admitting an element $s\in A$ such that $\Delta(s)=1$ ($s$ is called a slice for $\Delta$), then $A$ is a polynomial ring in the variable $s$ over the base ring $A^\Delta=\ker\Delta$, see \cite{DF} and \cite{VdE}.
In this section we exploit this fact to show that, up to localization, the algebras $A_{(\Delta,\Gamma)}$ and $R_{(\Delta,\Gamma)}$ are (Poisson-) Ore extensions of the (localized) kernel of $\Delta$ by the (Poisson) derivation $\Gamma$.

Let $(\Delta,\Gamma)$ be a solvable pair on $A$ with $\Delta\neq0$ and choose $r\in A$ such that $\Delta(r)\neq0$ and $\Delta^2(r)=0$ (any element of $A^{\leqslant 1}\setminus A^{\leqslant 0}$).
We consider the localization $A^\circ=S_r^{-1} A$ where $S_r=\{\Delta(r)^i\ |\ i\geqslant0\}$.
The pair of derivations $(\Delta,\Gamma)$ extends uniquely by localization into a solvable pair of derivations on $A^\circ$ again denoted by $(\Delta,\Gamma)$.
Note that $\Delta$ is locally nilpotent on $A^\circ$ since $\Delta(r)\in A^\Delta$. %, where $A^\Delta$ denote the kernel of $\Delta$ as a derivation from $A$ to $A$.
Observe that the element $s:=r\Delta(r)^{-1}\in A^\circ$ is a slice for $\Delta$ since
 \[\Delta(s)=\Delta(r)\Delta(r)^{-1}-r\Delta(r)^{-2}\Delta^2(r)=1.\]

\subsubsection{The Poisson case.}

We consider the Poisson algebra $(A^\circ)_{(\Delta,\Gamma)}$ arising from the solvable pair $(\Delta,\Gamma)$ on $A^\circ$. 
Note that the kernel $(A^\circ)^\Delta$ is a Poisson commutative subalgebra of $A^\circ$ that is stable by $\Gamma$ since for any $a\in (A^\circ)^\Delta$ we have $\Delta\Gamma(a)=\Gamma\Delta(a)+\Delta(a)=0$.
Hence $\Gamma$ induces a Poisson derivation of $(A^\circ)^\Delta$.
This allows us to define the so-called Poisson-Ore extension $(A^\circ)^\Delta[X;\Gamma]_P$ which is equal to the polynomial ring $(A^\circ)^\Delta[X]$ as a commutative algebra and whose Poisson bracket is defined to be the unique extension of the Poisson bracket of $(A^\circ)^\Delta$ such that $\{X,a\}=\Gamma(a)$ for all $a\in (A^\circ)^\Delta$, see \cite{Oh}. 

\begin{proposition}
The map $\phi : (A^\circ)^\Delta[X;\Gamma]_P \to (A^\circ)_{(\Delta,\Gamma)}$ given by
\[\phi\left(\sum_{i=0}^n a_i X^i\right) =\sum_{i=0}^n a_i s^i\]
where $a_i\in (A^\circ)^\Delta$ for all $i$, is a Poisson algebra isomorphism. 
\end{proposition}

\begin{proof} The inclusion $(A^\circ)^\Delta \to A^\circ$ is a Poisson algebras homomorphism.
Moreover for $a \in (A^\circ)^\Delta$, we have $\{s,a\} = \Delta(s)\Gamma(a)-\Gamma(s)\Delta(a) = \Gamma(a)$ since $\Delta(s)=1$. 
Thanks to the universal property of the Poisson-Ore extension $((A^\circ)^\Delta)[X;\Gamma]_P$ (see \cite[Proposition 1.1.15]{Lec}), there exists a unique Poisson algebra homomorphism from $((A^\circ)^\Delta)[X;\Gamma]_P$ to $(A^\circ)_{(\Delta,\Gamma)}$
sending $a \in (A^\circ)^\Delta$ to $a$ and $X$ to $s$.
It is given by $\phi$. 
Since $A^\circ=(A^\circ)^\Delta[s]$ is a polynomial ring in the slice $s$ over $(A^\circ)^\Delta$ (\cite[Corollary 1.2]{VdE}), it is clear that $\phi$ is an isomorphism.
\end{proof}

% Proposition sans la propriété universelle
%\begin{proposition}
%The map $\phi : (A^\circ)_{(\Delta,\Gamma)} \to ((A^\circ)^\Delta)_{(\Delta,\Gamma)}[X;\Gamma]_P$ given by
%\[\phi\left(\sum_{i=0}^n a_i s^i\right)=\sum_{i=0}^n a_i X^i\]
%for all $n\in\N$ and all $a_i\in A^\circ$ is a Poisson algebra isomorphism. 
%\end{proposition}
%
%\begin{proof}
%Since $A^\circ=(A^\circ)^\Delta[s]$ is a polynomial ring in $s$ over $(A^\circ)^\Delta$, it is clear that $\phi$ is an algebra isomorphism.
%Moreover for all $a,b\in (A^\circ)^\Delta$ we have
%\[\phi(\{a,b\}) = \phi\big(\Delta(a)\Gamma(b)-\Delta(b)\Gamma(a)\big) = \phi(0) = 0 = \{a,b\}=\{\phi(a),\phi(b)\}\]
%and for all $a\in (A^\circ)^\Delta$ we have
%\[\phi(\{s,a\})=\phi\big(\Delta(s)\Gamma(a)-\Delta(a)\Gamma(s)\big)=\phi(\Gamma(a))=\Gamma(a)=\{X,a\}=\{\phi(s),\phi(a)\}\] 
%and $\phi$ is a Poisson map.
%\end{proof}
\subsubsection{The noncommutative case}

We consider the algebra $R^\circ=(R^\circ)_{(\Delta,\Gamma)}=(A^\circ,\ast)$ arising from the solvable pair $(\Delta,\Gamma)$ on $A^\circ$.
Recall that
\[a\ast b = ab+\sum_{i\geqslant 1}\Delta^i(a)\binom{\Gamma}{i}(b)\]
for all $a,b\in A^\circ$.
In particular $(A^\circ)^\Delta$ is a commutative subalgebra of $R^\circ$.
Since $\Gamma$ leaves $(A^\circ)^\Delta$ invariant, we can consider the Ore extension $(A^\circ)^\Delta[X;\Gamma]$ which is, by definition \cite[Chapter 2]{GW}, a free left $(A^\circ)^\Delta$-module with basis $\{X^i\ |\ i\geqslant 0\}$ and whose multiplication extends the multiplication of $(A^\circ)^\Delta$ by the rule $Xa=aX+\Gamma(a)$ for all $a\in (A^\circ)^\Delta$.
For all integer $i\geqslant0$ let us denote by $s^{\ast i}$ the $i^{\textrm{th}}$ power of $s$ with respect to the product $\ast$. 

\begin{proposition}
\label{Oreext}
The map $\phi : (A^\circ)^\Delta[X;\Gamma] \to R^\circ$ given by
\[\phi\left(\sum_{i=0}^n a_i X^i\right)=\sum_{i=0}^n a_i s^{\ast i}\]
where $a_i\in (A^\circ)^\Delta$ for all $i$, is an algebra isomorphism.
\end{proposition}

\begin{proof} The inclusion map $(A^\circ)^\Delta \subseteq R^\circ$ is an algebra homomorphism.
Moreover for $a \in (A^\circ)^\Delta$, we have $a \ast s = as$ and $s \ast a= sa +\Gamma(a)$ since 
$\Delta(s)=1$ and $\Delta^2(s)=0$.
Hence from the universal property (see \cite[Exercise 2.F]{GW}) of $(A^\circ)^\Delta[X;\Gamma]$ there exists a unique algebra homomorphism from $(A^\circ)^\Delta[X;\Gamma]$ to $R^\circ$ sending $a \in (A^\circ)^\Delta$ to $a \in R^\circ$ and $X$ to $s$.
It is given by $\phi$.
So it remains to show that $\phi$ is bijective.
For, we will show by induction on $i\geqslant0$ that there exist elements $a_{j,i}\in (A^\circ)^\Delta$  with $0\leqslant j<i$ such that $s^{\ast i}=s^i + \sum_{j=0}^{i-1} a_{j,i}s^j$.
Cases $i=0$ and $i=1$ are trivial.
Recall that $\Delta(s)=1$ and $\Delta^2(s)=0$ in order to compute $s^{\ast (i+1)}$
\begin{align*}
s^{\ast (i+1)} = s \ast s^{\ast i}=& s \ast s^i + s \ast \sum_{j=0}^{i-1} a_{j,i}s^j =s^{i+1} + \Gamma(s^i) + \sum_{j=0}^{i-1} a_{j,i}s^{j+1} + \sum_{j=0}^{i-1} \Gamma(a_{j,i}s^j)\\
=&s^{i+1} + i(\Gamma(s)-s)s^{i-1} + is^i + \sum_{j=1}^{i} a_{j-1,i}s^{j} + \sum_{j=0}^{i-1} \Gamma(a_{j,i}) s^j + \sum_{j=0}^{i-1} a_{j,i} \Gamma(s^j)\\
=&s^{i+1} + i(\Gamma(s)-s)s^{i-1} + \sum_{j=0}^{i} b_{j,i}s^j + \sum_{j=0}^{i-1} j a_{j,i} (\Gamma(s) - s ) s^{j-1} + \sum_{j=0}^{i-1} ja_{j,i}s^j
\end{align*}
where the $b_{j,i}$ for $0\leqslant j\leqslant i$ are elements of  $(A^\circ)^\Delta$ expressed in terms of $i,$ $a_{j-1,i}$ and $\Gamma(a_{j,i})$.
Moreover we have $\Delta(\Gamma(s)-s) = \Gamma \Delta(s) = \Gamma(1)=0$.
Hence $\Gamma(s)-s \in (A^\circ)^\Delta$ and $s^{\ast (i+1)}$ has the desired form.
We then deduce that $(s^{\ast i})_{i\geqslant0}$ is a basis of $R^\circ$ over $(A^\circ)^\Delta$ since $(s^i)_{i \geqslant0}$ is (\cite[Corollary 1.2]{VdE}) and we conclude that $\phi$ is bijective.
\end{proof}

%\begin{corollary}
%Assume that $A^\circ$ is finitely generated {\color{red}(which is the case if $A$ is)}.
%Then the algebra $R^\circ$ is noetherian.
%\end{corollary}

%\begin{proof}
%The kernel $(A^\circ)^\Delta$ is finitely generated thanks to \cite{VdE} and the hypothesis on %$A^\circ$.
%Therefore the Ore extension $R^\circ\cong(A^\circ)^\Delta[X;\Gamma]$ is noetherian thanks to the Skew Hilbert Basis Theorem \cite[Theorem 2.6]{GW}.
%\end{proof}

\subsection{Skewfield of fractions}

When $A$ is a polynomial ring and the action of $\Gamma$ is diagonal, the results of the previous section allow us to relate $R^\circ$ to a particular Lie algebra and to described the skewfield of fractions of $R_{(\Delta,\Gamma)}$.
Recall that, given a field $F$, the Weyl skewfield $D_{1}(F)$ is the skewfield of fractions of the algebra generated over $F$ by two elements $u$ and $v$ such that $uv-vu=1$.
Linear diagonalizable solvable pairs will be defined in Section \ref{sec-diago}.

\begin{proposition}
\label{Weyl}
Assume that $A$ is a polynomial ring and that $(\Delta,\Gamma)$ is a nonzero linear diagonalizable solvable pair on $A$. 
Then $R_{(\Delta,\Gamma)}$ is, up to localization,  isomorphic to the enveloping algebra of a completely solvable Lie algebra $\mathfrak{g}$.
In particular, if $\mathfrak{g}$ is algebraic and non abelian, the skewfield of fractions of $R_{(\Delta,\Gamma)}$ is isomorphic to the Weyl skewfield $D_{1}(F)$ for a purely transcendental extension $F$ of $\kk$.
\end{proposition}

\begin{proof}
Since $\Delta(X_1)=X_0$ and $\Delta(X_0)=0$ we can apply the slice construction with the element $s=X_1X_0^{-1}$ of $A^\circ=A[X_0^{-1}]$.
In particular, the map $\pi:A^{\circ}\to A^{\circ}$ given by $\pi(a)=\sum_{p\geqslant 0}\frac{(-s)^p}{p!}\Delta^p(a)$ is an algebra homomorphism such that $(A^\circ)^\Delta=\pi(A^\circ)$, see \cite{VdE}.
Therefore the algebra $(A^\circ)^\Delta$ is generated by elements $Y_0^{\pm1},Y_2,\dots,Y_n$ where $Y_i$ denote the image of $X_i$ by $\pi$ (note that $\pi(X_0)=X_0$ and $\pi(X_1)=0$).  
Observe that ${Y_0,Y_2,\dots,Y_n}$ is a set of algebraically independent elements of $A^\circ$ since, for $i\geqslant 2$, we can write $Y_i=\pi(X_i)=X_i+F_i$ where $F_i\in\kk[X_0^{\pm1},X_1,\dots,X_{i-1}]$.
Moreover each $Y_i$ is an eigenvector for $\Gamma$ because one can verify that $\Gamma$ and $\pi$ commute thanks to the equality $\Delta\Gamma-\Gamma\Delta=\Delta$.

Denote by $\mathfrak{g}$ the completely solvable Lie algebra with basis $\{y_0,y_2,\dots,y_n,x\}$ and nonzero brackets $[x,y_i]=\lambda_i y_i$ for all $i$, where $\lambda_i\in \kk$ is the eigenvalue of $Y_i$ (equivalently $X_i$) for $\Gamma$.
The element $y_0$ is normal in $U(\mathfrak g)$ and is its clear that $(A^\circ)^\Delta[X;\Gamma]$ is isomorphic to the localization $U(\mathfrak g)[y_0^{-1}]$.
The first part of the proof is now complete since $R^\circ\cong(A^\circ)^\Delta[X;\Gamma]$ thanks to Proposition \ref{Oreext}.

The second assertion follows firstly from \cite{Mc}: skewfields of fractions of enveloping algebras of algebraic completely solvable Lie algebras are isomorphic to Weyl skewfields.
Moreover, the transcendence degree of the center of the enveloping skewfield of $\mathfrak{g}$ is equal to the index of $\mathfrak{g}$ whose value can be easily computed to be $n-1$ thanks to \cite{Ooms}.
The result follows.
\end{proof}

\begin{remark}
With similar arguments (using \cite{TY}) the following Poisson version can be obtained: under the assumption of Proposition \ref{Weyl}, the Poisson field of $A_{(\Delta,\Gamma)}$ is isomorphic to a Poisson Weyl field over a purely transcendental extension $F$ of $\kk$, that is, a rational functions field $F(U,V)$ with $\{U,V\}=1$.
\end{remark}

\section{The rank of $A_{(\Delta,\Gamma)}$}
\label{sec-pol-ring}

We recall the notion of rank of an affine Poisson variety following \cite{Van}.
Let $V=\kk^{n+1}$ be a Poisson variety, that is, an affine variety $V$ whose coordinate ring $A=\mathcal{O}(\kk^{n+1})=\kk[X_0,\dots,X_n]$ is endowed with a Poisson structure.
The Poisson matrix of $A$ is the skew-symmetric matrix $\Pi(A)=\big(\{X_i,X_j\}\big)_{0\leqslant i,j\leqslant n}$.
For any $p\in V$ the rank $\mbox{Rk}_p(A)$ of the Poisson bracket at $p$ is the rank of the Poisson matrix $\Pi(A)$ evaluated at $p$.
Finally we define the rank of $A$ to be the maximum of the $\mbox{Rk}_p(A)$ for $p\in V$. 
Note that $A$ is Poisson commutative if and only if its rank is zero. 
If $(\Delta,\Gamma)$ is a solvable pair on $A$ we have
\[\Pi(A_{(\Delta,\Gamma)})=
\begin{pmatrix}
\Delta(X_0)\\
\Delta(X_1)\\
\vdots\\
\Delta(X_n)
\end{pmatrix}
\begin{pmatrix}
\Gamma(X_0)\\
\Gamma(X_1)\\
\vdots\\
\Gamma(X_n)
\end{pmatrix}^T
-
\begin{pmatrix}
\Gamma(X_0)\\
\Gamma(X_1)\\
\vdots\\
\Gamma(X_n)
\end{pmatrix}
\begin{pmatrix}
\Delta(X_0)\\
\Delta(X_1)\\
\vdots\\
\Delta(X_n)
\end{pmatrix}^T.
\]
Therefore, when $A_{(\Delta,\Gamma)}$ is not Poisson commutative, its rank is equal to $2$.
When $\Delta$ and $\Gamma$ are linear, the algebra $A_{(\Delta,\Gamma)}$ is the semiclassical limit of the AS regular algebra $R_{(\Delta,\Gamma)}$ (see Corollary \ref{Cor-ASregul}), which explains the title of the article.

Note that, since the rank of a Poisson algebra $A$ is invariant under isomorphism thanks to \cite[Proposition 2.17]{Van}, there exist Poisson brackets that cannot be obtained from solvable pairs since Poisson structures with higher ranks do exist.
Moreover, even in rank $2$ it is not possible to obtain every Poisson structure, not even homogeneous ones, from a solvable pair, see Example~\ref{exdim2produit}.

\section{Linear solvable pairs} \label{sec-linear}

From this section included and until the end of this article we focus on the case of linear solvable pairs on polynomial algebras.
Recall that the polynomial algebra $A=\kk[X_0,\dots,X_n]=\bigoplus_{i\geqslant0}A_i$ is a connected graded algebra, where $A_i$ denote the vector space of homogeneous polynomials of degree $i$.
A derivation of $A$ is called linear if it is homogeneous of degree $0$.
Note that linear derivations are in one to one correspondence with endomorphisms of $A_1$ or equivalently with square matrices of order $n+1$ (acting on $A_1$ by left multiplication).
Moreover, locally nilpotent derivations of $A$ correspond to nilpotent endomorphisms/matrices.
When dealing with a linear derivation of $A$, typically $\Delta$, its rank will always mean the rank of its restriction  as an endomorphism of (a subspace of) $A_1$.

A Poisson structure on $A$ is called homogeneous if $\{A_i,A_j\}\subseteq A_{i+j}$ and in that case $A$ is said to be an homogeneous polynomial Poisson algebra. 
We provide a couple of invariants for homogeneous polynomial Poisson algebras.
Denote by ${\rm P.Der}_{\rm gr}(A)$ the Lie algebra of linear Poisson derivations of $A$.

\begin{proposition}
\label{isoda}
Let $A$ and $B$ be two isomorphic homogeneous polynomial Poisson algebras.
\begin{enumerate}[{\rm (1)}]
\item There exists a Lie algebra isomorphism between ${\rm P.Der}_{\rm gr}(A)$ and ${\rm P.Der}_{\rm gr}(B)$ that sends locally nilpotent derivations to locally nilpotent derivations.
\item There exists a vector space isomorphism between $A_1\cap Z_P(A)$ and $B_1\cap Z_P(B)$.
\end{enumerate}
\end{proposition}

\begin{proof}
Thanks to \cite[Proposition 8.8]{LGPV} a Poisson isomorphism between homogeneous polynomial Poisson structures can always be realized by a linear isomorphism.
We denote by $\varphi:A\to B$ such an isomorphism.
For (1), the map $\Psi:{\rm P.Der}_{\rm gr}(A)\rightarrow {\rm P.Der}_{\rm gr}(B)$ given by $\Psi=\varphi\delta\varphi^{-1}$ is the desired isomorphism and for (2), the restriction of $\varphi$ to $A_1\cap Z_P(A)$ provides us with the desired isomorphism since $\varphi(A_1)=B_1$ and 
$\varphi(Z_P(A))=Z_P(B)$.
\end{proof}

\begin{definition} A solvable pair $(\Delta,\Gamma)$ is said to be linear if $\Delta$ and $\Gamma$ are linear derivations of $A$.
\end{definition}

\begin{remark}
In a solvable pair, the derivation $\Delta$ is locally nilpotent.
It implies, when the solvable pair is linear, that $\Delta$ is a nilpotent endomorphism of $A_1$.
We remark that when $\Delta$ and $\Gamma$ are linear derivations of $A$ verifying $[\Delta,\Gamma]=\Delta$ then $\Delta$ is automatically locally nilpotent.
The proof goes as follows.
We have $[P(\Delta),\Gamma]= (TP')[\Delta]$ for every polynomial $P \in \kk[T]$ by assertion (2) of Lemma \ref{filtdg}.
If $P$ is the minimal polynomial of $\Delta$ acting on the finite dimensional space $A_1$, we obtain
that $TP'(T)$ is collinear to $P$.
Hence $P=T^\ell$ for some $\ell$ and $\Delta$ acting on $A_1$ is nilpotent.
Therefore $\Delta$ is locally nilpotent on $A$ and $(\Delta,\Gamma)$ a linear solvable pair. 
\end{remark}

The article~\cite{LS} consider the case where $\Delta$ is the maximal Jordan block.
In the following we will investigate algebras $A_{(\Delta,\Gamma)}$ and $R_{(\Delta,\Gamma)}$ arising from linear derivations $\Delta$ with more than one Jordan block. 
Let $(\Delta,\Gamma)$ be a linear solvable pair on $A$.
Thanks to Lemma~\ref{iso} one can always assume that $\Delta$ is in canonical Jordan form, up to a linear automorphism of $A$.
Recall that in canonical Jordan form, a nilpotent matrix is made of diagonal blocks of the form $J_{n_i}(0)$, 
i.e. blocks of size $n_i$ with $0$s everywhere except on the upper diagonal, where the entries are $1$. 
It can be assumed that the size of the diagonal blocks $J_{n_i}(0)$ is decreasing. 
So, given such a $\Delta$, what are the possible choices for $\Gamma$?

\begin{lemma}
The set of endomorphism $\Gamma\in{\rm End}_\kk(A_1)$ such that $[\Delta,\Gamma]=\Delta$ is given by
\[\Gamma_0+\mathcal{C}(\Delta)\]
where $\mathcal{C}(\Delta)=\{f\in{\rm End}_\kk(A_1)\ ;\ [f,\Delta]=0\}$ is the commutant of $\Delta$ in ${\rm End}_\kk(A_1)$ and where $\Gamma_0\in{\rm End}_\kk(A_1)$ is such that $[\Delta,\Gamma_0]=\Delta$.
\end{lemma}

\begin{proof}
Follows from the fact that for two solvable pairs $(\Delta,\Gamma)$ and $(\Delta,\tilde{\Gamma})$ 
we have $[\Delta, \Gamma-\tilde{\Gamma}]=0$.   
\end{proof}

\subsection{The case of $\Delta$ maximal}
\label{caseDeltamax}

When $\Delta$ is a maximal Jordan block one can choose $\Gamma_0={\rm diag}(0,1,\dots,n)$. %Moreover it is classical that the commutant of $\Delta$ is $\mathcal{C}(\Delta)={\rm Span}(\id,\Delta,\Delta^2,\dots,\Delta^{n})$.
Recall the algebras $A(n,a)$ and $R(n,a)$ from Example \ref{LSnot}.

\begin{theorem}
\label{dmax}
Given a solvable pair $(\Delta,\Gamma)$ with $\Delta$ a maximal Jordan block, there exists  $u\in{\rm GL}(A_1)\cap\mathcal{C}(\Delta)$ such that $u\Gamma u^{-1}=a {\rm Id}_{A_1}+\Gamma_0$, where $a\in\kk$ is such that ${\rm tr}(\Gamma)=(n+1)(a+n/2)$.
In particular $A_{(\Delta,\Gamma)}\cong A(n,a)$ and $R_{(\Delta,\Gamma)}\cong R(n,a)$.
\end{theorem}

The proof relies on the following classical facts.
Recall that $A_1=\kk X_0\oplus\cdots\oplus\kk X_n$ is a $\kk$-vector space of dimension $n+1$.
\begin{enumerate}[(1)]
\item We have $\mathcal{C}(\Delta)=\kk[\Delta]={\rm Span}(\id,\Delta,\Delta^2,\dots,\Delta^{n})$ as a subalgebra of ${\rm End}_\kk(A_1)$.
%We have $\mathcal{C}(\Delta)=\kk[\Delta]$ as a subalgebra of ${\rm End}_\kk(A_1)$.
%Moreover $\mathcal{C}(\Delta)={\rm Span}(\id,\Delta,\Delta^2,\dots,\Delta^{n})$.
\item Let $u\in\mathcal{C}(\Delta)$ and $P\in\kk[X]$ such that $u=P(\Delta)$.
Then $u\in{\rm GL}(A_1)$ if and only if $P(0)\neq0$.
%\item For any $P\in\kk[X]$ we have $[P(\Delta),\Gamma_0]=Q(\Delta)$, where $Q(X)=XP'(X)$ (follows easily by induction that $[\Delta^i,\Gamma_0]=i\Delta^{i-1}$ for all $i\geqslant 0$).
\end{enumerate}

%\begin{lemma} \label{comm}
%Consider a solvable pair $(\Delta,\Gamma)$. For any $P\in\kk[X]$ we have $[P(\Delta),\Gamma]=Q(\Delta)$, where $Q(X)=XP'(X)$.
%\end{lemma}
%\begin{proof}
%By linearity it suffices to show the result for $P=X^i$, $i\geqslant 0$.
%For $i=0$ and $i=1$ the result is clear.
%Let $i>0$ and assume that $[\Delta^i,\Gamma]=i\Delta^{i}$.
%Then
%\[[\Delta^{i+1},\Gamma]=\Delta^{i+1}\Gamma-\Gamma\Delta^{i+1}=\Delta^i([\Delta,\Gamma])+\Delta^i\Gamma\Delta-\Gamma\Delta^{i+1}=\Delta^{i+1}+[\Delta^i,\Gamma]\Delta=(i+1)\Delta^{i+1}.\]
%\end{proof}

\begin{proof}[Proof of Theorem \ref{dmax}]
The relation $[\Delta^i,\Gamma]=i\Delta^{i}$ shows that $\ker \Delta^i$ is stable by $\Gamma$. 
So let us consider a basis $(e_0,\ldots,e_n)$ of $A_1$ in which $\Delta$ has canonical Jordan form.
In such a basis $\Gamma$ is upper triangular.
Let us denote by $(\lambda_0,\ldots,\lambda_n)$ the diagonal part of $\Gamma$. 
We have 
\[\Gamma(e_i) = \Gamma(\Delta^{n-i}(e_n))= \Delta^{n-i}\Gamma(e_n) - (n-i) \Delta^{n-i}(e_n)\,.
\]
Thus by computing modulo $\textrm{vect}(e_0,\ldots, e_{i-1})$, we get that $\lambda_{i}=\lambda_{n}-(n-i)$ for all $0\leqslant i<n$.
In particular $\Gamma$ has $n+1$ distinct eigenvalues and thus a basis of eigenvectors. We then deduce that there exists an eigenvector $e'_n$ of $\Gamma$ which is not in $\ker \Delta^n$.
Using the relation $[\Delta^i,\Gamma]=i\Delta^{i}$, 
we obtain that $(\Delta^n(e'_n),\Delta^{n-1}(e'_n),\ldots, e'_n)$ is a basis of $A_1$ in which $\Delta$ has
canonical Jordan form and $\Gamma$ is diagonal $(a,a+1,\ldots, a+n)$ where $a=\lambda_0=\lambda_n-n$.
Finally we get
\[{\rm tr}(\Gamma)=a(n+1)+\frac{n(n+1)}{2}\]
showing that $a$ is uniquely determined by $\Gamma$.
\end{proof}

\begin{example}
\label{exn=1}
Let $n=1$.
There exists only two isomorphism classes for the algebras $A=A_{(\Delta,\Gamma)}$ and $R=R_{(\Delta,\Gamma)}$ with non necessarily maximal linear solvable pairs.
If $\Delta=0$ then $A=R=\kk[X_0,X_1]$ is a (Poisson) commutative algebra.
If $\Delta\neq0$ then it is necessarily a maximal Jordan block since $n=1$.
Therefore, thanks to Theorem \ref{dmax}, we have $A\cong A(1,a)$ and $R\cong R(1,a)$ with $a\in\kk$.
If $a=0$ we retrieve the commutative case (this is the example of Remark~\ref{rem-commutative-example}) 
and when $a\neq0$, $A$ is isomorphic to the Poisson algebra $\kk[x,y]$ with Poisson bracket $\{x,y\}=x^2$ and $R$ is isomorphic to the Jordan plane, that is, the algebra given by two generators $x,y$ and the relation $xy-yx=x^2$.
\end{example}

\subsection{Center in the linear case.}

The following result is a generalization of~Lemma~\ref{lem-Pcenter} in the linear case.
We describe the Poisson center of $A$ and the center of $R$ which allow us to determine exactly when $A$ is Poisson commutative and $R$ is commutative.
The results of this section will be improved by Proposition~\ref{prop-normal-element} but are stated here since they will be needed in the proof of Proposition~\ref{prop-normal-element}.
%In the next proposition the rank of $\Delta$ seen as an endomorphism of $A_1$ is denoted by $\rm{rank}(\Delta)$.

\begin{proposition}\label{prop-Pcenter-lin} Let $(\Delta,\Gamma)$ be a linear solvable pair on $A$. Assume that the rank of $\Delta$ is at least $2$.
Then we have $A_1 \cap Z_P(A)=\ker \Delta \cap \ker \Gamma \cap A_1$ and $A_1 \cap Z(R)=\ker \Delta \cap \ker \Gamma \cap A_1$.
\end{proposition}

\begin{proof} Since in this proof we only consider restrictions of $\Delta$ and $\Gamma$ to $A_1$, 
we simply denote by $\ker \Delta$ the set $\ker \Delta \cap A_1$ and similarly for $\Gamma$. 
Since we have $\ker \Gamma \cap \ker \Delta \subseteq Z_P(A)$ and $\ker \Gamma \cap \ker \Delta \subseteq Z(R)$ with equality if $\ker \Gamma \neq \ker \Delta$ thanks to Lemma~\ref{lem-Pcenter}, it only remains to study the case $\ker \Gamma=\ker \Delta$.

Consider $z \in Z_P(A) \cap A_1$. 
Let us first show that $z \in \ker \Delta^2$.
If not, $z, \Delta(z)$ and $\Delta^2(z)$ are linearly independent over $\kk$.
Then $0=\{z,\Delta(z)\}=\Delta(z)\Gamma \Delta(z) - \Delta^2(z)\Gamma(z)$. 
Since $\ker \Delta = \ker \Gamma$ and $z \notin \ker \Delta^2$ we get that $\Gamma\Delta(z) \neq 0$. 
Moreover $\Delta(z), \Gamma \Delta(z), \Delta^2(z)$ and $\Gamma(z)$ are homogeneous element of $A$ of degree $1$, hence are irreducible element of the unique factorization domain $A$.
Since $\Delta(z)$ and $\Delta^2(z)$ are not collinear, they are non associated irreducible elements of $A$.
The unique factorization of $\Delta(z)\Gamma \Delta(z) = \Delta^2(z)\Gamma(z)$ ensures us that 
there exists $\alpha \in \kk$ such that $\Gamma \Delta(z)= \alpha \Delta^2(z)$ and $\Gamma(z)= \alpha \Delta(z)$. 
Computing $\Delta(z)=[\Delta,\Gamma](z)= \alpha \Delta^2(z) -\alpha\Gamma \Delta(z)=0$, we get a contradiction. 
Hence $z \in \ker \Delta^2$.

Before pursuing the computation of the Poisson center of $A$, we switch for a moment to the study of the center of $R$.
Consider $z \in Z(R) \cap A_1$.
As in the Poisson case, we first show that $z \in \ker \Delta^2$.
If not, the integer such that $\Delta^n(z) \neq0$ et $\Delta^{n+1}(z)=0$ is greater than $1$.
We now compute $z \ast \Delta^{n-1}(z)=\Delta^{n-1}(z)\ast z$. For $\Delta((\Gamma-\id)(\Delta^{n-1}(z)) = \Gamma\Delta^{n}(z) + \Delta^n(z)-\Delta^n(z) =0$.
Indeed $\Delta^n(z) \in \ker \Delta=\ker \Gamma$. So $(\Gamma-\id)(\Delta^{n-1}(z)) \in \ker \Delta=\ker\Gamma$.
Hence $\Delta^{n-1}(z)\in \ker\binom{\Gamma}{2}$. Hence the relation $z \ast \Delta^{n-1}(z)=\Delta^{n-1}(z)\ast z$ becomes $\Delta(z)\Gamma(\Delta^{n-1}(z))=\Gamma(z)\Delta^{n}(z)$.
Since $n > 1$ the element $\Delta(z)$ and $\Delta^n(z)$ are linearly independents.
As in the preceding paragraph, we deduce that there exists $\alpha \in \kk$ such that $\alpha \Delta(z)=\Gamma(z)$ and $\alpha\Delta^{n}(z)=\Gamma(\Delta^{n-1}(z))$.
The relation $[\Delta^{n-1},\Gamma](z)=(n-1)\Delta^{n-1}(z)$ can then be written as $\alpha \Delta^{n}(z)-\alpha\Delta^{n}(z)=(n-1)\Delta^{n-1}(z)$.
Hence $(n-1)\Delta^{n-1}(z)=0$, which is absurd since $n>1$ and $\Delta^{n-1}(z)\neq0$.
We deduce that $z \in \ker \Delta^2$ and therefore that $z=\Delta^{n-1}(z) \in \ker\binom{\Gamma}{2}\,.$

We now finish the proof by showing that in both cases $z \in \ker \Delta$ when the rank of $\Delta $ is at least equal to $2$. 
If $z \notin\ker \Delta$ then there exists $y$ such that $\Delta(y)$ is not collinear to $\Delta(z)$. 
For the Poisson case, we have $0=\{z,y\}=\Delta(z)\Gamma(y) - \Delta(y)\Gamma(z)$.
For the case of $R$, we have $z\ast y = zy + \Delta(z)\Gamma(y)$ since $z \in \ker \Delta^2$. But 
$y \ast z = yz + \Delta(y)\Gamma(z)$ since $z \in \ker \binom{\Gamma}{2}$. Finally, we also have the equality $\Delta(z)\Gamma(y)=\Delta(y)\Gamma(z)$.
Since $\ker \Delta=\ker \Gamma$, we get that $\Gamma(y) \neq 0$ and $\Gamma(z) \neq 0$.
Using the fact that $\Delta(y)$ and $\Delta(z)$ are not collinear, we can mimic the argument of the two preceding paragraphs to get the relation $\Gamma(z)=\alpha \Delta(z)$ for some $\alpha \in \kk$. 
We then obtain $\Delta(z)=[\Delta,\Gamma](z)= \alpha \Delta^2(z) - \Gamma \Delta(z)$.
But $z \in \ker \Delta^2$ and hence $\Delta(z) \in \ker \Delta = \ker \Gamma$.
Therefore $\Delta(z)=0$ and $z \in \ker\Delta\cap \ker \Gamma$.
\end{proof}

\begin{corollary}\label{cor-poissoncom} Let $(\Delta,\Gamma)$ be a linear solvable pair on $A$. Then 
$A_1 \cap Z_P(A) \neq \ker \Delta \cap \ker \Gamma \cap A_1$ implies that $A$ is Poisson commutative and 
$A_1 \cap Z(R) \neq \ker \Delta \cap \ker \Gamma \cap A_1$ implies that $R$ is commutative.

Moreover $A$ is Poisson commutative if and only if $R$ is commutative if and only if $\Delta =0$ or $(\Delta,\Gamma)$ is conjugated to $(X_0\partial_{X_1},X_1\partial_{X_1})$.
%$(\Delta_0,\Gamma_0)$ where 
%$$\Delta_0=\begin{pmatrix} 0&1& \\ &0&& \\ &&\ddots&\\ &&&0 \end{pmatrix}= X_0\partial_{X_1} \qquad \textrm{and}\qquad 
%\Gamma_0=\begin{pmatrix} 0&&&& \\ &1&&& \\ &&0&&\\ &&&\ddots& \\ &&&&0 \end{pmatrix}= X_1\partial_{X_1}$$
\end{corollary}

\begin{proof} Assume that $A_1 \cap Z_P(A) \neq \ker \Delta \cap \ker \Gamma \cap A_1$ or $A_1 \cap Z(R) \neq \ker \Delta \cap \ker \Gamma \cap A_1$,
Lemma~\ref{lem-Pcenter} and Proposition~\ref{prop-Pcenter-lin} show that $\Delta=0$ (and $\Gamma \neq 0$) 
or $\Delta$ is of Jordan type $(2,1,\ldots,1)$ and $\ker \Gamma=\ker \Delta$.

In the first case $A$ is Poisson commutative and $R$ is commutative.
In the second case, up to a change of basis, we have $\Delta=X_0\partial_{X_1}$ and $(X_0,X_2,\dots,X_n)$ is a basis of $\ker\Delta=\ker\Gamma$.
Hence $\Gamma=(\alpha_0X_0+\alpha_1X_1+\cdots+\alpha_nX_n)\partial_{X_1}$.
The relation $[\Delta,\Gamma]=\Delta$ ensures that $\alpha_1=1$.
Replacing $X_1$ by $X_1+\alpha_0X_0+\alpha_2X_2+\cdots+\alpha_nX_n$ in the basis, we obtain $\Gamma=X_1\partial_{X_1}$ as desired.
Since $\ker \Gamma=\ker \Delta$ we have $\{X_i,X_j\} =0$ and $X_i \ast X_j=X_iX_j$ as soon as $i$ or $j$ belongs to $\{0,2,\ldots,n\}$.
In particular $R$ is commutative and $A$ is Poisson commutative.
%In the first case $A$ is clearly Poisson commutative and $R$ is commutative. In the second case, we can choose a basis
%$(e_0,\ldots, e_n)$ of $A_1$ such that $\Delta(e_1)=e_0$ and $(e_0,e_2\ldots, e_n)$ is a basis of $\ker \Delta = \ker \Gamma$. 
%In such a basis, the matrix of $\Delta$ and $\Gamma$ are given by 
%$$\begin{pmatrix} 0&1& \\ &0&& \\ &&\ddots&\\ &&&0 \end{pmatrix} \qquad \textrm{and}\qquad 
%\begin{pmatrix} 0&\alpha_0&0 &\cdots & 0\\ \vdots&\vdots&\vdots&&\vdots \\ \vdots&\vdots&\vdots&&\vdots\\ 0&\alpha_n&0&\cdots&0 \end{pmatrix}$$ 
%The relation $[\Delta,\Gamma]=\Delta$ ensures that $\alpha_1=1$.
%Changing the basis $(e_0,\ldots, e_n)$ to the basis
%$(e_0,e_1 + \alpha_0e_0 + \alpha_2e_2 + \cdots + \alpha_n e_n,e_2,\ldots,e_n)$ we get the matrices $\Delta_0$ and $\Gamma_0$.
%In such a basis, we have that $\{X_i,X_j\} =0$ and $X_i \ast X_j=X_iX_j$ as soon as $i$ or $j$ belongs to $\{0,2,\ldots,n\}$ (since $\ker \Gamma=\ker \Delta$).
%In particular $R$ is commutative and $A$ is Poisson commutative.

It remains so show if $A$ is Poisson commutative or $R$ is commutative then $\Delta=0$ or $(\Delta,\Gamma)$ is conjugated to $(X_0\partial_{X_1},X_1\partial_{X_1})$.
If $A_1 \cap Z_P(A) \neq \ker \Delta \cap \ker \Gamma \cap A_1$ or $A_1 \cap Z(R) \neq \ker \Delta \cap \ker \Gamma \cap A_1$ then the first part of the proof allows us to conclude. 
If $A_1 \cap Z_P(A) = \ker \Delta \cap \ker \Gamma \cap A_1$ or $A_1 \cap Z(R) = \ker \Delta \cap \ker \Gamma \cap A_1$ then $A_1=\ker \Delta \cap \ker \Gamma \cap A_1$
and $\Delta=0$.
\end{proof}

\begin{comment}
{\color{red}Je pense que la section suivante n'est pas necessaire. V 9 février: Ok, je suis d'accord ! J'ai mis en commentaire
\subsection{Application: a non-isomorphism theorem}

The rank 2 property of $(A,\{\cdot,\cdot\})$ may be precised in the linear case with the following proposition. 

\begin{proposition}\label{isotensor} Let $A=A_{(\Delta,\Gamma)}$ where $(\Delta,\Gamma)$ is a linear solvable pair
with $\Delta \neq 0$ and $A_1 \cap \ker \Gamma\cap \ker \Delta = \{0\}$.  If $A\cong\bigotimes_{i=1}^s A_{(\Delta_i,\Gamma_i)}$ 
for linear solvable pair $(\Delta_i,\Gamma_i)$ then $s=1$.
\end{proposition}

\begin{proof} Let $B=\bigotimes_{i=1}^s A_{(\Delta_i,\Gamma_i)}$.
Thanks to Corollary~\ref{cor-nonisom}, there exists a unique index $i$ such that $A_{(\Delta_i,\Gamma_i)}$
is not Poisson commutative. Hence the Poisson center of $B$ contains $\bigotimes_{j\neq i} A_{(\Delta_j,\Gamma_j)}$.
Thanks to~Proposition~\ref{isoda}~(2), we have $Z_P(A) \cap A_1=Z_P(B) \cap B_1$.  
But the hypothesis on $A_1 \cap\ker \Gamma\cap \ker \Delta$ shows that~$Z_P(A)\cap A_1=\{0\}$ (see~Lemma~\ref{lem-Pcenter}). 
We thus conclude that $s=1$. 
\end{proof}
}
\end{comment}

\subsection{A finer filtration for the linear case}
\label{finerfiltation}

In section~\ref{ssec-filtration} we defined a filtration $\varepsilon$ on any commutative $\kk$-algebra $A$ endowed with a solvable pair $(\Delta,\Gamma)$.
When $A$ is a polynomial ring and $\Delta$, $\Gamma$ are linear, we can define a finer filtration $\wteps$ on $A$.
The new filtration is defined as the extension to $A$ of the restriction to $A_1$ of the filtration $\varepsilon$.
Recall that $A_1$ denote the set of homogeneous polynomial of degree $1$ which is stable by $\Delta$ and $\Gamma$ since the pair $(\Delta,\Gamma)$ is linear.
Appendix~\ref{symalg-filtre} is devoted to the construction of this filtration $\wteps$ and the study of its properties, see in particular Example~\ref{example-epsilon-tilde}. 
When $\Delta$ is a maximal Jordan block, then $\wteps$ is the filtration used in~\cite[Section 3]{LS}. 

For this filtration $\wteps$, analogs of Lemma~\ref{filtdg} and Proposition~\ref{prop-linkRA} are valid. Precisely, we have the following proposition. 

\begin{proposition}\label{prop-epsilontilde}
Let $(\Delta,\Gamma)$ be a linear solvable pair on $A=\kk[X_0,\dots,X_n]$ and fix integers $i,j\geqslant0$. Set $A^{\leqslant i,\wteps} = 
\{x \in A,\ \wteps(x) \leqslant i\}$.
\begin{enumerate}
 \item[(1)] We have $\Delta(A^{\leqslant i,\wteps})\subseteq A^{\leqslant i-1,\wteps}$ and so $\wteps(\Delta(f))\leqslant \wteps(f)-1$ for every $f \in A$.
 \item[(2)] We have $\Gamma(A^{\leqslant i,\wteps})\subseteq A^{\leqslant i,\wteps}$. Hence $\wteps(P(\Gamma)(f))\leqslant \wteps(f)$ for every $P \in \kk[T]$ and $f \in A$.
 \item[(3)] If $f,g\in A$ are such that $\wteps(f)=i$ and $\wteps(g)=j$, then $\wteps(fg)=i+j$.
 \item[(4)] The family $(A^{\leqslant i,\wteps})_{i\in \N}$ is a Poisson algebra filtration of $A_{(\Delta,\Gamma)}$ of degree $-1$, meaning that it is an algebra filtration of $A$ together with 
 $\{A^{\leqslant i,\wteps},A^{\leqslant j,\wteps}\}\subseteq A^{\leqslant i+j-1,\wteps}$ for all $i,j\geqslant 0$.
 \item[(5)] The family $(A^{\leqslant i})_{i\in \N}$ is an algebra filtration of $R_{(\Delta,\Gamma)}$. Moreover, the associated graded algebra $\gr^{\wteps}(R)$ is equal to $\gr^{\wteps}(A)$.
 \item[(6)] For any $f,g \in A$ we have
 \[f\ast g-g\ast f = \Delta(f)\Gamma(g)-\Delta(g)\Gamma(f)+\sum_{k\geqslant 2}\left(\Delta^k(f)\binom{\Gamma}{k}(g)-\Delta^k(g)\binom{\Gamma}{k}(f)\right)\]
In particular, if $\wteps(f)=i$ and $\wteps(g)=j$ then $\wteps(f*g - g*f) \leqslant i + j - 1$.
\item[(7)] The commutative algebra $\gr^{\wteps}(R)=\gr^{\wteps}(A)$ can be endowed with the following three Poisson brackets
\begin{enumerate}
 \item[(a)] $\{\overline{f},\overline{g}\}':= (f*g-g*f) + A^{\leqslant i+j-2,\wteps}$
 \item[(b)] $\{\overline{f},\overline{g}\}'':= \{f,g\}+ A^{\leqslant i+j-2,\wteps}$
 \item[(c)] $\{\overline{f},\overline{g}\}''':= \overline{\Delta}(\overline{f})\overline{\Gamma}(\overline{g}) - \overline{\Gamma}(\overline{f})\overline{\Delta}(\overline{g}) \in A^{\leqslant i+j-1,\wteps}/A^{\leqslant i+j-2,\wteps}$
\end{enumerate}
for homogeneous elements $\overline{f}$ and $\overline{g}$ of respective ${\wteps}$-degree $i$ and $j$.
In (c) the pair of maps $(\overline{\Delta},\overline{\Gamma})$ is the solvable pair of homogeneous derivations of $\gr^{\wteps}(A)$ of respective degree $-1$ and $0$ which is induced by the solvable pair of filtered derivations $\Delta$ and $\Gamma$ of $A$ (see Proposition~\ref{prop-graduation-delta-gamma}).
\item[(8)] The Poisson structures defined in (3) are all equal and make $\gr^{\wteps}(A)=\gr^{\wteps}(R)$ into a graded Poisson algebra of degree $-1$.
\end{enumerate}
\end{proposition}

\begin{proof} The proof is very close from the one of Lemma~\ref{filtdg} and of Proposition~\ref{prop-linkRA} but since the construction of the filtration is more complicated the arguments should be given in a somewhat different order.
We start with assertion~(3). Proposition~\ref{prop-graduation-induite} shows that $\gr^{\wteps}(A)$ is a domain.
This shows~(3).
To prove~(1) and~(2), we start with $x \in A_1$.
By definition of $\wteps$ we have $\wteps(\Delta(x)) \leqslant \wteps(x) -1$.
Moreover, Lemma~\ref{lem-crochet} shows that $\ker \Delta^i \cap A_1$ is stable by $\Gamma$, hence by $P(\Gamma)$ showing that $\wteps(P(\Gamma)(x))\leqslant \wteps(x)$.
We write now $f \in A$ as a sum of products of elements of $A_1$ and~(1) and (2) follow from~(3) and the fact that $\Delta$ and $\Gamma$ are derivations. 
Assertion (4) follows from~\eqref{pbdg} and then assertions (1) and (2) and~(3).
Assertion (5) and (6) follow from~\eqref{ast} and then we get assertions (1) and (2) and~(3).

(7). The fact that $\{-,-\}'$ is a Poisson bracket on $\gr(R)$ follows from the filtered version of the semi-classical limit construction, see \cite[Section 2.4]{Goo}.
The fact that $\{-,-\}''$ is a well-defined biderivation satisfying the Jacobi identity on $\gr(A)$ follows by tedious but straightforward computation from the fact that $\{-,-\}$ is a filtered Poisson bracket on $A$.
Finally $\{-,-\}'''$ is a Poisson bracket since $(\overline{\Delta},\overline{\Gamma})$ is a solvable pair of derivations of $\gr^{\wteps}(A)$.

(8). The Poisson brackets $\{-,-\}'$ and $\{-,-\}''$ are the same since both $f\ast g-g\ast f$ and $\{f,g\}$ belong to $A^{\leqslant i+j-1,\wteps}$ combined with $f\ast g-g\ast f-\{f,g\}\in A^{\leqslant i+j-2,\wteps}$ thanks to Assertion (1), (2), (3) and (6).
Finally for homogeneous elements $\overline{f}$ and $\overline{g}$ of $\wteps$-degree $i$ and $j$ we have 
\begin{align*}
 \{\overline{f},\overline{g}\}'''
 & = \overline{\Delta}(\overline{f})\overline{\Gamma}(\overline{g}) - \overline{\Gamma}(\overline{f})\overline{\Delta}(\overline{g}) \\
 & = (\Delta(f)+A^{\leqslant i-2,\wteps})(\Gamma(g)+A^{\leqslant j-1,\wteps}) -  (\Delta(g)+A^{\leqslant j-2,\wteps})(\Gamma(f)+A^{\leqslant i-1,\wteps}) \\
 & = (\Delta(f)\Gamma(g)+A^{\leqslant i+j-2,\wteps}) -  (\Delta(g)\Gamma(f)+A^{\leqslant i+j-2,\wteps}) \\
 & = (\Delta(f)\Gamma(g) - \Delta(g)\Gamma(f))+A^{\leqslant i+j-2,\wteps}  = \{f,g\} + A^{\leqslant i+j-2,\wteps} = \{\overline{f},\overline{g}\}''.
 \end{align*}
Hence $\{-,-\}''$ and $\{-,-\}'''$ are the same. 
\end{proof}

\begin{remark}[Matrix version]\label{rem-matrixversion} 
Assertion~(8) says that the Poisson structure on $\gr^{\wteps}(A)$ is associated to the solvable pair $(\overline{\Delta},\overline{\Gamma})$. 
On $A_1$, the filtrations~$\wteps$ and $\varepsilon$ coincide, hence by an appropriate choice of basis of $A_1$ and by using the Jordan reduction theorem adapted to this filtration, we obtain the following block decompositions
$$\Delta=\begin{pmatrix}0& I_{m_0,m_1} && \\ &\ddots&\ddots& \\ &&\ddots& I_{m_{r-2},m_{r-1}} \\ &&&0 \end{pmatrix}
\qquad \textrm{and} \qquad \Gamma = \begin{pmatrix}\Gamma_1 & \Gamma_{1,2}&\cdots & \Gamma_{1,r}  \\ &\ddots&\ddots& \vdots \\ &&\ddots&\Gamma_{r-1,r} \\ 
&&&\Gamma_r \end{pmatrix}$$
where, for $p\geqslant r$, we set
$$I_{p,r}= \begin{pmatrix} 1 && \\[-2ex] &\ddots & \\[-2ex]   && 1 \\[-1ex] 0 & \cdots& 0 \\[-1ex] \vdots && \vdots \\[-1ex] 
0&\cdots& 0 \end{pmatrix} 
\in M_{p,r}(\kk)\,.$$
The corresponding $\overline{\Delta}$ and $\overline{\Gamma}$ are then given by 
$$\overline{\Delta}=\begin{pmatrix}0& I_{m_0,m_1} && \\ &\ddots&\ddots& \\ &&\ddots& I_{m_{r-2},m_{r-1}} \\ &&&0 \end{pmatrix}
\qquad \textrm{and} \qquad \overline{\Gamma} = \begin{pmatrix}\Gamma_1 & &   \\ &\ddots& \\ &&\Gamma_r \end{pmatrix}\,.$$
\end{remark}

\begin{example}
Let $(\Delta,\Gamma)$ be a linear solvable pair on $A=\kk[X_0,\ldots,X_n]$ and assume that  $\Delta$ is of rank $1$. 
Then $R$ (resp. $A$) is isomorphic to an Ore extension (resp. a Poisson-Ore extension) over a commutative (resp. Poisson commutative) polynomial ring in $n$ variables. 

Indeed, after a suitable change of basis of $A_1$, by Remark~\ref{rem-matrixversion} and the relation $[\Delta,\Gamma]=\Delta$ we have $\Delta=  X_0 \partial_{X_n}$ and 
$\Gamma=\alpha X_0 \partial_{X_0} + \sum_{i=1}^{n-1}P_i\partial_{X_i} + \big((\alpha+1)X_n+ P_n\big)\partial_{X_n}$ where $P_i \in \kk[X_0,\ldots,X_{n-1}]$ for $1 \leqslant i \leqslant n$.
In particular, $\Gamma$ induces a derivation of $\ker \Delta= \kk[X_0,\ldots,X_{n-1}]$.
For any $Q \in \ker \Delta$ we have $\{X_n,Q\}=X_0 \Gamma(Q)$.
Therefore $A_{(\Delta,\Gamma)}$ is isomorphic to the Poisson-Ore extension $(\ker\Delta)[X_n;X_0\Gamma]_P$.
Also, for any $Q \in \ker \Delta$, we have 
\[X_n \ast Q - Q \ast X_n= QX_n + X_0\Gamma(Q)- QX_n= X_0 \Gamma(Q)=X_0\ast\Gamma(Q).\]
For any $i \in \N$ we denote by ${X_n}^{\ast i}$ the $i^{\rm{th}}$ power of $X_n$ for the product $\ast$.
By an easy induction we obtain that for any $i \in \N$ there exist $P_{ij}\in \kk[X_0,\ldots,X_{n-1}]$ such that $X_n^{\ast i}={X_n}^i + \sum_{j=0}^{i-1} P_{ij} {X_n}^j$. Hence $({X_n}^{\ast i})_{i \in \N}$ is a basis of
of $A$ over $\ker\Delta=\kk[X_0,\ldots, X_{n-1}]$ and $R_{(\Delta,\Gamma)}$ is isomorphic to the Ore extension $(\ker\Delta)[X_n;X_0\Gamma]$.
\end{example}

\section{Trigonalizable linear solvable pair} \label{sec-trigo}

We continue with the notations of Section~\ref{sec-linear}: $A=\kk[X_0,\ldots,X_n]$ is a polynomial algebra with standard grading and, for all $k\geqslant 0$, we denote by $A_k$ the vector space of degree $k$ homogeneous polynomials. 

\begin{definition} A linear solvable pair $(\Delta,\Gamma)$ is said to be trigonalizable if the linear map induced by $\Gamma$ on $A_1$ is.
\end{definition}

\begin{proposition}\label{prop-trigo} Let $(\Delta,\Gamma)$ be a linear solvable pair. Then it is trigonalizable if and only if the linear map $\Gamma_1$ induced by $\Gamma$ on $A_1\cap \ker \Delta$ is. 
\end{proposition}

\begin{proof} Using notations of~Remark~\ref{rem-matrixversion}, the linear map induced by $\Gamma$ on $A_1\cap \ker \Delta$ is $\Gamma_1$.
Hence if $\Gamma$ is trigonalizable then $\Gamma_1$ is. 
Reciprocally, by computing the superdiagonal blocks of the relation $\Delta\Gamma -\Gamma \Delta = \Delta$, we have for every $1 \leqslant i\leqslant r-1$ the following triangular block decomposition
$$\Gamma_i = \begin{pmatrix} \Gamma_{i+1} - I& C_i \\ & B_i\end{pmatrix}$$
where $I$ denote the identity matrix of the appropriate size.
Hence $\chi_{\Gamma_{i+1}}(X+1)$ divides $ \chi_{\Gamma_i}(X)$, where $\chi_M$ denote the characteristic polynomial of the square matrix $M$.
We then deduce that when $\chi_{\Gamma_1}$ is trigonalizable, every $\chi_{\Gamma_i}$ is.
Hence $\Gamma$ acts on $A_1$ as a trigonalizable linear map.
\end{proof}

\subsection{The Artin-Schelter regular property}\label{ssec-AS-regul}

In this section we prove that the algebra $R=R_{(\Delta,\Gamma)}$ is Artin-Schelter regular when $(\Delta,\Gamma)$ is a linear solvable pair.
Artin-Schelter regular algebras \cite{AS} are thought to be noncommutative analogue of commutative polynomial rings in the following sense.

\begin{definition}
A connected $\N$-graded $\kk$-algebra $R$ is called {\em Artin-Schelter regular} or {\em AS-regular} if:
\begin{enumerate}
\item $ \mbox{gldim } R < \infty$;
\item $R$ has finite Gelfand-Kirillov dimension;
\item $\mbox{Ext}^i_R(\kk_R, R_R) \cong \begin{cases}  0 & \text{ if $i \neq \mbox{gldim }R$} \\
{}_R \kk[\ell] & \text{ if $i = \mbox{gldim } R$.}
\end{cases} $
\end{enumerate}
where $\kk[\ell]$ means that the module is degree-shifted by some amount $\ell \in \Z$.
\end{definition}
%Condition (3) above is called the {\em AS-Gorenstein} condition.

%Recall that $A=\kk[X_0,\dots,X_n]$ is a commutative polynomial ring in $n+1$ indeterminates.
%In particular $A$ is graded by the classical $d$-grading defined by $d(X_k)=1$ for all $0\leqslant k\leqslant n$.
%The corresponding homogeneous components are denoted by $A_k$ for any $k\geqslant 0$.
Since both $\Delta$ and $\Gamma$ are linear derivations we have $\Delta(A_k)\subseteq A_k$ and $\Gamma(A_k)\subseteq A_k$ 
for all integer $k\geqslant 0$.
Therefore $A_k\ast A_\ell\subseteq A_{k+\ell}$ for any $k,\ell\geqslant 0$ thanks to equation (\ref{ast}). 
Thus the algebra $R=(A,\ast)$ is $\N$-graded, generated in degree $1$ and with Hilbert series given by
\[\mbox{hilb}(R)=\mbox{hilb}(A)=\frac{1}{(1-t)^{n+1}}.\] 

The following result is the key argument for our inductive proof of Theorem \ref{theo-ASregul}.
Since this result relates algebras $R_{(\Delta,\Gamma)}$ (resp. $A_{(\Delta,\Gamma)}$) of different dimensions we will specify the (projective) dimension of the underlying polynomial ring by using the notation $R_{(\Delta,\Gamma)}^n$ (resp. $A_{(\Delta,\Gamma)}^n$). 

\begin{proposition} Let $(\Delta,\Gamma)$ be a linear solvable pair on $A=\kk[X_0,\ldots,X_n]$ such that $\Delta(X_0)=0$ and there exists $\alpha \in \kk$ verifying $\Gamma(X_0)=\alpha X_0$.
\label{quotient}
Set $A=A_{(\Delta,\Gamma)}^n$ and $R=R_{(\Delta,\Gamma)}^n$ and denote by $\Delta$ and $\Gamma$ the matrices of $\Delta$ and $\Gamma$ acting on $A_1$ with respect to a basis starting by $X_0$.
%We denote by $\Delta'$ and $\Gamma'$ the derivations of $\kk[X_0,\dots,X_{n-1}]$ obtained from $\Delta$ and $\Gamma$ by deleting from their matrix forms the first row and first column. 
%We denote by $\overline{\Delta}$ and $\overline{\Gamma}$ are the derivations of the quotient algebra $A/\langle X_0\rangle$ induced by $\Delta$ and $\Gamma$.
%Then
\begin{enumerate}
%\item[(1)] We have $X_0\ast R=R\ast X_0=X_0A$, $X_0$ is normal in $R$ and $X_0$ is Poisson normal in $A$.
\item[(1)] The quotient algebra $R/\langle X_0\rangle$ is isomorphic to $R_{(\Delta',\Gamma')}^{n-1}$ for the solvable pair $(\Delta',\Gamma')$ of $A^{n-1}=\kk[Y_0,\dots,Y_{n-1}]$ obtained from the solvable pair $(\Delta,\Gamma)$ by deleting the first rows and first columns.
\item[(2)] $X_0$ is Poisson normal in $A$ and the Poisson algebra $A/\langle X_0\rangle$ is isomorphic to $A_{(\Delta',\Gamma')}^{n-1}$ for the same solvable pair $(\Delta',\Gamma')$ of $A^{n-1}$.
\end{enumerate}
In particular, $R/\langle X_0\rangle$ is a deformation of the Poisson algebra $A/\langle X_0\rangle$.
%From the matrix point of view, the derivation $\Delta'$ (resp. $\Gamma'$) is obtained by deleting the first row and first column of $\Delta$ (resp. $\Gamma$).
\end{proposition}

\begin{proof}
%Moreover $X_0$ is Poisson normal in $A$ thanks to equation (\ref{pbdg}) since $\Delta(X_0)=0$ and $X_0$ is an eigenvector of $\Gamma$.

(1)  First note that $X_0$ is strongly normal in $R$ and Poisson normal in $A$. The algebras $R/\langle X_0\rangle$, $A/X_0A$ and $R^{n-1}_{(\Delta',\Gamma')}$ 
can all be identified as graded vector spaces.
Let $\pi:R\rightarrow R/\langle X_0\rangle$ be the quotient map and set $\overline{X}_i=\pi(X_{i})$ for all $0\leqslant i\leqslant n$.
Then the product in $R/\langle X_0\rangle$ is given by
\begin{align*}
\overline{X}_i\ast \overline{X}_j
%=\pi(X_{i+1}\ast X_{j+1})=\sum_{\ell\geqslant0}\pi\Delta^\ell(X_{i+1}) \pi\binom{\Gamma}{\ell}(X_{j+1})
=\sum_{\ell\geqslant0}\overline{\Delta}^\ell(\overline{X}_i) \binom{\overline{\Gamma}}{\ell}(\overline{X}_j)
\end{align*}
where $\overline{\Delta}$ and $\overline{\Gamma}$ denote the derivations induced by$\overline{\Delta}$ and $\overline{\Gamma}$ on the quotient $A/\langle X_0\rangle\cong\kk[Y_0,\dots,Y_{n-1}]$.
It is now a simple verification to check that $(\overline{\Delta},\overline{\Gamma})$ is a solvable pair on $A/\langle X_0\rangle$ that agrees with the solvable pair $(\Delta',\Gamma')$ of $\kk[Y_0,\dots,Y_{n-1}]$ under the identification $\overline{X}_{i+1}\mapsto Y_i$ for all $0\leqslant i<n$.

(2) $X_0$ is Poisson normal since it is in the kernel of $\Delta$ and is an eigenvector of $\Gamma$.
It is direct that the Poisson bracket on $A/X_0A$ is given by the solvable pair $(\overline{\Delta},\overline{\Gamma})$ and we conclude as in (1).
\end{proof}

We now prove that $R_{(\Delta,\Gamma)}$ is Artin-Schelter regular for any solvable pair $(\Delta,\Gamma)$.
The result follows from the trigonalizable case. 

\begin{theorem}\label{theo-ASregul} Let $(\Delta,\Gamma)$ be a linear solvable trigonalizable pair.
The algebra $R=R^n_{(\Delta,\Gamma)}$ is Artin-Schelter regular of global dimension $n+1$.
\end{theorem}

\begin{proof}
The proof follows by induction on $n$, exactly as in the proof of \cite[Theorem 3.8]{LS} since the only hypotheses used by the authors are that $R$ is connected graded generated in degree $1$ with polynomial growth and admits a normal sequence $\Omega$ of homogeneous regular elements with $R/\Omega R\cong\kk$.
Here we can also use the sequence $\Omega=(X_0,X_1,\dots,X_n)$, where $(X_0,\ldots, X_n)$ is a basis of $A_1$ as in Remark~\ref{rem-matrixversion} and $\Gamma$ is triangular (Proposition~\ref{prop-trigo}). 
%The initialisation follows from Example \ref{exn=1} and the induction step is based on Proposition \ref{quotient}.
\end{proof}

%The preceding theorem extends to the general case where the linear pair $(\Delta,\Gamma)$ is no more assumed to be trigonalizable.
%Artin-Schelter regularity is a general property of the algebra $R_{(\Delta,\Gamma)}$.

\begin{corollary}\label{Cor-ASregul}
Let $(\Delta,\Gamma)$ be a linear solvable pair.
The algebra $R=R^n_{(\Delta,\Gamma)}$ is Artin-Schelter regular of global dimension $n+1$.
\end{corollary}

\begin{proof} Consider $\widetilde{\kk}$ a finite field extension of $\kk$ such that the characteristic polynomial of $\Gamma$ acting on $A_1$ is split over $\widetilde{\kk}$.
Thanks to Theorem~\ref{theo-ASregul} the algebra 
$\widetilde{R}=\widetilde{\kk} \otimes R$ is Artin-Schelter regular since the unique extension of $(\Delta,\Gamma)$ to $\widetilde{R}$ is a trigonalizable solvable pair.

Since $\widetilde{\kk} \otimes \rule[0.3ex]{1ex}{0.5ex}$ is an exact functor which sends projective modules over $R$ to projective modules over $\widetilde{R}$, for any $R$-modules $M,N$, we have 
\[\textrm{Ext}^{\ast}_{\widetilde{R}}(\widetilde{\kk} \otimes M, \widetilde{\kk} \otimes N) \cong \textrm{Ext}^{\ast}_{R}(M, \widetilde{\kk} \otimes N) \cong 
\textrm{Ext}^{\ast}_{R}(M, N)^{[\widetilde{\kk}:\kk]}
\]
Therefore we have $\mbox{gldim } R \leqslant \mbox{gldim } \widetilde{R}< \infty$.
Since $R$ and $\widetilde{R}$ are $\N$-graded connected algebras, the case $M=N=k$ and~\cite[Theorem 11, Theorem 13 and Proposition 15]{eilenberg} show that $\mbox{gldim } R =\mbox{gldim } \widetilde{R}$.
By considering the case $M=k$ and $N=R$, we obtain that 
$\textrm{Ext}^{\ast}_{\widetilde{R}}(\widetilde{\kk}, \widetilde{R})  \cong \textrm{Ext}^{\ast}_{R}(k, R)^{[\widetilde{\kk}:\kk]}$.
In particular, if $\textrm{Ext}^{\ast}_{\widetilde{R}}(\widetilde{\kk}, \widetilde{R}) \cong \widetilde{k}$ then $\textrm{Ext}^{\ast}_{R}(k, R) \cong k$ since the field extension is finite.
It follows that the algebra $R$ is Artin-Schelter regular. 
\end{proof}

\subsection{Normal elements}

The next proposition completely describes normal elements of $R$ and Poisson normal elements in $A$ in the case of a linear solvable pair: they are the strongly normal element (see Definition~\ref{dfn-strongly-ne}).
This generalizes the results of Lemma~\ref{lem-Pcenter}, Proposition~\ref{prop-Pcenter-lin} and is a converse of Lemma~\ref{lem-strongly} in the linear case.

\begin{proposition}\label{prop-normal-element} 
Let $(\Delta,\Gamma)$ be linear solvable pair.
Assume that $A$ is not Poisson commutative (see Corollary~\ref{cor-poissoncom}).
Then% 
%Assume that $\Delta$ has either 
%\begin{itemize}
%\item[(a)] at least one block of size at least $3$, or
%\item[(b)] no block of size more than $2$ and at least two blocks of size $2$, or
%\item[(c)] no block of size more than $2$, a single block of size $2$ and that $a_1\neq0$.
%\end{itemize}
\begin{enumerate}
\item $N$ is normal in $R$ iff $N$ is strongly normal iff $N$ is Poisson normal in $A$. 
\item $N$ is central in $R$ iff $N\in\ker\Delta\cap\ker\Gamma$ iff $N$ is Poisson central in $A$.
\end{enumerate}
\end{proposition}

\begin{proof}
The proof splits in two cases according to the value of the rank of $\Delta$ acting on $A_1$.
When this rank is equal to $1$, the proof relies on an explicit computation.
When this rank is strictly greater than one, the proof can be mimicked from the one of \cite[Proposition 3.21]{LS} by noticing that this proof only relies on the existence of an $\wteps$-homogeneous, irreducible, strongly normal element $G$ which is not in $\ker\Gamma$. %and such that $G\ast R=R\ast G$ is a completely prime ideal of $R$. 
% of an irreducible element $G\in A$ that is homogeneous for $\wteps$ (see Section~\ref{ssec-AS-regul} for the definition of $\wteps$) and such that $\Delta(G)=0$, $\Gamma(G)=\lambda G$ for a scalar $\lambda\in\kk^\times$ and that $G\ast R=R\ast G$ is a completely prime ideal of $R$.
If this is the case, it can verify that the ideal $G\ast R=R\ast G$ of $R$ is completely prime since $G$ is an irreducible element of the unique factorization domain $A$.
%Note that thanks to appendix \ref{cpideal}, the ideal $G\ast R=R\ast G$ is always a completely prime ideal of $R$ since $G$ is an irreducible element of the unique factorization domain $A$.
Note that the existence of such a $G$ implies that assertion $(2)$ can also be seen as a consequence of Lemma~\ref{lem-Pcenter}.

First assume that $(\Delta,\Gamma)$ is a trigonalizable solvable pair.
In the following, we only consider actions of $(\Delta,\Gamma)$ on $A_1$ and its subspaces. 
Using the notations of Remark~\ref{rem-matrixversion} and Proposition~\ref{prop-trigo}, we see that $\Gamma_1$ is trigonalizable.
If $\Gamma_1$ has a nonzero eigenvalue $\lambda$, then an eigenvector $G$ of $\Gamma_1$ with respect to $\lambda$ satisfies the desired properties.
Hence we can assume that $\Gamma_1$ is nilpotent.
As in the proof of Proposition~\ref{prop-trigo}, the relation $\Delta\Gamma -\Gamma \Delta = \Delta$ shows that for every $1 \leqslant i\leqslant r-1$ we have the following triangular block decomposition
$$\Gamma_i = \begin{pmatrix} \Gamma_{i+1} - I& C_i \\ & B_i\end{pmatrix}$$
where $I$ denote the identity matrix of the appropriate size.
Hence $i-1$ is the only eigenvalue of $\Gamma_i$ and thus $\Gamma_i - (i-1)\id$ is nilpotent.
The triangular shape of the block decomposition of $\Gamma$ then shows that $\ker \Delta^{i} = \bigoplus_{j=0}^{i-1} N_j$ where $N_j$ is the generalized eigenspace of $\Gamma$ 
with respect to $j$.
The relation $\Delta \Gamma -\Gamma \Delta=\Delta$ can be rewritten $\Delta (\Gamma-j\id) = (\Gamma-(j - 1)\id)\Delta$ and then
$\Delta (\Gamma-j\id)^k = (\Gamma-(j - 1)\id)^k\Delta$ for all integer $k\geqslant0$.
Hence $\Delta$ maps $N_j$ into $N_{j-1}$.
Moreover the relation $\ker \Delta^{i} = \bigoplus_{j=0}^{i-1} N_j$ shows that for all integer $j \geqslant 1$, $\Delta$ maps injectively $N_j$ into $N_{j-1}$.

Assume now that $\Delta$ admits a Jordan block of size at least $3$.
Then $2$ is an eigenvalue of $\Gamma$.
If we consider $X_2$ such that $\Gamma(X_2)=2X_2$, then $X_1=\Delta(X_2) \in 
\ker(\Gamma-\id)$, $X_0=\Delta(X_1) \in \ker(\Gamma)$, the family $(X_0,X_1,X_2)$ is linearly independent and then $G=2X_2X_0-{X_1}^2 \neq 0$ verifies $\Delta(G)=0$ and $\Gamma(G)=2G$ as desired. 

Hence we can now assume that $\Delta$ admits only Jordan blocks of size $2$.
If $\Delta$ admits at least two such blocks then $\dim(\ker \Delta^2/\ker \Delta)\geqslant 2$.
Hence $\dim N_1 \geqslant 2$. 
If $\dim \ker(\Gamma-\id) \geqslant 2$ then choose $(X_1,Y_1) \in \ker(\Gamma-\id)$ linearly independent then $G=X_1\Delta(Y_1)-Y_1\Delta(X_1) \neq 0$ verifies $\Delta(G)=0$ and $\Gamma(G)=G$ as wanted. 
For $\Delta(Y_1),\Delta(X_1) \in \ker \Delta\cap \ker \Gamma$.
If $\dim \ker(\Gamma-\id) = 1$, since $\dim N_1\geqslant 2$, there exists $(X_1,Y_1)\in {N_1}^2$ 
such that $\Gamma(X_1)=X_1$ and $\Gamma(Y_1)=Y_1+X_1$ (apply canonical Jordan form to $\Gamma$ acting on $N_1$).
Set then $X_0=\Delta(X_1)$ and $Y_0=\Delta(Y_1)$.
We have $X_0,Y_0 \in \ker\Delta$ and $\Gamma(X_0)= 0$ and $\Gamma(Y_0)=X_0$.
Hence $G=X_1Y_0 - X_0Y_1 \neq 0$ verifies $\Gamma(G)=G$ and $\Delta(G)=0$.

It remains to consider the case where $\Delta$ admits only one Jordan block which is of size $2$ that is to say the rank of $\Delta$ is $1$.
In this case, we have
$$\Delta=\begin{pmatrix} 0 & 0& 1 \\ & 0&0 \\ && 0\end{pmatrix} \qquad \textrm{and} \qquad \Gamma=\begin{pmatrix} 0 & L_1 & 0 \\ &\widetilde{\Gamma_1}& 0 \\ && 1 \end{pmatrix} $$
where the second column of the matrices $\Delta$ and $\Gamma$ represents matrices with $(n-1)$ columns.
By choosing the corresponding basis $(X_0,\ldots,X_n)$ for $A_1$, we have $\Delta= X_0 \partial_{X_n}$ and $\Gamma(X_0)=0$, $\Gamma(X_j)\in \kk[X_0,\ldots,X_{n-1}]$ for $j \in \{1,\ldots,n-1\}$ and $\Gamma(X_n)=X_n$.
We will show that every Poisson normal element $Q$ lies in $\ker\Delta \cap \ker \Gamma$ and that every normal element $Q$ in $R$ also lies in $\ker\Delta \cap \ker \Gamma$.
For, write $Q=\sum_{i=0}^{s}P_i {X_n}^i$ with $P_i \in \kk[X_0,\ldots,X_{n-1}]$ and an integer $s\geqslant0$.
We have 
$$\Delta(Q)= \sum_{i=0}^{s}iP_iX_0 {X_n}^{i-1} \qquad \textrm{and} \qquad\Gamma(Q)= \sum_{i=0}^{s}(\Gamma(P_i) + iP_i) {X_n}^{i}\,.$$
Moreover there exists $j_0\in \{1,\ldots,n-1\}$ such that $\Gamma(X_{j_0})\neq 0$, otherwise, from Corollary~\ref{cor-poissoncom}, we have that $A$ is Poisson commutative. 
Assume that $Q$ is Poisson normal and compute $\{Q,X_{j_0}\}=\Delta(Q)\Gamma(X_{j_0})$.
Since $Q$ is Poisson normal there exists $F \in A$ such that $\Delta(Q)\Gamma(X_{j_0})=FQ$.
If $F \neq 0$ then by comparing the degree in $X_n$ of the two sides of the equality leads to something impossible.
Hence $F=0$ and $\Delta(Q)=0$ since $\Gamma(X_{j_0}) \neq 0$.
So $Q=P_0\in k[X_0,\ldots,X_{n-1}]$.
By computing $\{Q,X_n\}=-\Gamma(Q)X_0=-\Gamma(P_0)X_0$, from the Poisson normality of $Q$, we obtain that there exists $F \in A$ such that $\Gamma(P_0)X_0=FP_0$.
We then write $P_0 = {X_0}^\ell R$ with $X_0 \nmid R$. Since $X_0 \in \ker \Gamma$, we obtain the relation $\Gamma(R)X_0 = FR$.
Hence $X_0 \mid F$, so $\Gamma(R)=F_1R$ for some $F_1 \in A$.
By comparing the total degree we obtain that $F_1 \in \kk$, and by multiplying by ${X_0}^\ell$ we obtain that
$Q=P_0$ verifies $\Gamma(P_0)=F_1 P_0$.
But $\Gamma$ is locally nilpotent on $\kk[X_0,\ldots, X_{n-1}]$ since it acts as a nilpotent linear map on the span of $X_0,\ldots,X_{n-1}$, hence $F_1=0$ and $Q=P_0 \in \ker \Gamma \cap \ker \Delta$ is strongly normal. 

Let us now consider that $Q \neq 0$ is normal in $R$ and compute $X_{j_0}\ast Q = X_{j_0}Q$.
Since $Q$ is normal, there exists $f \in R$ such that $X_{j_0}\ast Q=Q\ast f$. 
By applying the $\varepsilon$-degree, we deduce that $\varepsilon(f)=\varepsilon(X_{j_0})=0$, that is to say $f \in \ker \Delta$.
By sending the preceding equality in $\gr(R)$ which is a domain thanks to Corollary~\ref{cor-domain}, 
we obtain that $\ol{X_{j_0}}=\ol{f} \in \ker(A^{\leqslant 0}/\{0\})$ and then $X_{j_0}=f$. 
The relation $X_{j_0}\ast Q=Q\ast f$ can then be rewritten as 
\[\sum_{i\geqslant 1} \Delta^i(Q) \binom{\Gamma}{i}(X_{j_0})=0 \qquad \textrm{with} \qquad \binom{\Gamma}{i}(X_{j_0})\in \ker \Delta\]
By considering the smallest integer $N\geqslant 1$ such that $\Delta^N(Q)=0$ and by composing the preceding relation with $\Delta^{N-2}$, we obtain that $\Delta^{N-1}(Q)\Gamma(X_{j_0})=0$.
But $\Gamma(X_{j_0})\neq 0$, hence $N=1$ and $Q \in \ker \Delta=k[X_0,\ldots,X_{n-1}]$.
We now compute $X_n \ast Q = X_nQ + X_0 \Gamma(Q)$. Since $Q$ is normal, there exists $f$ such that $X_n\ast Q = Q\ast f=Qf$.
Hence $Qf=X_nQ + X_0 \Gamma(Q)$ and we deduce that the total degree of $f$ is not greater than $1$ and that $\deg_{X_n}(f)=1$.
By writing $f=a X_n + f_1$ with $f_1 \in \kk[X_0,\ldots,X_{n-1}]$ and $a \in \kk$, we obtain $a=1$ and then $Qf_1= X_0 \Gamma(Q)$.
By applying the same argument that the one in the Poisson normal case, we obtain that $\Gamma(Q)=0$.
Therefore $Q \in \ker \Gamma \cap \ker \Delta$ is strongly normal.

To conclude, it remains to consider the case where the pair $(\Delta,\Gamma)$ is not trigonalizable.
Consider an element $Q$ that is either Poisson normal in $A$ or normal in $R$. 
Let $\ol{\kk}$ be the algebraic closure of $\kk$ and consider the scalar extensions $\ol{A}=\ol{\kk}\otimes A$ and $\ol{R}=\ol{\kk}\otimes R$ of $A$ and $R$ respectively.
The pair $(\Delta,\Gamma)$ is a trigonalizable linear pair on $\ol{A}$ and $1 \otimes Q$ is Poisson normal in $\ol{A}$ or normal in $\ol{R}$.
Hence $1 \otimes Q$ is strongly normal.
Hence $1 \otimes \Delta(Q)=0$ and $1 \otimes \Gamma(Q)=\lambda \otimes Q$ for some $\lambda \in \ol{\kk}$.
But $1 \otimes \Gamma(Q)\in 1 \otimes A$ so $\lambda \in \kk$ and $\Gamma(Q)=\lambda Q$.
So $Q$ is strongly normal.
%First if $\Delta$ has at least one block of size $3$ (in particular the first block) then we take $G=X_0$ when $a_1\neq0$ or $G=X_0X_2-\frac{1}{2}X_1^2$ when $a_1=0$.
%Second if $\Delta$ has at least two blocks of size $2$ (in particular the first two blocks), we denote by $X_i$ the variables associated to the first block and $Y_i$ the variables associated to the second block, and we take $G=X_0$ when $a_1\neq0$, $G=Y_0$ when $a_2\neq0$ or $G=X_0Y_1-X_1Y_0$ when $a_1=a_2=0$ (this last element has $\Gamma$-weight equals to $a_1+a_2+1=1$).
%\textcolor{red}{Finally in case of $\mbox{Rk}(\Delta)=1$, there exists $j \neq 1$ such that $\Delta(X_j)=0$ and $\Gamma(X_j)=a_jX_j$ with $a_j\neq0$ (otherwise $A$ would be Poisson commutative, Corollary~\ref{cor-poissoncom}) and choose $G=X_j$.}
\end{proof}

\begin{remark}\label{rem-trigo-strongne}
Let $(\Delta,\Gamma)$ be a linear trigonalizable solvable pair.
According to the previous proof we have the following disjoint three cases.
\begin{itemize}
\item The rank of $\Delta$ is $1$ and the restriction of $\Gamma$ to $\ker \Delta$ is zero. In this case, $A$ is Poisson commutative and $R$ is commutative.
\item The rank of $\Delta$ is $1$ and the restriction of $\Gamma$ to $\ker \Delta$ is nonzero and nilpotent.
In this case, every normal element in $R$ is central and every Poisson normal element in $A$ is Poisson central.
Moreover $A$ is not Poisson commutative and $R$ is not commutative. 
\item Otherwise, there exists a strongly normal element with a nonzero $\Gamma$-eigenvalue, that is to say, there exists an element which is Poisson normal but not Poisson central in $A$, as well as normal but not central in $R$.
\end{itemize}

Moreover, in the third case, we can assume that this normal and non central element belongs to the image $\im \Delta$ of $\Delta$.
Indeed, by the relation $\Gamma \Delta = \Delta (\Gamma - \id)$ we obtain that $\im \Delta$ is stable by $\Gamma$.
Hence $\ker \Delta \cap \im\Delta$ is stable by $\Gamma$.
Since $\Gamma$ is trigonalizable its restriction to $\ker \Delta \cap \im\Delta$ is too. 
In particular, this restriction has an eigenvector.
Hence there exists $x \in \ker \Delta \cap \im\Delta$ and $\lambda \in \kk$ such that $\Gamma(x)=\lambda x$.
If $\lambda \neq 0$, we are done. 
If $\lambda = 0$, consider $y$ such that $\Delta(y)=x$.
Let $z$ be a strongly normal element whose eigenvalue for $\Gamma$, denoted by $\mu$, is nonzero ($z$ may not be in $A_1$). 
Then $xz \in \ker \Delta$ and $xz$ is an eigenvector for $\Gamma$ with eigenvalue $\mu+\lambda= \mu \neq 0$.
Finally we have $\Delta(yz)=z\Delta(y)=zx \in \im \Delta$ as desired. 
\end{remark}

\subsection{On the isomorphism problem}

In this section we show that for a linear solvable pair $(\Delta,\Gamma)$, the Jordan type of $\Delta$ is determined by the isomorphism class of $R$ or the Poisson isomorphism class of $A$.
The Jordan type of a linear derivation of a polynomial algebra $A$ means the Jordan type of its restriction to $A_1$.

\begin{theorem}\label{thm-isom} Let $(\Delta,\Gamma)$ and $(\Delta',\Gamma')$ be linear solvable pairs on $\kk[X_0,\ldots,X_n]$ with $\Delta \neq 0$ and $\Delta' \neq 0$. 
We set $A=A_{(\Delta,\Gamma)}$, $A'=A_{(\Delta',\Gamma')}$, $R=R_{(\Delta,\Gamma)}$ and $R'=R_{(\Delta',\Gamma')}$.

\begin{itemize}
\item If $A$ and $A'$ are Poisson isomorphic, then $\Delta$ and $\Delta'$ have the same Jordan type.
\item If $R$ and $R'$ are isomorphic, then $\Delta$ and $\Delta'$ have the same Jordan type. 
\end{itemize}
\end{theorem}

\begin{lemma}
\label{relDEltaDelta'}
Assume that there exists a strongly normal element in $A$ with nonzero $\Gamma$-eigenvalue and that $A$ is not Poisson commutative.
\begin{enumerate}
 \item If $\phi : A \rightarrow A'$ is a Poisson isomorphism, then there exists $\al \in \kk^\ast$ such that $\phi \Delta \phi^{-1} = \alpha \Delta'$.
 \item If $\phi : R \rightarrow R'$ is an isomorphism, then $\phi(\ker \Delta^n) = \ker(\phi \Delta^n \phi^{-1})= \ker \Delta'^n$ for every $n \in \N$.
\end{enumerate}
\end{lemma}

\begin{proof}
We denote by $x$ the strongly normal element of $A$ whose $\Gamma$-eigenvalue $\lambda$ is nonzero.

(1) By Lemma~\ref{lem-strongly} the element $x$ is Poisson normal in $A$ and the associated Poisson derivation is $\lambda \Delta$. 
Since $\phi$ is a Poisson isomorphism, $\phi(x)$ is Poisson normal in $A'$, hence strongly normal by Proposition~\ref{prop-normal-element}.
Therefore there exists $\mu \in \kk$ such that $\Gamma'(\phi(x))=\mu \phi(x)$ and the Poisson derivation of $A'$ associated to $\phi(x)$ is $\mu \Delta'$. 
Because this derivation is also equal to $\phi (\lambda \Delta) \phi^{-1}$, we obtain the result since $\mu \neq 0$.

(2) Thanks to Lemma~\ref{lem-strongly}, since $x$ is strongly normal it is also normal in $R$ and the automorphism of $R$ associated to $x$ is $\phi_\lambda$.
We then deduce that $\phi(x)$ is normal in $R'$, and from Proposition~\ref{prop-normal-element}, that it is strongly normal in $R'$. 
Hence there exists $\mu \in \kk$ such that $\Gamma'(\phi(x))=\mu \phi(x)$.
Since the automorphism of $R'$ associated to $\phi(x)$ is $\phi \phi_\lambda \phi^{-1}$, we deduce that $\phi \phi_\lambda \phi^{-1}=\phi_\mu$ and hence $\mu \neq 0$.
By the last part of Lemma~\ref{compo} we have $\ker \Delta'^n = \ker (\phi_\mu - \id)^n = \ker (\phi (\phi_\lambda -\id)^n \phi^{-1})=\phi(\ker (\Delta^n))$ for every $n \in \N$.
\end{proof}

\begin{proof}[Proof of the theorem.]
If $R$ is commutative (resp. $A$ is Poisson commutative), then so is $R'$ (resp. $A'$).
By Corollary~\ref{cor-poissoncom} both $\Delta$ and $\Delta'$ are of rank $1$, hence are of the same Jordan type.
Assume now that $R$ is not commutative (resp. $A$ is not Poisson commutative).

If every normal element in $R$ (resp. Poisson normal element in $A$) is central (resp. Poisson central), then the same is true for $R'$ (resp $A'$).
Thanks to Remark~\ref{rem-trigo-strongne} the rank of both $\Delta$ and $\Delta'$ is $1$.
Hence $\Delta$ and $\Delta'$ have the same Jordan type.

Otherwise there exists a non central normal element in $R$ (resp. non Poisson central Poisson normal element in $A$).
Let $\phi : R \rightarrow R'$ be a $\kk$-algebra isomorphism (resp. $\phi : A \rightarrow A'$ be a Poisson isomorphism).
By \cite[Theorem 1]{BeZ} (resp. \cite[Proposition 8.8]{LGPV}) we may assume that $\phi$ is graded. 
By extending the scalars to an algebraic closure of $\kk$, we can furthermore assume that $(\Delta,\Gamma)$ and $(\Delta',\Gamma')$ are trigonalizable solvable pairs. 

We start by the Poisson case.
By assertion (1) of Lemma \ref{relDEltaDelta'} we have $\phi \Delta \phi^{-1} = \alpha \Delta'$ with $\alpha \in \kk^\ast$.
The derivations $\alpha \Delta'$ and $\Delta'$ have the same Jordan type.
and since $\phi$ is graded, the derivation $\phi \Delta \phi^{-1}$ and $\Delta$ have the same Jordan type.
The result follows.
 
In the case of $R$, thanks to assertion (2) of Lemma \ref{relDEltaDelta'}, the isomorphism $\phi$ preserves the filtration $\varepsilon$.
Since moreover $\phi$ is graded it must also preserve the filtration~$\wteps$.
Hence $\phi$ induces a graded Poisson isomorphism between~$\gr^{\wteps}(R) = \gr^{\wteps}(A)$ and $\gr^{\wteps}(R')= \gr^{\wteps}(A')$.
From~Remark~\ref{rem-matrixversion}, the Jordan type of $\Delta$ (resp. $\Delta'$) is the same that the one of $\overline{\Delta}$ (resp. $\overline{\Delta'}$).
%From~Remark~\ref{rem-matrixversion}, the Jordan type of $\Delta$ (resp. $\Delta'$) acting on $A_1$ (resp. $A'_1$) is the same that the one of $\overline{\Delta}$ (resp. $\overline{\Delta'}$) acting on $\gr^{\wteps}(A_1)$ (resp. $\gr^{\wteps}(A'_1)$).
Then the Poisson isomorphism case allows us to conclude.
\end{proof}

\section{Diagonalizable linear solvable pair: Poisson derivations} \label{sec-diago}

In this section we show that under a generic condition on $\Gamma$ the derivation $\Delta$ may be recovered from the Poisson algebra $A$ as, up to a nonzero scalar, the only locally nilpotent derivation in the center of ${\rm P.Der}_{\rm gr}(A)$.   

\begin{definition} A linear solvable pair $(\Delta,\Gamma)$ is said to be diagonalizable if there exists a basis of $A$ consisting of eigenvectors for $\Gamma$.
\end{definition}

\begin{lemma}
\label{diagA1daigA}
The existence of a basis of $A$ consisting of eigenvectors for $\Gamma$ is equivalent to the fact that the restriction of $\Gamma$ to $A_1$ is diagonalizable. 
\end{lemma}

\begin{proof}
 Indeed, if we consider a basis of eigenvectors of $\Gamma$ acting on $A_1$ then the set of monomials formed with elements of this basis of $A_1$ is a basis of $A$ consisting of eigenvectors for $\Gamma$. 

Reciprocally, assume that $\mathcal{B}$ is a basis  of $A$ consisting of eigenvectors for $\Gamma$. 
Let us show that the eigenvectors of $\Gamma$ contained in $A_1$ generates $A_1$. 
Consider $v \in A_1$ and express it as a linear combination of elements of $\mathcal{B}$.
Gather the eigenvectors associated to the same eigenvalue to write $v=w_1+ \cdots + w_r$ where $w_i$ is an eigenvector for $\Gamma$ associated to the eigenvalue $\lambda_i\in\kk$ and where $\lambda_i \neq \lambda_j$ when $i \neq j$.
By considering the family $(v,\Gamma(v),\ldots, \Gamma^{r-1}(v))\in {A_1}^r$ and 
inverting the Vandermonde matrix associated to $(\lambda_1,\ldots,\lambda_r)$, we get that $w_i \in A_1$ for all $i$. 
This conclude the proof.  
\end{proof}

For a diagonalizable linear solvable pair $(\Delta,\Gamma)$ we denote by $n_1 \geqslant n_2 \geqslant \cdots  \geqslant n_r$ the size of the Jordan blocks of $\Delta$ acting on $A_1$. 
Such a $\Delta$ is said to be of Jordan type $\overline{n}=(n_1,\ldots, n_r)$ where we set $n+1=n_1+ \cdots + n_r$.

\begin{proposition}\label{prop-diag-pair} Let $(\Delta,\Gamma)$ be a diagonalizable linear solvable pair. Assume that $\Delta$ is of Jordan type $\overline{n}=(n_1,\ldots, n_r)$.
Then there exists a basis of $A_1$ and $\overline{a}=(a_1,\ldots,a_r) \in \kk^r$ such that $\Delta$ is in canonical Jordan form and $\Gamma$ is diagonal of the form
$\Gamma =\diag(a_1,\ldots, a_1+n_1 -1, a_2,\ldots, a_2+ n_2 -1,\ldots, a_r, \ldots, a_r+n_r-1)$.
\end{proposition}

\begin{proof} This is an adaptation of the proof of the existence of the canonical Jordan form. 
For any integer $i\geqslant0$ we get that $F_i:=\ker \Delta^i$ is stable by $\Gamma$ by using the relation $[\Delta^i,\Gamma]=i\Delta^{i}$.
There exists $s\geqslant0$ such that $F_s=A_1$ since $\Delta$ is nilpotent on $A_1$.
We construct by a decreasing induction subspaces $(G_i)_{1\leqslant i \leqslant s}$ such that 
\begin{enumerate}
\item $F_{i-1}\oplus G_{i}=F_i$;
\item $\Delta$ maps injectively $G_{i+1}$ into $G_{i}$; 
\item $G_i$ is stable by $\Gamma$.
\end{enumerate}

Since $F_{s-1}$ is stable by $\Gamma$ and $\Gamma$ is diagonalizable, there exists $G_{s}$ which is stable by $\Gamma$
such that $F_s=A_1 = F_{s-1} \oplus G_{s}$. Assume that $G_{s},\ldots, G_{s+1}$ have been constructed satisfying~(1), (2) and (3). 
Then clearly $\Delta$ maps $G_{i+1}$ into $F_{i}$. Moreover $\Delta(G_{i+1}) \cap F_{i-1} = \{0\}$ (since $G_{i+1}\cap F_i = \{0\}$)
and $G_{i+1} \cap \ker \Delta =\{0\}$. Moreover $\Delta(G_{i+1})$ is stable by $\Gamma$. Hence $F_{i-1} \oplus \Delta(G_{i+1})$
too and there exists a subspace $W_i$ of $F_i$ stable by $\Gamma$ (since $\Gamma$ is diagonalizable) 
such that $F_{i-1} \oplus \Delta(G_{i+1}) \oplus W_i =F_i$. It suffices now to consider $G_i=\Delta(G_{i+1}) \oplus W_i$ 
to get the result. 

Remark that if $x$ is an eigenvector for $\Gamma$ (associated to $\lambda$) 
and $\Delta(x) \neq 0$ then $\Delta(x)$ is an eigenvector for $\Gamma$ (associated to $\lambda -1$).
Thus if we consider a basis of $G_s$ composed of eigenvectors of $\Gamma$, then we consider the image of this basis by $\Delta$ and complete into a basis of $G_{s-1}$ (adding eigenvector of $\Gamma$ belonging to $W_{s-1}$). 
We proceed by induction to produce the desired basis.  
\end{proof}

\begin{definition}
\label{Annaa}
Thanks to Lemma \ref{iso} and Proposition \ref{prop-diag-pair} one can assume without loss of generalities that for any diagonalizable solvable pair $(\Delta,\Gamma)$ the Poisson algebra $A_{(\Delta,\Gamma)}$ is presented in the basis, denoted by $(X_0,\ldots,X_n)$, given by Proposition \ref{prop-diag-pair}. 
We denoted by $A(\overline{n},\overline{a})$ this Poisson algebra $A_{(\Delta,\Gamma)}$.
\end{definition}

\begin{remark}\label{rem-pair-diago}
It will sometimes be convenient to denote by $\lambda_{\ell}$ the eigenvalue for $\Gamma$ of $X_{\ell}$ in $A(\overline{n},\overline{a})$.
The Poisson bracket of $A(\overline{n},\overline{a})$ is then given by
\[\{X_i,X_j\}=\lambda_j\Delta(X_i)X_j-\lambda_i\Delta(X_j)X_i\]
for all $0\leqslant i,j\leqslant n$.
%Moreover Lemma~\ref{dmax} shows that when $(\Delta,\Gamma)$ is a linear solvable pair with $\Delta$ maximal, then $\Gamma$ should be diagonalizable and $(\Delta,\Gamma)$ is a diagonalizable linear solvable pair.
\end{remark}

\begin{remark}\label{rem-pair-diago-bis}
Note that when $(\Delta,\Gamma)$ is a linear solvable pair with $\Delta$ maximal, then $(\Delta,\Gamma)$ is automatically  a diagonalizable linear solvable pair by Theorem~\ref{dmax}.
\end{remark}

\subsection{Poisson derivations for diagonalizable linear solvable pairs}

In this section we study Poisson derivations for diagonalizable linear solvable pairs. 
We have seen in Lemma~\ref{lem-pder} that $\Delta$ is always a Poisson derivation.
In the linear case, the so-called \textit{Eulerian derivation} $\mathcal{E}$ sending $X_i$ to itself for all $0\leqslant i\leqslant n$ is also a Poisson derivation.
Our aim is to recover $\Delta$, or more precisely its conjugacy class, from the algebraic structure of $A$.
We first study the case of the maximal Jordan block since the proof of the general will rely on it.

\subsubsection{Poisson derivations for maximal Jordan block}

In this section we may assume that $A=A(n,a)$ (see Example \ref{LSnot}) thanks to Section \ref{caseDeltamax}.
We show that there are no other linear Poisson derivations apart from those in the $\kk$-linear span of $\Delta$ and $\mathcal{E}$. 

\begin{lemma}
\label{xo}
Assume that $n \geqslant2$ or that $n=1$ and $a\neq0$.
For any $\delta\in{\rm P.Der}_{\rm gr}(A)$ there exists a scalar $\lambda$ such that $\delta(X_0)=\lambda X_0\in\langle X_0\rangle$.
\end{lemma}

\begin{proof}
First if $a=0$ (so $n \geqslant 2$) then $X_0$ is, up to scalar multiple, the unique Poisson central element in $A_1$ thanks to \cite[Proposition 3.1(2)]{LS}. 
The result follows since $\delta(X_0)$ is also a Poisson central element in $A_1$, see Remark~\ref{rem-der-center}.

Assume now $a\neq0$.
Using the fact that $\{X_0,X_n\}=-aX_0X_{n-1}$ and comparing for $0<i\leqslant n$ the coefficients of $X_iX_{n-1}$ in the relation $\delta(\{X_0,X_n\}) = \{\delta(X_0),X_n\}+\{X_0,\delta(X_n)\}$, we obtain that $\delta(X_0)=\lambda X_0$ for some $\lambda \in \kk$.
\end{proof}

\begin{theorem}
\label{poisdermax} Assume that $n \geqslant2$ or that $n=1$ and $a\neq0$.
Then ${\rm P.Der}_{\rm gr}(A)=\kk \mathcal{E}\oplus \kk\Delta$.
%In particular $\Delta$ is, up to nonzero scalar multiples, the only nonzero linear locally nilpotent Poisson derivation of $A$.
In particular the set of locally nilpotent derivations of ${\rm P.Der}_{\rm gr}(A)$ is the one dimensional subspace spanned by $\Delta$.
\end{theorem}

\begin{proof} We proceed by induction on $n \geqslant1$. 
First assume that $a\neq-(n-1)$.
Let $n=1$ (so that $a\neq0$) and $\delta\in{\rm P.Der}_{\rm gr}(A)$.
By Lemma \ref{xo} we know that $\delta(X_0)=\lambda X_0$ for a scalar $\lambda$. 
If $\delta(X_1)=\mu X_0+\nu X_1$ for some scalars $\mu,\nu\in\kk$, we have 
$\delta(\{X_0,X_1\})=-2a\lambda X_0^2$ and $\{\delta(X_0),X_1\}+\{X_0,\delta(X_1)\}=-2a(\lambda+\nu)X_0^2$ so that $\lambda=\nu$. 
The result is proved since we have $\delta=\lambda\mathcal{E}+\mu\Delta$.
%&\delta(\{X_0,X_1\})=\delta(-aX_0^2)=-2auX_0^2-2avX_0X_1,\\
%&\{\delta(X_0),X_1\}+\{X_0,\delta(X_1)\}=-a(u+t)X_0^2.
%\end{align*}
%Therefore $u=t$ and $v=0$ and the result is proved since we have $\delta=u\mathcal{E}+w\Delta$.

Let $n>1$ and assume true the statement of the theorem for all algebras $A(n-1,b)$ with $b\neq-(n-2)$.
Thanks to Lemma \ref{xo} there exists $\lambda\in\kk$ such that $\delta(X_0)=\lambda X_0\in\langle X_0\rangle$.
Therefore $\delta$ induces a graded Poisson derivation of the Poisson algebra $A/\langle X_0\rangle\cong A(n-1,a+1)$, see \cite[Proposition 3.7(1)]{LS} for the isomorphism.
Since $a+1\neq-(n-2)$ we can apply the induction hypothesis to get that
%the induction hypothesis 
there exist $\mu,\mu',\nu_i\in\kk$ such that %$\overline{\delta}(\overline{X_i})=\mu \overline{X_i}+\mu' \overline{X_{i-1}}$ for all $i\in\{2,\dots,n\}$ and $\overline{\delta}(\overline{X_1})=\mu \overline{X_1}$.
%Hence there exists $\nu_i\in\kk$ such that
\begin{align*}
\delta(X_i)=\left\{
\begin{array}{ll}
\mu X_i+\mu' X_{i-1}+\nu_i X_0  &\qquad 2\leqslant i\leqslant n,      \\
\mu X_1+\nu_1 X_0  &\qquad i=1.
\end{array}
\right.
\end{align*}
To conclude the proof it remains to show that $\mu=\lambda$, that $\nu_1=\mu'$ and that $\nu_i=0$ for all $2\leqslant i\leqslant n$ since then $\delta=\mu\mathcal{E}+\mu'\Delta$.
If $a\neq0$ then by applying $\delta$ to the Poisson brackets $\{X_0,X_1\}$ (resp. $\{X_0,X_2\}$) a straightforward computation yields $\mu=\lambda$ (resp. $\nu_1=\mu'$).
If $a=0$ then by applying $\delta$ to the Poisson brackets $\{X_1,X_2\}$ a straightforward computation yields $\mu=\lambda$ and $\nu_1=\mu'$.
%For $i>2$ we have
%By applying $\delta$ to the Poisson bracket $\{X_0,X_1\}$ (resp. $\{X_0,X_2\}$) a straightforward computation yields $\mu=\lambda$ (resp. $\nu_1=\mu'$).
Let $2\leq i\leqslant n$.
By comparing the coefficient of $X_0^2$ in the relation $\delta(\{X_1,X_i\})= \{\delta(X_1),X_i\}+\{X_1,\delta(X_i)\}$, we obtain that $\nu_i=0$.

It remains to deal with the case $a = -(n-1)$.
We proceed again by induction on $n \geqslant 2$. 
It is only necessary to prove the initialization (for $A(2,-1)$) because the induction step proceeds exactly as in the case $a\neq-(n-1)$.
Direct computations from the Poisson brackets $\{X_0,X_1\}=X_0^2$, $\{X_0,X_2\}=X_0X_1$ and $\{X_1,X_2\}=X_0X_2$ yield $\delta=\mu\mathcal{E}+\mu'\Delta$.
\end{proof}

%\begin{remark}
%We note that this is consistent with the Poisson automorphisms group given in \cite[Proposition 4.10]{LS}. 
%%See \cite[Theorem 4.9]{LS} for the noncommutative case. 
%\end{remark}

\subsubsection{When $\Delta$ is not necessarily maximal}
\label{deltablock}

We now assume that $(\Delta,\Gamma)$ is a diagonalizable linear solvable pair but we do not assume that $\Delta$ is a maximal Jordan block.
This situation is more complicated as the following example illustrates.

\begin{example}
\label{smallex}
Let $\Delta$ be of Jordan type $(2,2)$ and $\Gamma={\rm diag}(a,a+1,b,b+1)$. 
We obtain the Poisson brackets
\begin{align*}
&\{X_0,X_1\}=-aX_0^2,\quad \{X_0,X_2\}=0,\quad \{X_0,X_3\}=-aX_0X_2\\
&\{X_1,X_2\}=bX_0X_2,\quad \{X_1,X_3\}=(b+1)X_0X_3-(a+1)X_1X_2, \quad \{X_2,X_3\}=-bX_2^2.
\end{align*}
It is easy to verify that ${\rm P.Der}_{\rm gr}(A_{(\Delta,\Gamma)})=\kk\Delta\oplus\kk\mathcal{E}_{01}\oplus\kk\mathcal{E}_{23}$, where $\mathcal{E}_{01}$ is the derivation sending $X_0$ to $X_0$ and $X_1$ to $X_1$ and both $X_2,X_3$ to $0$ (a similar definition applies to $\mathcal{E}_{23}$).
Therefore $\dim{\rm P.Der}_{\rm gr}(A_{(\Delta,\Gamma)})=3$ which is $1$ more the number of blocks
of $\Delta$.
%These computations may be used to generalize a bit~Proposition~\ref{isotensor}: 
%$A_{(\Delta,\Gamma)}$ may not be isomorphic to $A(1,c)\otimes A(1,d)$, $A(2,c)\otimes A(0,d)$ or $A(0,c)\otimes A(2,d)$
%nor to $A(3,c)$, $A(1,c)\otimes A(0,d)\otimes A(0,e)$ and $A(0,c)\otimes A(0,d)\otimes A(0,e)\otimes A(0,f)$.
\end{example}

\begin{definition}\label{dfn-generic} 
Let $(\Delta,\Gamma)$ be a linear solvable pair.
We say that $(\Delta,\Gamma)$ is generic or that $\Gamma$ is generic if all the eigenvalues of $\Gamma$ acting on $A_1$ are distinct and nonzero.
In particular a generic linear solvable pair is a diagonalizable linear solvable pair. 
\end{definition}

Recall from Definition \ref{Annaa} that without loss of generality we have $A_{(\Delta,\Gamma)}=A(\overline{n},\overline{a})$ where the derivation $\Delta$ can be chosen to be in canonical Jordan from (with $n_1 \geqslant n_2 \geqslant \cdots  \geqslant n_r$ denoting the size of the blocks of $\Delta$) and that $\Gamma$ can be chosen diagonal (with the notation of~Remark~\ref{rem-pair-diago} for its eigenvalues).

\begin{theorem}
\label{derp} Assume that $(\Delta,\Gamma)$ is a generic linear solvable pair with $r \geqslant 2$.
\begin{enumerate}[{\rm(1)}]
\item If $n_1=1$ (i.e $\Delta=0$), then ${\rm P.Der}_{\rm gr}A(\overline{n},\overline{a})=\mathfrak{gl}_{n+1}(\kk)$.

\item If $\overline{n}=(2,1,\ldots,1)$, then ${\rm P.Der}_{\rm gr}A(\overline{n},\overline{a})$ is the $2n$-dimensional solvable non nilpotent Lie algebra $\mathscr{L}$ given by 
\[\mathscr{L}= \left\{\delta_{u,v,c_2,\ldots,c_n,d_2,\ldots,d_n}:=\left(\begin{array}{cc@{\ }c@{}c@{}c}  u  & v &&& \\ 0 & u & &&  \\ 
0&c_2&d_2&& \\[-1.1ex] \vdots& \vdots&&\ddots& \\[-1.1ex]0 &c_n&&& d_n \end{array}\right),
\quad u,v,c_2,\ldots, c_n, u_2,\ldots,u_n \in \kk \right\}\]
Moreover, we have $Z(\mathscr{L})=\kk {\rm id} \oplus \kk \Delta$ and 
$[\mathscr{L},\mathscr{L}]=\{\delta_{0,0,c_2,\ldots,c_n,0,\ldots,0}, c_2,\ldots,c_n \in \kk\}$. 
The ascending and descending central series of $\mathscr{L}$ are stationary from rank $1$. 

\item Otherwise ${\rm P.Der}_{\rm gr}A(\overline{n},\overline{a})$ is the $(r+1)$-dimensional abelian Lie algebra given by
\[\kk \Delta \oplus \bigoplus_{i=1}^{r} \kk \mathcal{E}_i\] 
where $\mathcal{E}_i$ is the block diagonal matrix where the diagonal block are $0$ except the $i^\textrm{th}$ which is $\id_{n_i}$.
\end{enumerate}

In particular, up to a scalar factor, $\Delta$ is the only locally nilpotent derivation in the center of 
${\rm P.Der}_{\rm gr}A$, except when $\Delta=0$. 
\end{theorem}

\begin{proof} (1). The case where $\Delta=0$ is trivial. 

(2) and (3). Let $\delta \in {\rm P.Der}_{\rm gr}A(\overline{n},\overline{a})$.
We first prove the following claim: if $n_i  \geqslant 2$ for some $i$, then $\delta$ stabilizes all the blocks except possibly the $i^{\rm th}$. 

It is enough to consider the case $i=1$, i.e. $n_1 \geqslant 2$.
It clearly suffices to consider the case of the second block.
We proceed by induction on $j \in \{n_1,\ldots, n_1+n_2-1\}$ to prove that $\delta(X_j) \in \bigoplus_{j=n_1}^{n_1+n_2-1} \kk X_i$. 
It suffices, for $n_1 \leqslant j < n_1+n_2$, to compare the coefficients of $X_0X_\ell$ for $\ell \neq n_1$ if $j=n_1$ and $\ell \not \in \{n_1,\ldots, n_1+n_2-1\}$ otherwise in the relation $\delta(\{X_1,X_j\})=\{\delta(X_1),X_{n_1}\} + \{X_1,\delta(X_{n_1})\}$.

We now split the argument in two cases:
\begin{enumerate}[{\rm (i)}]
\item there are at least two blocks of size larger or equal to $2$,
\item $n_i=1$ for all $i>1$ and $n_1 \geqslant2$.
%\item $n_i=1$ for $i>1$ and $n_1=2$.
\end{enumerate}
%If there are at least two blocks of sizes larger or equal to $2$, then
In case (i) by applying the claim to two different blocks of size larger or equal than $2$ we see that $\delta$ must stabilize every blocks.
In particular $\delta$ induces a Poisson derivation on each Poisson subalgebra $B_i$ generated by the variables $X_j$ corresponding to a block.
It is clear that $B_i\cong A(n_i-1,a_i)$ for each $1\leqslant i\leqslant r$, so that the induced derivation has a form $\delta|_{B_i}=u{\mathcal{E}}_{n_i}+v\Delta|_{B_i}$ thanks to Theorem \ref{poisdermax}.
Hence $\delta=\sum_{i=1}^r u_i{\mathcal{E}}_{n_i}+v_i\Delta|_{B_i}$ for some $u_i,v_i\in\kk$.
We will show that the $v_i$ are all equal, thus proving assertion (3) in case (i).
First note that there are nothing to check if the corresponding block is of size $1$ since in that case $\Delta_{|B_i}$ is the zero map.
Now if $n_i \geqslant2$ then we compare the coefficients of $X_0X_{n_i}$ in the relation 
$\delta(\{X_1, X_{n_i+1}\})= \{\delta(X_1),X_{n_i+1}\}+\{X_1,\delta(X_{n_i+1})\}$ to obtain that $v_i=v_1$, as desired.

In case (ii) we first show that $\delta(X_i)\in\bigoplus_{j=0}^{n_1-1} \kk X_j$ for all $0\leqslant i\leqslant n_1-2$.
For let $0<i< n_1\leqslant \ell\leqslant n$.
We obtain the result by identifying the coefficients of $X_{k}X_\ell$ for all $k \geqslant n_1-1$ in the relation $\delta(\{X_i,X_\ell\}) = \{\delta(X_i),X_\ell\}+\{X_i,\delta(X_\ell)\}$, using that $\Delta(X_\ell)=0$ and $\delta(X_\ell)=u_\ell X_\ell$ for some $u_\ell \in \kk$.

When $n_1=2$ we conclude by inspection of the Poisson bracket $\{X_0,X_1\}$ to obtain the form for $\delta$ given in (2).
The final assertions of (2) follow from simple computations.

When $n_1 \geqslant3$ we can moreover prove that $\delta(X_{n_1-1})\in\bigoplus_{j=0}^{n_1-1} \kk X_j$ 
by comparing the coefficients of $X_0X_k$ for $k \geqslant n_1$ in the relation $\delta(\{X_1,X_{n_1-1}\})= \{\delta(X_1),X_{n_1-1}\}+\{X_1,\delta(X_{n_1-1})\}$
and using the fact that $\delta$ sends the Poisson subalgebra generated by $X_0,\dots,X_{n_1-2}$ into the Poisson subalgebra generated by $X_0,\dots,X_{n_1-1}$. 
In conclusion, $\delta$ stabilizes the subalgebra generated by $X_0,\dots,X_{n_1-1}$ which is isomorphic to $A(n_1-1,a_1)$ and the result follows from Theorem \ref{poisdermax}. 
This concludes the proof of (3).
%%% cas à étudier : bloc (3,1) non générique, bloc 3,2 non générique. 
\end{proof}

Recall from Theorem~\ref{thm-isom} that the derivation $\Delta$ can be recovered from the Poisson structure of $A$.
The Theorem \ref{derp} is a mean to do so in practice.
More precisely we have the following corollary.

\begin{corollary} Let $(\Delta,\Gamma)$ is a diagonalizable linear solvable pair with $\Gamma$ generic and $\Delta \neq 0$. 
Then $\Delta$ may be recovered from the algebraic structure of $A$ since it is, up to a scalar, the only locally nilpotent derivation contained in the center of ${\rm P.Der}_{\rm gr}A$.
\end{corollary}

%%% faire référence à l'exemple 2\times 2 dans le cas général pour le cas dégénéré sur la dimension des algèbres de Lie. 

\subsubsection{A remark for non necessarily diagonal $\Gamma$}
\label{nondiaggam}
In general the commutant of $\Delta$ is not necessarily made of blocks corresponding to commutant of the Jordan blocks of $\Delta$. 
We illustrate this fact in the case where $\Delta$ has Jordan blocks of size $(2,1)$.
Then the general form for $\Gamma$ is
\[\begin{pmatrix}
a & b & c \\
0 & a+1 & 0 \\
0 & d & e
\end{pmatrix}
\]
In $A_{(\Delta,\Gamma)}$ and $R_{(\Delta,\Gamma)}$ we have
\begin{alignat*}{2}
\{X_0,X_1\}&=-aX_0^2,       &    [X_0,X_1]&=-aX_0\ast X_0, \\
\{X_0,X_2\}&=0,             &    [X_0,X_2]&=0, \\
\{X_1,X_2\}&=eX_0X_2+cX_0^2, &\qquad\qquad [X_1,X_2]&=eX_0\ast X_2+cX_0\ast X_0.
\end{alignat*}
Since both $X_0$ and $X_2$ are in the kernel of $\Delta$ we remark that $\Gamma(X_1)$ does not appear in the formulae of the Poisson bracket and of the product $\ast$.
Therefore, without loss of generality we can assume that $b=d=0$.
In particular $\Gamma$ is triangular and its eigenvalues are $a,a+1$ and $e$.
\begin{itemize}
\item[Case 1.] If $e\neq a$ and $e\neq a+1$ then $\Gamma$ is diagonalizable and one can assume $c=0$ by Proposition \ref{prop-diag-pair}.
\item[Case 2.] Assume that $e=a+1$.
The solvable pair $(\Delta,\Gamma)$ is conjugate to the diagonal solvable pair $(\Delta,\mbox{diag}(a,a+1,a+1))$ via the automorphism of $A$ fixing $X_0$ and $X_1$ and sending $X_2$ to $X_2+cX_0$.
Again we can assume $\Gamma$ diagonalizable.
%By setting $X_2'=X_2+cX_0$ we have
%\begin{alignat*}{2}
%\{X_0,X_1\}&=-aX_0^2,       &    [X_0,X_1]&=-aX_0\ast X_0, \\
%\{X_0,X_2'\}&=0,            &    [X_0,X_2']&=0, \\
%\{X_1,X_2'\}&=(a+1)X_0X_2', &\qquad\qquad [X_1,X_2']&=(a+1)X_0\ast X_2'.
%\end{alignat*}
%Therefore the pair $(\Delta,\Gamma)$ is conjugate to the diagonal solvable pair $(\Delta,\mbox{diag}(a,a+1,a+1))$, i.e. we can assume $\Gamma$ diagonalizable.
\item[Case 3.] Assume then that $e=a$.
If $c\neq0$ then $\Gamma$ is not diagonalizable so that the solvable pair $(\Delta,\Gamma)$ is not diagonalizable (Remark \ref{diagA1daigA}).
Note that the case $c\neq0$ is isomorphic to the case $c=1$ via the change of variable $X_1'=\frac{1}{c}X_1$.
%Then we have
%\[\{X_0,X_1'\}=-a'X_0^2, \quad \{X_0,X_2\}=0,\quad \{X_1',X_2\}=a'X_0X_2+X_0^{2}\]
%and
%\[[X_0,X_1']_\ast=-a'X_0\ast X_0, \quad [X_0,X_2]_\ast=0,\quad [X_1',X_2]_\ast=a'X_0\ast X_2+X_0\ast X_0\]
%where $a'=a/c\in\kk$ so that we can realize this bracket and product via the solvable pair $(\Delta,\Gamma)$ with $c=1$.
\end{itemize}

As a corollary of the following lemma we obtain that Case 3 above with $c\neq0$ cannot be realized by a solvable diagonalizable pair $(\Delta',\Gamma')$.
Recall that the Lie algebra of linear Poisson derivations and its  subset of locally nilpotent derivations is an invariant of homogeneous Poisson polynomial algebras by Proposition \ref{isoda}.

\begin{lemma}
Consider the Poisson algebra $A=A_{(\Delta,\Gamma)}$ from Case 3.
If $a\neq0$ then the space of linear Poisson derivations is of dimension $4$ 
\[
\Pder_{\rm{gr}}(A)=\left\{
\begin{pmatrix}
u_{00} & u_{01} & u_{02} \\
 0     & u_{00} & 0      \\
 0     & u_{21} & u_{00} 
\end{pmatrix} 
\right\}
%={\textrm{Vect}}(\mathcal{E},\Delta,D_{02},D_{21})
\]
In particular $\Pder_{\rm{gr}}(A)$ contains a $3$-dimensional subspace of locally nilpotent derivations. 
If $a=0$ then the space of linear Poisson derivations is of dimension $6$
\[
\Pder_{\rm{gr}}(A)=\left\{
\begin{pmatrix}
u_{00} & u_{01} & u_{02} \\
 0     & u_{11} & u_{12}      \\
 0     & u_{21} & 2u_{00}-u_{11} 
\end{pmatrix} 
\right\}
%=\rm{Vect}(E_{02},E_{11},\Delta,D_{02},D_{21},D_{12})
\]
%Moreover 
%\[Z(\Pder_{\rm{gr}}(A))=
%\left\{\begin{array}{ll}
%\rm{Vect}(\mathcal{E},\Delta) \qquad & a\neq 0 \\
%\rm{Vect}(\mathcal{E})               & a=0
%\end{array}\right.
%\]
\end{lemma}

%In this Lie algebra $E$ and $\Delta$ are central and we have $[D_{02},D_{21}]=\Delta$, thus $Z(\mbox{P.der}_{\rm{gr}}(A))=\mbox{Vect}(E,\Delta)$ and we recover $\Delta$ up to a nonzero scalar as the only nonzero locally nilpotent derivation of $Z(\mbox{P.der}_{\mbox{gr}}(A))$. 
%However $\Delta$ is not anymore central (for instance $[E_{02},\Delta]=\Delta$).
\begin{proof}
The computational proof is omitted. 
\end{proof}

Assume now that $A_{(\Delta,\Gamma)}$ of case $(3)$ is realized by the solvable diagonalizable pair $(\Delta',\Gamma')$.
Then by Theorem~\ref{poisdermax}, the Jordan type of $\Delta'$ is $(2,1)$, since otherwise $\mbox{P.der}_{\rm{gr}}(A)$ is of dimension $2$ (maximal block) or $9$ (Poisson commutative). 
Set $\Gamma'=\diag(\alpha,\alpha+1,\varepsilon)$, with $\alpha \neq 0$ or $\varepsilon \neq0$ since otherwise $A$ would be Poisson commutative. 
If $\alpha \neq \varepsilon$ we compute that
\[
\Pder_{\rm{gr}}(A)=\left\{
\begin{pmatrix}
u_{00} & u_{01} & 0 \\
 0     & u_{00} & 0 \\
 0     & u_{21} & u_{22} 
\end{pmatrix} 
\right\}
\]
is $4$-dimensional, but the subspace of locally nilpotent Poisson derivation is $2$-dimensional.
This cannot happen by the lemma.
If $\alpha = \varepsilon$ then we compute that
\[
\Pder_{\rm{gr}}(A)=\left\{
\begin{pmatrix}
u_{00} & u_{01} & u_{02} \\
 0     & u_{00} & 0 \\
 0     & u_{21} & u_{22} 
\end{pmatrix} 
\right\}
\]
is of dimension $5$.
This cannot happen by the lemma.
%So $A_{(\Delta,\Gamma)}$ cannot be realized by a solvable solvable diagonalizable pair $(\Delta',\Gamma')$.

\begin{rem}
The fact that the case $\Gamma=\begin{pmatrix}
a & 0 & 0 \\
0 & a+1 & 0 \\
0 & d & a+1
\end{pmatrix}$
with $d\neq0$ can be realized by a diagonal solvable pair is a consequence of the fact that $\Gamma(X_1)$ never appears in the formulae. 
This is because all the other variables are in the kernel of $\Delta$.
Therefore we should expect more non diagonalizable cases when the rank of $\Delta$ is a least $2$.
\end{rem}

\section{Diagonalizable linear solvable pair: the algebra $R_{(\Delta,\Gamma)}$}\label{sec-diago-R}

The aim of this section is to give a presentation for the algebra $R$ in the case of a diagonalizable linear solvable pair $(\Delta,\Gamma)$.
According to Proposition~\ref{prop-diag-pair} we can choose a basis $(X_0,\ldots,X_n)$ of $A_1$ such that $\Delta$ is in canonical Jordan form (we denote its Jordan type by $\overline{n}=(n_1,\dots,n_r)$  with $n_1 \geqslant n_2 \geqslant\cdots \geqslant n_r$ and $n+1=\sum_{i=1}^r n_i$) and $\Gamma$ is diagonal of the form $\Gamma =\diag(a_1,\ldots, a_1 + n_1 - 1, a_2,\ldots, a_2 + n_2 -1,\ldots, a_r, \ldots, a_r+n_r-1)$.

%and moreover that $\Gamma$ is diagonal\footnote{We note that triangular is still OK to define the $d$-grading 
%and the $\varepsilon$-filtration we will use latter on.}.
%More precisely $\Delta$ can be taken in Jordan normal form.
%Lets denote its type by $\overline{n}=(n_1,\dots,n_r)$ with $n_1 \geqslant n_2 \geqslant\cdots \geqslant n_r$ and $n+1=\sum_{i=1}^r n_i$.
%
%We assume that $\Gamma$ has a block decomposition adapted to the one of $\Delta$.
%Then, without loss of generalities, $\Gamma$ is a diagonal matrix of the form: 
%\[\Gamma =\diag(a_1,\ldots, a_1+n_1 -1, a_2,\ldots, a_2+ n_2 -1,\ldots, a_r, \ldots, a_r+n_r-1).\]
%%Let $\overline{a}=(a_1,\ldots,a_r)$, we denote by $R(\overline{n},\overline{a})$ the corresponding $R_{(\Delta,\Gamma)}$.
%%We set $n+1=n_1+ \cdots + n_r$.
For any $k \in \{1,\ldots,r\}$ and $0 \leqslant j< n_k$ we denote by $Y_{j,k}=X_{n_1+\cdots+n_{k-1}+j}$ the $j^\textrm{th}$ variable of the $k^{\textrm{th}}$ block.
Fix $k \in \{1,\ldots,r\}$ and $0 \leqslant j< n_k$.
Thanks to equation (\ref{ast}) we have for any $f\in R$ that 
\[f\ast Y_{j,k}= \sum_{\ell \geqslant 0} \Delta^\ell(f) \binom{\Gamma}{\ell}(Y_{j,k})= \sum_{\ell  \geqslant 0} \Delta^\ell(f) \binom{a_k+j}{\ell}Y_{j,k} =\phi_{a_k+j}(f)Y_{j,k}\,.\]
%\quad\mbox{and}\quad X_j\ast f=X_jf+\sum_{\ell \geqslant1}\Delta^\ell(X_j)\binom{\Gamma}{\ell}(f)\,.
In particular, for any $k' \in \{1,\ldots,r\}$ and $0 \leqslant j'< n_{k'}$ we obtain the relations
\[\phi_{-a_k-j}(Y_{j',k'}) \ast Y_{j,k}=Y_{j',k'}Y_{j,k} = \phi_{-a_{k'}-j'}(Y_{j,k}) \ast Y_{j',k'}\,.\]
These relations may be rewritten as
\begin{equation}
\label{relations}
\sum_{\ell=0}^{j'} \binom{-a_{k}-j}{\ell} Y_{j'-\ell,k'} \ast Y_{j,k}= \sum_{\ell=0}^{j} \binom{-a_{k'}-j'}{\ell} Y_{j-\ell,k} \ast Y_{j',k'}
\end{equation}
and we have the following proposition.

\begin{proposition}\label{prop-presentation} 
The algebra $R$ is given by generators $Y_{j,k}$ for $1 \leqslant k \leqslant r$ and $1 \leqslant j \leqslant n_k$, and the homogeneous relations~(\ref{relations}).
\end{proposition}

\begin{proof} We first show that the set $\mathcal{G}=\{Y_{j,k}\ |\ {1 \leqslant k \leqslant r, 1 \leqslant j \leqslant n_k}\}$ is a generating set for $R$.
From Proposition~\ref{prop-graduation-induite}, $\gr^{\wteps}(A)$ is generated by the image of the $Y_{j,k}$ in $\gr^{\wteps}(A)$.
But from Proposition~\ref{prop-epsilontilde}, $\gr^{\wteps}(A)=\gr^{\wteps}(R)$. 
Hence $\mathcal{G}$ is a set of generators for $R$ as desired.

Denote by $T$ the algebra given by generators $Y_{j,k}$ for $1 \leqslant k \leqslant r$ and $1 \leqslant j \leqslant n_k$, and relations~(\ref{relations}).
From the preceding argument $R$ is a quotient of $T$.
Moreover note that relations ~(\ref{relations}) are homogeneous with respect to the degree $\deg(Y_{j,k})=1$ for all $j,k$.
Hence the natural map from $T$ onto $R$ is graded.
We endow the set $\mathcal{G}$ with the order given by $Y_{j,k}\leqslant Y_{j',k'}$ if $k < k'$, or if $k=k'$ and $j \leqslant j'$. 
This induces an order on the monomials in the $Y_{j,k}$'s. 
The relations~(\ref{relations}) may be rewritten as 
\[Y_{j,k}\ast Y_{j',k'} = Y_{j',k'} \ast Y_{j,k} + \textrm{lower monomials}\]
for 
all $(j,k),(j',k')$.
Hence every element of $T$ is a linear combination of monomials $Y_{j_1,k_1}\cdots Y_{j_s,k_s}$ in the $Y_{j,k}$'s where $Y_{j_\ell,k_\ell} \leqslant Y_{j_{\ell+1},k_{\ell+1}}$ for all $\ell\in\{1,\dots,s-1\}$. 
In particular, the dimension of the homogeneous component of degree $d$ of $T$ is smaller than the one of $R$.
But since $R$ is a graded quotient of $T$ the canonical map from $T$ onto $R$ must be one-to-one.
\end{proof}

\begin{example}
Let $\Delta$ be of Jordan type $(2,2)$ and let $\Gamma={\rm diag}(a,a+1,b,b+1)$ be as in Example \ref{smallex}.
Thanks to Proposition~\ref{prop-presentation} we obtain the following complete set of relations between the generators of $R=R_{(\Delta,\Gamma)}$ :
\begin{align*}
X_0\ast X_1-X_1\ast X_0 &= -a X_0\ast X_0, \qquad X_2\ast X_3-X_3\ast X_2 = -b X_2\ast X_2,\\
X_0\ast X_2-X_2\ast X_0 &= 0, \\
X_0\ast X_3-X_3\ast X_0 &= -aX_0\ast X_2, \qquad X_1\ast X_2-X_2\ast X_1 = bX_0\ast X_2\\
X_1\ast X_3-X_3\ast X_1 &= (b+1)X_0\ast X_3 -(a+1)X_2\ast X_1\\
&=(b+1)X_0\ast X_3 -(a+1)X_1\ast X_2+(a+1)b X_0\ast X_2
\end{align*}
%Moreover, we can use~Proposition~\ref{prop-normal-element} to compute the normal elements of $R$ contained in $A_1$. We have $\ker(\Delta|_{A_1})=\mbox{Vect}(X_0,X_2)$ so that $X_0$ and $X_2$ are normal in $R$. Moreover if $a \neq b$ then the set of normal elements in $A_1$ is $\{\alpha X_0, \alpha X_2\ |\ \alpha \in \kk\}$.
%If $a=b$ there exists more normal elements: every linear combination of $X_0$ and $X_2$ is normal.
%\[X_2 \quad\mbox{and}\quad X_0+\lambda X_2 \quad \lambda\in\kk.\]    
%Note that $X_0X_3-X_2X_1=X_0 \ast X_3-X_2 \ast X_1$ is also a normal element in $R$ which is central if and only if $a+b+1=0$. 
\end{example}

\section{Unimodarity and Calabi-Yau property}\label{sec-CYunimod}

In this section we study the unimodularity of the Poisson algebra $A=A_{(\Delta,\Gamma)}$ in the polynomial case and determine when the algebra $R=R_{(\Delta,\Gamma)}$ is Calabi-Yau.
Whereas the unimodularity of $A$ is completely characterized by Proposition~\ref{prop-modular-derivation}, we only compute the Nakayama automorphism of $R$ in the case of a generic diagonalizable linear solvable pair (Theorem~\ref{th-nakayama}).
This result generalizes \cite[Theorem 4.16]{LS}. 

\subsection{Unimodularity of $A_{(\Delta,\Gamma)}$ in the polynomial case}
\label{unimod}

Unimodular Poisson algebras are thought to be Poisson analogous of Calabi-Yau algebras, see \cite{Do}.
Recall that a polynomial Poisson algebra $A=\kk[X_0,\dots,X_n]$ is called unimodular if its modular derivation is zero, where the modular derivation $m$ of $A$ is given in the polynomial case by
\[m(f)=\sum_{k=0}^n\frac{\partial \{X_k,f\}}{\partial X_k}\] 
for all $f\in A$, see~\cite[Lemma 2.4]{LWW}.

\begin{proposition}\label{prop-modular-derivation}
Let $A=\kk[X_0,\dots,X_n]$ and $(\Delta,\Gamma)$ be a solvable pair on $A$.
The modular derivation $m$ of the Poisson algebra $A_{(\Delta,\Gamma)}$ is equal to $m=(1-\diver(\Gamma))\Delta$.
In particular, as long as $\Delta\neq0$, the Poisson algebra $A_{(\Delta,\Gamma)}$ is unimodular if and only if $\diver(\Gamma)=1$.
When $\Gamma$ is linear this is equivalent to say that $\mathrm{tr}(\Gamma)=1$.
\end{proposition}

\begin{proof}
Thanks to \cite[Equality 4.22 p.108]{LGPV} the modular derivation $m$ of $A_{(\Delta,\Gamma)}$ is equal to minus the divergence of the Poisson bivector field $\pi=\Delta\wedge\Gamma$ of $A_{(\Delta,\Gamma)}$.
Thanks to \cite[Proposition 4.16]{LGPV} and the fact that the divergence of a locally nilpotent derivation is zero (\cite[Proposition 1.3.51]{VdE2}) we have
\begin{align*}
 m
 & = -\diver(\pi) \\
 & = -\diver(\Delta\wedge\Gamma) \\
 & = -\big(\diver(\Gamma)\Delta - \diver(\Delta)\Gamma - [\Delta,\Gamma]\big)\\
 & = -\diver(\Gamma)\Delta+\Delta \\
 & = (1-\diver(\Gamma))\Delta
\end{align*}
as desired.
For the assertion in the linear case, recall that $\diver(\Gamma)=\sum_{i=0}^n\dfrac{\partial \Gamma(X_i)}{\partial X_i}$.
\end{proof}

We illustrate this result with several examples in the linear case.

\begin{example}
If $(\Delta,\Gamma)$ is a diagonalizable linear solvable pair $(\Delta,\Gamma)$, recall that
\begin{enumerate}
\item its Jordan type is $\overline{n}=(n_1,\ldots, n_r)$ where $n_1  \geqslant n_2  \geqslant \cdots  \geqslant n_r$ are the size of the Jordan blocks of $\Delta$ acting on $A_1$ ;
\item there exists a basis of $A_1$ and 
$\overline{a}=(a_1,\ldots,a_r) \in \kk^r$ such that $\Delta$ is in canonical Jordan form and $\Gamma$ is diagonal of the form
$\Gamma =\diag(a_1,\ldots, a_1+n_1 -1, a_2,\ldots, a_2+ n_2 -1,\ldots, a_r, \ldots, a_r+n_r-1)$
\item we denote by $A(\overline{n},\overline{a})$ the corresponding $A_{(\Delta,\Gamma)}$ ;
\item we set $n+1=n_1+ \cdots + n_r$.
\end{enumerate}
The parameter space of unimodular Poisson algebras $A(\overline{n},\overline{a})$ is a $(r-1)$-dimensional affine subspace of $\kk^{r}$ since the Poisson algebra $A(\overline{n},\overline{a})$ is unimodular if and only if $\mbox{tr}(\Gamma)=1$ if and only if
\[\sum_{t=0}^{r-1}n_{t+1}a_{t+1}+\binom{n_{t+1}}{2}=1.\]
\end{example}

\begin{example}
If $r=1$ (a single Jordan block) then $A$ is unimodular if and only if
\[(n+1)a+\binom{n+1}{2}-1=0 \iff a=\frac{-(n+2)(n-1)}{2(n+1)} \]
as expected from \cite{LS}, since $n+1=n_1+\cdots+n_r=n_1$. 
\end{example}

%{\color{red}J'ai mis en commentaire le cas bloc max en dim 3 qui est d\'ej\`a connu, et j'ai bien r\'eduit l'example suivant.}
\begin{comment}
\begin{example}
Let $\Delta$ be of Jordan type $(3)$ and set $\Gamma=\mbox{diag}(a,a+1,a+2)$ for some $a\in\kk$.
Then $A=A_{(\Delta,\Gamma)}=\kk[X_0,X_1,X_2]$ has Poisson brackets given by
\[\{X_0,X_1\}=-aX_0^2,\qquad \{X_1,X_2\}=(a+2)X_0X_2-(a+1)X_1^2,\qquad \{X_2,X_0\}=aX_0X_1\]
and $R=R_{(\Delta,\Gamma)}$ has product
\begin{align*}
& X_0 \ast X_1-X_1 \ast X_0 = - a X_0 \ast X_0 \\
& X_1 \ast X_2-X_2 \ast X_1 = (a+2) X_0 \ast X_2-(a+1) X_1 \ast X_1 +\binom{a+2}{2}X_0 \ast X_1 \\
&X_2 \ast X_0 - X_0 \ast X_2 = a X_0\ast X_1 +\binom{a}{2}X_0\ast X_0.
\end{align*}
The Poisson algebra $A$ is unimodular if and only if $a=-2/3$.
Moreover this condition is equivalent to the fact that the Poisson structure on $A$ is Jacobian. 
In that case the potential is 
\[\phi=\frac{2}{3}X_0^2X_2-\frac{1}{3}X_0X_1^2=\frac{2}{3}X_0\left(X_0X_2-\frac{1}{2}X_1^2\right).\]

When $a=-2/3$, the algebra $R$ derives from the potential 
\[\Phi=X_0\ast X_1\ast X_2 - X_2 \ast X_1\ast X_0 +2 X_2 \ast X_0 \ast X_1 - 2 X_0 \ast X_2 \ast X_1   -\frac{2}{3} {X_0}^{\ast 2}\ast X_2+\frac{1}{3}{X_1}^{\ast 2} \ast X_0-\frac{1}{9}{X_0}^{\ast 2}\ast X_1\]
This means that the quotient of the free algebra $\kk\langle X_0,X_1,X_2\rangle$ by the ideal $\left(\frac{\partial \Phi}{\partial X_2},\frac{\partial \Phi}{\partial X_1},\frac{\partial \Phi}{\partial X_0}\right)$ is isomorphic to $R$.
\end{example}
\end{comment}

\begin{example} Let $\Delta$ be of Jordan type $(2,1)$ and $\Gamma$ non necessarily diagonalizable.
The general form for $\Gamma$ is $\begin{pmatrix}
a & b & c \\
0 & a+1 & 0 \\
0 & d & e
\end{pmatrix}$, hence the Poisson algebra $A_{(\Delta,\Gamma)}$ is unimodular if and only if $2a+e=0$.
\end{example}

\subsection{Calabi-Yau property for $R_{(\Delta,\Gamma)}$ in the generic diagonalizable case}

The aim of this section is to compute the Nakayama automorphism of $R$ in order to determine when $R$ is a Calabi-Yau algebra.
We denote by $R^e=R\otimes_\kk R^{\rm{op}}$ the enveloping algebra algebra of $R$.

\begin{defn}\label{def-CY}
We say that $R$ is {\em skew Calabi-Yau} (or {\em skew CY}) if
\begin{itemize}
\item[(i)] $R$ is {\em homologically smooth}:  $R$ has a finite projective resolution as a left $R^e$-module such that each term is finitely generated;
\item[(ii)]  There are an algebra automorphism $\mu$ of $R$ and an integer $d$ such that
\[ 
\Ext^i_{R^e}(R, R^e) \cong \begin{cases} 0 & \text{if $i \neq 0$} \\
{}^1 R^\mu & \text{ if $i=d$}
\end{cases}
\]
where ${}^1 R^\mu$ is the $R$-bimodule which is isomorphic to $R$ as a $\kk$-vector space and such that $r\cdot s \cdot t = r s \mu(t)$.
\end{itemize} 
If $R$ is skew CY, the automorphism $\mu$ is called the {\em Nakayama automorphism} of $R$.  If $\mu$ is inner, then $R$ is {\em Calabi-Yau} or {\em CY}.
\end{defn}
 
By \cite[Lemma~1.2]{RRZ}, any AS-regular connected graded algebra is skew CY. In particular, the algebras $R_{(\Delta,\Gamma)}$ are skew CY in the linear case thanks to Theorem \ref{theo-ASregul}.

\begin{proposition}\label{prop-commutation} Let $(\Delta,\Gamma)$ be a linear solvable pair such that $A$ is not Poisson commutative (see Corollary~\ref{cor-poissoncom}). 
Let $\Phi$ be an automorphism of $R$ compatible with $\widetilde{\varepsilon}$ such that $\ol{\Phi}=\id$ (see Lemma~\ref{lem-autom-deriv}). Assume that 
%every normal element of $R$ is strongly normal and that 
there exists a strongly normal element $N$ of $R$ whose eigenvalue with respect to $\Gamma$ is nonzero. 
Then $\Phi$ commute with $\Delta$ and $\phi_c$ for every $c \in \kk$. 
\end{proposition}

\begin{proof}The proof decomposes into three steps.
The first step consists to show that $\Phi$ commutes with one $\phi_\lambda$ by considering the normal element $N \in R$.
In the second step, we deduce from the first that $\Phi$ commutes with $\Delta$ by considering that $\Phi$ commutes with $\phi_{k\lambda}$ for $k \in \N$.
The third step is easy: since $\Phi$ commute with $\Delta$, it commutes with every formal power series in $\Delta$, hence with $\phi_c$ for every $c \in \kk$. 

First step. Consider $N \in \ker \Delta \cap \ker (\Gamma-\lambda\id)$ with $\lambda \neq 0$.
Then $N$ is a normal element, hence $\Phi(N)$ is too.
Since every normal element in $R$ is strongly normal thanks to Proposition~\ref{prop-normal-element}, there exists $\mu \in \kk$ such that $\Gamma(\Phi(N))=\mu \Phi(N)$.
Since Lemma~\ref{lem-autom-normalR} and Lemma~\ref{lem-strongly} extend with the same proof to the filtration~$\widetilde{\epsilon}$ and since the automorphism associated to the normal element $\Phi(N)$ is $\Phi \circ \phi_\lambda \circ \Phi^{-1}=\phi_\mu$,
we deduce that the Poisson derivation of $\gr^{\wteps}(R)$ associated to $N$ is $\mu \ol{\Delta}= \ol{\Phi}\circ (\lambda \ol{\Delta})\circ \ol{\Phi}^{-1}$, where the notation $\ol{\,\cdot\,}$ represent the induced map in $\gr^{\wteps}(R)$.
However we have $\ol{\Phi}=\id$, hence $\mu \ol{\Delta}=\lambda \ol{\Delta}$.
Since $\Delta \neq 0$, we have $\ol{\Delta} \neq 0$.
Thus we obtain $\mu=\lambda$ and therefore that $\Phi$ and $\phi_\lambda$ commute. 

Second step. For $f \in R$, choose an integer $n$ such that $\Delta^{n+1}(f)=0$ and $\Delta^{n+1}(\Phi(f))=0$. The aim is to express $\Delta(f)$ as a linear combination of $(\phi_{k\al}(f))_{0 \leqslant k \leqslant n}$.
Since $\lambda \neq 0$, the matrix of generalized binomial coefficients
$$A=\left(\binom{\lambda k}{\ell}\right)_{0 \leqslant k,\ell \leqslant n}$$
is invertible (see Lemma~\ref{lem-determinant}).
Hence there exists $(\alpha_0,\ldots, \alpha_{n})\in\kk^{n+1}$ such that 
$$A \begin{pmatrix} \alpha_0 \\ \vdots \\ \alpha_n\end{pmatrix} = \begin{pmatrix} 0 \\ 1 \\ 0 \\ \vdots \\0 \end{pmatrix}\,.$$
We then deduce that $\Delta(f)=\sum_{k =0}^n \alpha_k \phi_{k\lambda}(f)$ and $\Delta(\Phi(f))=\sum_{k =0}^n \alpha_k \phi_{k\lambda}(\Phi(f))$.
But $\Phi$ commutes with $\phi_\lambda$, therefore it commutes with $\phi_{k\lambda}=({\phi_{\lambda}})^k$ for any $k \in \{0,\ldots,n\}$ (Lemma \ref{compo}).
From the preceding equalities we obtain that $\Phi(\Delta(f))=\Delta(\Phi(f))$.
\end{proof}

\begin{remark} When the solvable pair $(\Delta,\Gamma)$ is not necessarily linear, the preceding proposition can be adapted with the following statement.
Let $\Phi$ be an automorphism of $R$ compatible with $\varepsilon$ and such that $\ol{\Phi}=\id$.
Assume that every normal element of $R$ is strongly normal and that there exists a strongly normal element $N$ of $R$ whose eigenvalue with respect to $\Gamma$ is nonzero. 
Then $\Phi$ commute with $\Delta$ and $\phi_c$ for every $c \in \kk$. 
\end{remark}

The preceding proposition applies in particular to the Nakayama automorphism of $R$ as detailed in next example. 

\begin{example}\label{ex-nakayama} Let $(\Delta,\Gamma)$ be a linear solvable pair. From Proposition~\ref{prop-graduation-induite}, $\gr^{\widetilde{\epsilon}}(A)$ is a polynomial algebra, 
hence we can apply~\cite[Theorem 5.7]{WZ} to get that the Nakayama automorphism $\mu$ of $R$ is compatible with $\widetilde{\epsilon}$ and verifies $\ol{\mu}=\id$ 
and $\ol{(\mu-\id)}=(1-\textrm{tr}(\ol{\Gamma}))\ol{\Delta}= (1-\textrm{tr}(\Gamma))\ol{\Delta}$.
\end{example}

\begin{theorem}\label{th-nakayama} Assume that $(\Delta,\Gamma)$ is a generic diagonalizable linear solvable pair on $A=k[X_0,\ldots, X_n]$ (see Definition~\ref{dfn-generic}).
Then the Nakayama automorphism of $R$ is $\phi_{1-\mathrm{tr}(\Gamma)}$.
\end{theorem} 

\begin{proof}%We work on $A_1$ the set of homogeneous polynomial of degree $1$.
Let $\mu$ be the Nakayama automorphism of $R$.
It is a graded automorphism for the standard degree%(\textcolor{red}{ADD REF})
.
Since $\ker \Delta$ is stable by $\Gamma$ and $\Gamma$ is diagonalizable, it induces a diagonalizable endomorphism of $\ker \Delta$. Hence there exists a strongly normal element $N$ 
whose eigenvalue with respect to $\Gamma$ is nonzero since $0$ is not an eigenvalue of $\Gamma$. 
From Proposition~\ref{prop-commutation} and Example~\ref{ex-nakayama}, $\mu$ commute with $\Delta$ and with $\phi_c$ for all $c \in \kk$.
In particular, Proposition~\ref{prop-normal-element} shows that every strongly normal element of $A_1$ is sent by $\mu$ to a strongly element of $A_1$ with the same eigenvalue with respect to $\Gamma$. 

Consider a basis $(X_0,\ldots, X_n)$ of $A_1$ such that $\Delta$ is in canonical Jordan form of type $\ol{n}=(n_1,\ldots,n_r)$ and $\Gamma$ is diagonal (see Proposition~\ref{prop-diag-pair}).
For $1 \leqslant k\leqslant r$ and $0 \leqslant i \leqslant n_r-1$ we denote by $Y_{i,k}=X_{i+ n_1+ \cdots + n_{k-1}}$ the $i^{\textrm{th}}$ element of the $k^{\textrm{th}}$ block. 
In particular %$Y_{0,k}$ is a strongly normal element of $A_1$ and 
$\ker \Delta$ is the linear span of the $Y_{0,k}$ for $1 \leqslant k\leqslant r$. 

Since the family $\mathcal{G}=\{Y_{j,k}\ |\ {1 \leqslant k \leqslant r, 1 \leqslant j \leqslant n_k}\}$ is a generating set of $R$, it is enough to prove that $\mu$ and $\phi_{1-\mathrm{tr}(\Gamma)}$ coincide on $\mathcal{G}$.
%Since the eigenvalue of $\Gamma$ acting on $A_1$ are distinct, we deduce that there exists $\al_k \in \kk$, $\mu(Y_{0,k})= \al_k Y_{0,k}$. 
Since $\ol{\mu}=\id$ (Example~\ref{ex-nakayama}) and $\widetilde{\varepsilon}(Y_{0,k})=0$, we obtain $\mu(Y_{0,k})=Y_{0,k}=\phi_{1-\mathrm{tr}(\Gamma)}(Y_{0,k})$.
Since $\varepsilon(Y_{1,k})=1$, \cite[Theorem 5.7]{WZ} and Proposition~\ref{prop-modular-derivation} show that $\mu(Y_{1,k})=Y_{1,k}+(1-\mathrm{tr}(\Gamma))\Delta(Y_{1,k})=\phi_{1-\mathrm{tr}(\Gamma)}(Y_{1,k})$.

Therefore we have the desired equalities for every $k$ such that $n_k \leqslant 2$.
Assume now that $n_k  \geqslant 3$.
We prove by induction on $i \geqslant 2$ that $\mu(Y_{i,k})=\phi_{1-\mathrm{tr}(\Gamma)}(Y_{i,k})$. 
Assume that $i \geqslant 1$ and $\mu(Y_{j,k})=\phi_{1-\mathrm{tr}(\Gamma)}(Y_{j,k})$ for $j \leqslant i$.
Since both $\mu$ and $\phi_{1-\mathrm{tr}(\Gamma)}$ commutes with $\Delta$, we have
\[\Delta \big((\mu-\phi_{1-\mathrm{tr}(\Gamma)})(Y_{i+1,k})\big)=(\mu-\phi_{1-\mathrm{tr}(\Gamma)})(\Delta(Y_{i+1,k}))=(\mu-\phi_{1-\mathrm{tr}(\Gamma)})(Y_{i,k})=0\,.\]
Hence there exist scalars $\alpha_0,\dots,\alpha_r$ such that
\begin{equation}
\label{eqmuY}
\mu(Y_{i+1,k}) = \phi_{1-\mathrm{tr}(\Gamma)}(Y_{i+1,k}) + \sum_{\ell =0}^r \alpha_{\ell} Y_{0,\ell}
\end{equation}
If $f \in R$ is an eigenvector for $\Gamma$ associated to the eigenvalue $\lambda$, then for every $g \in R$ we have $g \ast f= \phi_\lambda(g)f$. 
Set $\Gamma(Y_{0,k})=\lambda Y_{0,k}$ where $\lambda \in \kk$.
Then $\Gamma(Y_{1,k})=(\lambda +1)Y_{1,k}$ and we obtain $\Gamma(Y_{i+1,k})=(\lambda +i+1)Y_{i+1,k}$ by an easy induction on $i\in\{0,\dots,n_k-1\}$. 
Hence we obtain that 
\begin{equation}
\label{eqphiY}
\phi_{-\lambda-1}(Y_{i+1,k})\ast Y_{1,k}=Y_{i+1,k}Y_{1,k}=\phi_{-\lambda-i-1}(Y_{1,k})\ast Y_{i+1,k}
\end{equation}
For the sake of simplicity we set $c=1-\mathrm{tr}(\Gamma)$.
By applying $\mu$ and $\phi_c$ to the relation \eqref{eqphiY} and by using the commutation of $\mu$ with $\phi_\alpha$, of $\phi_\beta$ with $\phi_\alpha$ for every $\alpha,\beta \in \kk$, and the relation \eqref{eqmuY}, we obtain 
\begin{align*}
0=&\phi_{-\lambda-1}(\mu(Y_{i+1,k}))\ast \mu(Y_{1,k})-\phi_{-\lambda-i-1}(\mu(Y_{1,k}))\ast \mu(Y_{i+1,k})\\
=&\phi_{-\lambda-1}\left(\phi_c(Y_{i+1,k}) + \sum_{\ell =0}^r \alpha_{\ell} Y_{0,\ell}\right)\ast \phi_c(Y_{1,k})-\phi_{-\lambda-i-1}(\phi_c(Y_{1,k}))\ast \left(\phi_c(Y_{i+1,k})+ \sum_{\ell =0}^r \alpha_{\ell} Y_{0,\ell}\right)\\
=&\phi_{-\lambda-1}\left(\sum_{\ell =0}^r \alpha_{\ell} Y_{0,\ell}\right) \ast \phi_c(Y_{1,k}) - \phi_{-\lambda-i-1}(\phi_c(Y_{1,k}))\ast \sum_{\ell =0}^r \alpha_{\ell} Y_{0,\ell}.
\end{align*}
By applying $\phi_{-c}$ to the last equality, we obtain the relation 
$$\left(\sum_{\ell =0}^r \alpha_{\ell} Y_{0,\ell}\right) \ast Y_{1,k}= \phi_{-\lambda-i-1}(Y_{1,k})\ast \sum_{\ell =0}^r \alpha_{\ell} Y_{0,\ell}$$
since $f=\sum_{\ell =0}^r \alpha_{\ell} Y_{0,\ell} \in \ker \Delta$ and so $\phi_\alpha(f)=f$ for all $\alpha \in \kk$.
By using the expression of the product $\ast$, this last relation can be rewritten as $Y_{0,k}\sum_{\ell=0}^r \alpha_\ell (\lambda_\ell-(\lambda+i+1)) Y_{0,\ell}=0$ where $\lambda_\ell$ is the eigenvalue for $\Gamma$ associated to the eigenvector $Y_{0,\ell}$. 
Since $\Gamma$ has distinct eigenvalues, we obtain that $\alpha_\ell=0$ for all~$\ell$ and $\mu(Y_{i+1,k}) = \phi_{1-\mathrm{tr}(\Gamma)}(Y_{i+1,k})$.
This concludes the proof.
\end{proof}

\begin{corollary}\label{cor-calabiyau} 
Let $(\Delta,\Gamma)$ be a generic diagonalizable linear solvable pair on $A=k[X_0,\ldots,X_n]$. Then $A_{(\Delta,\Gamma)}$ is unimodular if and only if $R_{(\Delta,\Gamma)}$ is Calabi-Yau if and only if ${\rm{tr}}(\Gamma)=1$. 
\end{corollary}

\begin{proof} Since $R$ is a connected graded algebra, every inner automorphism is trivial. Hence $R$ is Calabi-Yau if and only if its Nakayama automorphism is trivial. 
The result follows from~Proposition~\ref{prop-modular-derivation} and Theorem~\ref{th-nakayama} since $\Phi_c=\id$ if and only if $c=0$.
\end{proof}

\appendix
\section{Combinatorial relations}
\label{combinatorial}

Let $B$ be an associative $\kk$-algebra and $b \in B$.
Recall that that for any $k \in \N$ the generalized binomial coefficient $\binom{b}{k}$ is defined by $\binom{b}{k}=\frac{b(b-1)\cdots (b-k+1)}{k!}$.
For $k \in \Z$ with $k <0$, we set $\binom{b}{k}=0$.

\begin{lemma} Let $a,b \in B$ and assume that $ab=ba$.
For any $n \geqslant 0$ we have
\begin{align}
&\binom{a}{n} + \binom{a}{n-1} = \binom{a+1}{n}  \label{pascal}\\
&\sum_{\ell=0}^n\binom{a}{\ell}\binom{b}{n-\ell}=\binom{a+b}{n} \qquad\mbox{(Chu-Vandermonde identity)} \label{vandermonde}\\
&\binom{a}{n}=(-1)^n\binom{n-a-1}{n} \label{binom1}
\end{align}
\end{lemma}

%{\color{red} Est-ce utile de garder cela ? Vincent 14/02 : je pense que oui : il faut garder au moins les formules en disant voici les relations utilisées dans l'article. Je laisse tout dans la
%V1}

We need a better understanding of commutation relations in the enveloping algebra of the two dimensional solvable Lie algebra spanned by $\Delta$ and $\Gamma$.

\begin{lemma} \label{comDG} Let $\Gamma,\Delta \in B$ and assume that $[\Delta,\Gamma]=\Delta\Gamma-\Gamma\Delta=\Delta$.
For any $i \geqslant 0$ and $k \geqslant0$ we have
\begin{align*}
&\left[\Delta,\binom{\Gamma}{k}\right]=\binom{\Gamma}{k-1}\Delta,
&\left[\Delta^i,\binom{\Gamma}{k}\right]=\sum_{\ell=1}^{i}\binom{i}{\ell}\binom{\Gamma}{k-\ell}\Delta^i
\end{align*}
\end{lemma}

\begin{proof}
The first assertion is clear for $k=0$.
Let $k \geqslant1$.
Since $\Delta(\Gamma-u)=(\Gamma-(u-1))\Delta$ for any $u \in \kk$ we have 
\begin{align*}
k!\left[\Delta,\binom{\Gamma}{k}\right]&=\big(\Delta\Gamma(\Gamma-1)\cdots(\Gamma-(k-1))-\Gamma(\Gamma-1)\cdots(\Gamma-(k-1))\Delta\big)\\
&=\big((\Gamma+1)\Gamma\cdots(\Gamma-(k-2))\Delta-\Gamma(\Gamma-1)\cdots(\Gamma-(k-1))\Delta\big)\\
&=\Gamma\cdots(\Gamma-(k-2))(\Gamma+1-(\Gamma-(k-1)))\Delta\\
&=k\Gamma\cdots(\Gamma-(k-2))\Delta\\
&=k(k-1)!\binom{\Gamma}{k-1}\Delta\\
&=k!\binom{\Gamma}{k-1}\Delta
\end{align*}
and the result is proved.
We prove the second assertion by induction on $i \geqslant0$. The case $i=0$ is clear. The initialization $i=1$ has just been proved.
Assume that the result is true for some $i \geqslant1$ and all $k \geqslant0$.
Then
\begin{align*}
\left[\Delta^{i+1},\binom{\Gamma}{k}\right]&=\Delta^i\Delta\binom{\Gamma}{k}-\binom{\Gamma}{k}\Delta^{i+1}\\
&=\Delta^i\left(\binom{\Gamma}{k}\Delta+\binom{\Gamma}{k-1}\Delta\right)-\binom{\Gamma}{k}\Delta^{i+1}\\
&=\left[\Delta^i,\binom{\Gamma}{k}\right] \Delta +\Delta^i\binom{\Gamma}{k-1}\Delta\\
&=\left(\sum_{\ell=1}^i\binom{i}{\ell}\binom{\Gamma}{k-\ell}+\sum_{\ell=0}^i\binom{i}{\ell}\binom{\Gamma}{k-\ell-1}\right)\Delta^{i+1}\\
&=\sum_{\ell=1}^{i+1}\binom{i+1}{\ell}\binom{\Gamma}{k-\ell}\Delta^{i+1}.
\end{align*}
\end{proof}

\begin{remark}\label{rem-reecriture} 
The second relation of~Lemma~\ref{comDG} can be rewritten 
\[\Delta^i\binom{\Gamma}{k}=\sum_{\ell=0}^{i}\binom{i}{\ell}\binom{\Gamma}{k-\ell}\Delta^i =
\sum_{\ell=0}^{\min(i,k)}\binom{i}{\ell}\binom{\Gamma}{k-\ell}\Delta^i= 
\sum_{\ell=0}^{k}\binom{i}{\ell}\binom{\Gamma}{k-\ell}\Delta^i= 
\sum_{\ell=0}^{k}\binom{i}{k-\ell}\binom{\Gamma}{\ell}\Delta^i\]
Indeed, for $\ell > k$, $\binom{\Gamma}{k-\ell}=0$ and for $\ell > i$, $\binom{i}{\ell}=0$. 
\end{remark}

We finish this combinatorial appendix by the computation of the following determinant which will be useful for the determination of the Nakayama automorphism of $R$. 

\begin{lemma}\label{lem-determinant}
Let $\kk$ be a field of characteristic $0$ and $A$ be a $\kk$-algebra. For $a \in A$ and $n \in \N$ set 
$$M(a) = \left(\binom{ka}{\ell}\right)_{0\leqslant k,\ell \leqslant n} \in M_{n+1}(A)$$
and $m(a)=\det(M(a))$. Then $$m(a) = a^{n(n+1)/2}\,.$$
\end{lemma}

\begin{proof}
We can write $\binom{ka}{\ell}= \sum_{i=0}^{\ell} c_{i\ell} (ka)^i$ with $c_{i\ell} \in \Q$.
Set $c_{i\ell}=0$ for $\ell< i \leqslant n$.
We have $M(a)= V(a)C$ where $C=(c_{k\ell})_{0\leqslant k,\ell \leqslant n}\in M_{n+1}(\Q)$ and $V(a)=((ka)^\ell)_{0\leqslant k,\ell \leqslant n} \in M_{n+1}(A)$.
Since $C$ is an upper triangular matrix with diagonal coefficients $c_{\ell\ell}=1/\ell!$ and $V(a)$ is a Vandermonde matrix, we obtain the desired equality.
\end{proof}

%{\color{red} Est-ce utile de garder l'appendice B ? Je pense que non.}

\section{Tensor product of Poisson algebras}
\label{tensor}

Let $A$ and $B$ be Poisson algebras. 
The tensor product $A\otimes B$ is a Poisson algebra for the Poisson bracket given by
\begin{equation*}
\{a\otimes b,a'\otimes b'\}_{A\otimes B}=\{a,a'\}_A\otimes bb'+aa'\otimes\{b,b'\}_B
\end{equation*}
for $a,a'\in A$ and $b,b'\in B$.
A derivation $\delta$ of $A$ extends to a derivation $\widehat{\delta}$ of $A\otimes B$ by setting
\[\widehat{\delta}(a\otimes b)=\delta(a)\otimes b\]
for $a\in A$ and $b\in B$.
Moreover, if $\delta$ is a Poisson derivation of $A$ then $\widehat{\delta}$ is a Poisson derivation of $A\otimes B$.
The same is true for a (Poisson) derivation of $B$ with the obvious modification.

\section{Symmetric Algebra of a filtered vector space}
\label{symalg-filtre}

Let us consider $V= \cup_{i \in \N} V_i$ a filtration of the vector space $V$, that is to say, for every $i \in \N$,
$V_i$ is a subspace of $V$ and $V_i \subseteq V_{i+1}$. 
This filtration on $V$ induces algebra filtrations on the tensor algebra $T(V)$ of $V$ and on the symmetric algebra $S(V)$
of $V$, that are given by
$$T_\al^n= \!\!\!\!\!\!\!\!\!
\sum_{\begin{array}{c} \scriptstyle{(\al_1,\ldots,\al_n) \in \N^n}\\ \scriptstyle{\al_1 + \cdots + \al_n=\al} \end{array}} 
\!\!\!\!\!\!\!\! V_{\al_1} \otimes \cdots \otimes V_{\al_n}
= \!\!\!\!\!\!\!\!\!
\sum_{\begin{array}{c} \scriptstyle{(\al_1,\ldots,\al_n) \in \N^n}\\ \scriptstyle{\al_1 + \cdots + \al_n \leqslant \al} \end{array}} 
\!\!\!\!\!\!\!\! V_{\al_1} \otimes \cdots \otimes V_{\al_n} \subseteq V^{\otimes n}\quad \textrm{and}\quad
S_\al^n= \sigma(T_\al^n)$$
\noindent where $\sigma\colon T(V) \rightarrow S(V)$ is the canonical map.

Set $M_0= V_0$ and for $i \geqslant 1$ consider a supplementary subspace $M_i$ of $V_{i-1}$ inside $V_i$.
We thus get that $V_i= \bigoplus_{j \leqslant i} M_j$ and $V= \bigoplus_{i \in \N} M_i$ can be seen as a graded vector space which we denote by $\gr(V)$. 
Using this graduation, we get a bigraduation on $T(V)= \bigoplus_{(n,\al) \in \N^2} T^{n,\al}$ and 
$S(V)=\bigoplus_{(n,\al) \in \N^2} S^{n,\al}$ given by 
$$T^{n,\al}= \!\!\!\!\!\!\!\!\!
\bigoplus_{\begin{array}{c} \scriptstyle{(\al_1,\ldots,\al_n) \in \N^n}\\[-1ex] \scriptstyle{\al_1 + \cdots + \al_n=\al} \end{array}} 
\!\!\!\!\!\!\!\! M_{\al_1} \otimes \cdots \otimes M_{\al_n}\quad \textrm{and}\quad
S^{n,\al}= \!\!\!\!\!\!\!\!\! \bigoplus_{\begin{array}{c} \scriptstyle{r \in \N,\ (n_1,\ldots,n_r) \in \N^r} \\ [-1ex]
\scriptstyle{n_1 + \cdots + n_r=n}\\[-1ex] \scriptstyle{n_1\al_1 + \cdots + n_r \al_r =\al}\end{array}} \!\!\!\!\!\!\!\!\! 
S^{n_1}(M_{\al_1}) \otimes \cdots \otimes S^{n_r}(M_{\al_r})$$
See~\cite[Alg\`ebre chap III, \S 5.5 Proposition 7, \S 6.6 Proposition 10]{bourbaki}. 
We set $T_\al= \bigoplus_{n \in \N} T_\al^n$ and $S_\al = \bigoplus_{n \in \N} S_\al^n$.

\begin{proposition}\label{prop-graduation-induite} The associated graded rings to the filtration $T(V)= \cup_{\al \in \N} T_\al$ and $S(V)= \cup_{\al \in \N} S_\al$
are isomorphic to $T(\gr(V))$ and $S(\gr(V))$, respectively. 
\end{proposition}

\begin{proof} For $(\alpha_1,\ldots,\alpha_n) \in \N^n$, let us define 
$\pi_{\alpha_1,\ldots,\alpha_n}^{\beta_1,\ldots,\beta_n}\colon V_{\alpha_1} \otimes \cdots \otimes V_{\alpha_n} 
\rightarrow M_{\beta_1} \otimes \cdots \otimes M_{\beta_n}$ the map associated to the decomposition of $V_{\alpha_1} \otimes \cdots \otimes V_{\alpha_n}$ as a direct sum
$$V_{\alpha_1} \otimes \cdots \otimes V_{\alpha_n}= \bigoplus_{\begin{array}{c} 
\scriptstyle{(\beta_1,\ldots,\beta_n) \in \N^n}\\[-1ex] \scriptstyle{\forall\,i, \ 1 \leqslant \beta_i \leqslant \alpha_i} \end{array}} 
M_{\be_1} \otimes \cdots \otimes M_{\be_n}$$
\noindent Note that if there exists $i$ such that $\be_i > \al_i$ then $\pi_{\alpha_1,\ldots,\alpha_n}^{\beta_1,\ldots,\beta_n}=0$.
When $(\alpha_1,\ldots,\alpha_n)=(\beta_1\ldots,\beta_n)$, we simply denote $\pi_{\alpha_1,\ldots,\alpha_n}$
instead of $\pi_{\alpha_1,\ldots,\alpha_n}^{\alpha_1,\ldots,\alpha_n}$

Let us now define 
$$\pi =\oplus \pi_{\alpha_1,\ldots,\alpha_n}\colon 
\!\!\!\!\!\!\!\!\!
\bigoplus_{\begin{array}{c} \scriptstyle{(\al_1,\ldots,\al_n) \in \N^n}\\[-1ex] \scriptstyle{\al_1 + \cdots + \al_n=\al} \end{array}} 
\!\!\!\!\!\!\!\! V_{\alpha_1} \otimes \cdots \otimes V_{\alpha_n} \longrightarrow 
\bigoplus_{\begin{array}{c} \scriptstyle{(\al_1,\ldots,\al_n) \in \N^n}\\[-1ex] \scriptstyle{\al_1 + \cdots + \al_n=\al} 
\end{array}} M_{\alpha_1} \otimes \cdots \otimes M_{\alpha_n}$$
\noindent We want to show that $\pi$ induces a map from $T_\al^n$ to $T^{n,\al}$. It suffices to show that 
if 
$$w=(w_{\al_1,\ldots,\al_n})_{(\al_1,\ldots,\al_n)} \in \bigoplus_{\begin{array}{c} \scriptstyle{(\al_1,\ldots,\al_n) \in \N^n}\\[-1ex] \scriptstyle{\al_1 + \cdots + \al_n=\al} \end{array}} \!\!\!\!\!\!\!\! V_{\alpha_1} \otimes \cdots \otimes V_{\alpha_n}$$
verifies $\sum_{(\al_1,\ldots,\al_n)} w_{\al_1,\ldots,\al_n}=0 \in T^n(V)$ then $\pi(w)=0$. But using the direct sum decomposition
of the $V_i$ we get that 
$$T_\al^n= \bigoplus_{\begin{array}{c} \scriptstyle{(\be_1,\ldots,\be_n) \in \N^n}\\[-1ex] \scriptstyle{\be_1 + \cdots + \be_n\leqslant \al} 
\end{array}} M_{\beta_1} \otimes \cdots \otimes M_{\beta_n}$$
\noindent Hence the hypothesis $\sum_{(\al_1,\ldots,\al_n)} w_{\al_1,\ldots,\al_n}=0 \in T^n(V)$ may be written 
for all $(\be_1,\ldots,\be_n)$ such that $\be_1 + \cdots +\be_n \leqslant \al$, we get 
$\sum_{(\al_1\ldots,\al_n)} \pi_{\alpha_1,\ldots,\alpha_n}^{\beta_1,\ldots,\beta_n}(w_{\al_1,\ldots,\al_n})=0$. 

In particular, when $\be_1+ \cdots + \be_n=\al$ and $(\al_1,\ldots,\al_n) \neq (\be_1,\ldots,\be_n)$ there exists 
$i$ such that $\be_i > \al_i$. Hence
$0=\sum_{(\al_1\ldots,\al_n)} \pi_{\alpha_1,\ldots,\alpha_n}^{\beta_1,\ldots,\beta_n}(w_{\al_1,\ldots,\al_n})=
\pi_{\alpha_1,\ldots,\alpha_n}(w_{\al_1,\ldots,\al_n})$ as wanted. 
We still denote by $\pi\colon T_\al^n \rightarrow T^{n,\al}$ the induced map. We clearly have $\ker \pi = T_{\al-1}^n$
showing the first result since the kernel of $\pi_{\al_1,\ldots,\al_n}$ is included into $T_{\al-1}^n$
(see~\cite[Alg\`ebre, chap.2, \S 3, Proposition~6]{bourbaki}).

Let us now consider the case of the symmetric algebra. We denote by $\sigma_{\al}^n\colon T_\al^n \rightarrow S_\al^n$
and by $\sigma^{n,\al}\colon T^{n,\al} \rightarrow S^{n,\al}$ the maps induced by $\sigma\colon T(V) \rightarrow S(V)$.   

The surjective map $\pi: T_\al^n \rightarrow T^{n,\al}$ sends $\ker \sigma_\al^n$ onto $\ker \sigma^{n,\al}$. 
Hence $\pi$ induces a surjective map also denoted by $\pi\colon S_\al^n \rightarrow S^{n,\al}$. To conclude, it remains to show that 
$\ker \pi=S_{\al-1}^n$. This is clear that $S_{\al-1}^n \subseteq \ker \pi$. The reverse inclusion follows from diagram chasing in
the following diagram (where the column are exact and the second row too)
$$\xymatrix{ &\ker \sigma_\al^n \ar@{->>}[r]^-{\pi} \ar@{^{(}->}[d] & \ker \sigma^{n,\al} \ar@{^{(}->}[d] \\ 
T_{\al-1}^n \ar@{^{(}->}[r]\ar@{->>}[d]^{\sigma_{\al-1}^n} & T_{\al}^n \ar@{->>}[r]^{\pi} \ar@{->>}[d]^{\sigma_{\al}^n} & 
T^{n, \al} \ar@{->>}[d]^{\sigma^{n,\al}}\\
S_{\al-1}^n \ar@{^{(}->}[r] & S_{\al}^n \ar@{->>}[r]^{\pi} & S^{n,\al} }$$
\end{proof}

\begin{proposition}\label{prop-graduation-delta-gamma} Assume that there exists linear maps $\Delta: V \rightarrow V$ and $\Gamma: V \rightarrow V$ such that $\Delta(V_i)\subseteq V_{i-1}$ and $\Gamma(V_i) \subseteq V_i$ for every $i \in \N$, and such that $[\Delta,\Gamma]=\Delta$.
The extension of $\Delta$ as a derivation of $T(V)$, resp. $S(V)$, verifies $\Delta(T_{\al}^n) \subseteq \Delta(T_{\al-1}^n)$, resp. $\Delta(S_{\al}^n) \subseteq \Delta(S_{\al-1}^n)$. 
The extension of $\Gamma$ as a derivation of $T(V)$, resp. $S(V)$, stabilize the preceding filtration on $T(V)$, resp. $S(V)$. 
Hence $\Delta$ and $\Gamma$ induce graded derivations $\overline{\Delta}$ of degree $-1$ and $\overline{\Gamma}$ of degree $0$ on $T(\gr(V))$ and $S(\gr(V))$.
Moreover we have $[\overline{\Delta},\overline{\Gamma}]=\overline{\Delta}$.

In addition, $\Delta$ induces a degree $-1$ map on $\gr(V)$ and $\Gamma$ induces a degree $0$ map on $\gr(V)$ and their extensions as derivations on $T(\gr(V))$ and $S(\gr(V))$ coincide with 
$\overline{\Delta}$ and $\overline{\Gamma}$.
\end{proposition}

\begin{proof} The extension of a linear map $f: V \rightarrow V$ as a derivation of $T(V)$ acts on 
$V_{\al_1} \otimes \cdots \otimes V_{\al_n}$ as $f\otimes \id \otimes \cdots  \otimes \id +  \cdots +  \id \otimes \cdots\otimes \id 
\otimes f$.
Hence if $f=\Delta$ then $T_\al^n$ is mapped into $T_{\al-1}^n$ and if $f=\Gamma$ then $T_{\al}^n$ is mapped
into $T_{\al}^n$ showing that they induce graded endomorphisms $\overline{\Delta}$ of degree $-1$ and $\overline{\Gamma}$ of degree $0$ of $T(\gr(V))$ and $S(\gr(V))$.
Moreover these endomorphisms are derivations of the associated graded rings and verify $[\overline{\Delta},\overline{\Gamma}]=\overline{\Delta}$.

For the last part of the proof, since the maps considered are derivations of $T(\gr(V))$ and $S(\gr(V))$ it suffices to show that they coincide on $\gr(V)$.
This follows readily from the commutative diagrams

$$\xymatrix{V_\alpha \ar^{\Delta}[r] \ar[d] & V_{\alpha-1} \ar[d] \\ V_\alpha/V_{\alpha-1} \cong M_\alpha \ar^-{\overline{\Delta}}[r] & V_{\alpha-1}/V_{\alpha-2}\cong M_{\alpha-1}} \textrm{\quad and \quad} 
\xymatrix{V_\alpha \ar^{\Gamma}[r] \ar[d] & V_{\alpha} \ar[d] \\ V_\alpha/V_{\alpha-1} \cong M_\alpha \ar^-{\overline{\Gamma}}[r] & V_{\alpha}/V_{\alpha-1}\cong M_{\alpha}} $$
\end{proof}

\begin{example}\label{example-epsilon-tilde} Consider $V$ a finite dimensional vector space and $\Delta$ and $\Gamma$ two linear endomorphisms of $V$ verifying $[\Delta,\Gamma]=\Delta$.
Since $\Delta$ is nilpotent and $\ker \Delta^i$ is stable by $\Gamma$ (see the proof of Theorem~\ref{dmax}), the family of subspaces $V_i = \ker \Delta^i$ is a filtration of $V$ satisfying the hypothesis of Proposition~\ref{prop-graduation-delta-gamma}. 
\end{example}

{\bf Acknowledgment.} The authors thank Patrick Le Meur for fruitful discussions on Artin-Schelter regular algebras.

%%%%%%%%%%%%%%%%  Bibilography %%%%%%%%%%%%%%%%%%%%%%%%%%

\bibliographystyle{amsalpha}

%\bibliography{biblio}

\providecommand{\bysame}{\leavevmode\hbox to3em{\hrulefill}\thinspace}
\providecommand{\MR}{\relax\ifhmode\unskip\space\fi MR }
\providecommand{\MRhref}[2]{%
  \href{http://www.ams.org/mathscinet-getitem?mr=#1}{#2}
}
\providecommand{\href}[2]{#2}

\end{document}